%% file: main.tex
\documentclass[11pt,a4paper]{article}
\usepackage{graphicx} 
\usepackage{amsmath, amssymb}
\usepackage{algorithm}
\usepackage{algpseudocode}
\algrenewcommand{\algorithmiccomment}[1]{\hskip3em$\triangleright$ #1}
\usepackage{booktabs}    
\usepackage{caption}     
\usepackage{array}       
\usepackage[dvipsnames]{xcolor}   
\usepackage{colortbl} 
\usepackage{subcaption} 
\usepackage{float}

\usepackage[shortlabels,inline]{enumitem}

\usepackage{multirow}

\numberwithin{equation}{section}
\input{preprint-formatting}



\title{Soft-Update MSA for Nonconvex Stochastic Optimal Control: Theory and Algorithms}

\date{}
\author{
Shaolin Ji\thanks{Zhongtai Securities Institute for Financial Studies,
Shandong University, Jinan, Shandong 250100, PR China
(jsl@sdu.edu.cn).}
\qquad
Xinmeng Liu\thanks{Corresponding author. Zhongtai Securities Institute for
Financial Studies, Shandong University, Jinan, Shandong 250100, PR China
(202212046@mail.sdu.edu.cn).}
}
\begin{document}

\maketitle

\begin{abstract}
For the method of successive approximations (MSA), a fundamental challenge is that the
stochastic maximum principle (SMP) provides an optimality condition rather than a descent
rule. 
To overcome this difficulty for nonconvex stochastic control problems, we develop a soft-update MSA (SU-MSA) with a second-order-adjoint correction of the Hamiltonian and a neural realization called DeepSMP.
Both an explicit stepsize rule guaranteeing cost descent and the convergence of the algorithm are established.
Moreover, the global SMP condition is asymptotically satisfied when the corrected Hamiltonian is convex, while in a special case the cost is further proved to converge to the global optimum.
To extend SU-MSA to high dimensions, we develop DeepSMP using a differentiable smooth approximation of the corrected Hamiltonian and analyze the convergence of the resulting objective and its gradient.
Numerical experiments demonstrate the effectiveness of the proposed framework.
\end{abstract}

\begin{quote}

\end{quote}

\section{Introduction}

The rapid development of deep learning has opened new possibilities for solving high-dimensional stochastic control problems; see, e.g., \cite{Kushner1990,LudwigSirignanoHuangPapanicolaou2012,EHanJentzen2017,CarmonaLauriere2022}. A common approach is to parameterize the control by a neural network and optimize the resulting cost by gradient-based methods. In parallel, the method of successive approximations (MSA) provides an iterative framework built directly upon the Pontryagin maximum principle. At each iteration, MSA solves the state and adjoint equations and updates the control through Hamiltonian minimization. Its central difficulty is that the maximum principle gives an optimality condition rather than a descent rule: a Hamiltonian-improving control also perturbs the state and adjoint processes, and the resulting error may offset the nominal improvement. Consequently, the basic MSA need not converge.

Existing modifications recover stability by restricting or regularizing the control perturbation. In deterministic control, representative approaches include convex-combination updates, short-interval replacement, and augmented-Hamiltonian schemes; see, e.g., \cite{KrylovChernousko1963,Mitter1966,Lyubushin1979,LiChenTaiE2018}.  For controlled diffusions, the perturbation analysis becomes more delicate when the control enters the diffusion coefficient. Kerimkulov, \v{S}i\v{s}ka, and Szpruch \cite{KerimkulovSiskaSzpruch2021} introduced quadratic penalties on variations in the drift, diffusion, and state derivative of the Hamiltonian, and established convergence of the resulting modified MSA. Ji and Xu \cite{JiXu2022} extended this framework to stochastic recursive control with possibly nonconvex control domains, using the second-order adjoint process together with BMO martingale estimates to relax structural assumptions required in the earlier analysis. More recently, they incorporated the second-order adjoint correction directly into a spike-update scheme, yielding an $O(\varepsilon)$ corrected-Hamiltonian contribution with an $O(\varepsilon^{3/2})$ remainder \cite{JiXu2025}.

A further issue concerns the information used to generate the control update. 
Under the differentiability setting of the local stochastic maximum principle, the control derivative of the Hamiltonian represents the first variation of the cost functional. In particular, the gradient flow studied by Sethi and \v{S}i\v{s}ka \cite{SethiSiska2024} is the $L^2$-gradient flow of $J$ in the control space, whereas direct neural methods such as \cite{HanE2016} perform gradient descent after parameterizing the control.
Although the optimization variables are different, both are driven by first-order information of the cost functional.

This limitation is particularly transparent in the proximal MSA of Sethi and \v{S}i\v{s}ka \cite{SethiSiska2024}, whose candidate update minimizes
\begin{equation}\label{eq:intro-proximal-update}
    H_t^u(a)+\frac{\rho}{2}|a-u_t|^2.
\end{equation}
The proximity term stabilizes the iteration, but its fixed isotropic weight $\rho I$ entangles candidate selection with update magnitude control, so conservative choices may over-regularize the candidate update.
If the current control pointwise minimizes the ordinary Hamiltonian, it
also minimizes \eqref{eq:intro-proximal-update}, causing the proximal
iteration to stagnate even though this control need not satisfy the global
stochastic maximum principle involving the second-order diffusion
correction \cite{Peng1990}.
More importantly, near an ordinary-Hamiltonian local minimizer, the positive proximal term further strengthens the local convexity seen by a gradient-based inner solver, making this local
minimum more attractive and causing gradient-based updates to stagnate more easily. Thus, the proximal term controls the update magnitude but cannot introduce the second-order information required to escape the ordinary-Hamiltonian landscape.

Therefore, the main goal of this paper is to use the endogenous second-order diffusion correction in the global stochastic maximum principle directly for candidate selection, while controlling the realized update independently through a soft update. For a reference control \(u\), we introduce the corrected Hamiltonian
\begin{equation}\label{eq:intro-corrected-H}
    \mathcal H_t^u(a)
    :=
    H_t^u(a)
    +
    \frac12
    \operatorname{Tr}\!\left[
        \Delta\sigma_t(a;u)^\top
        P_t^u
        \Delta\sigma_t(a;u)
    \right],
\end{equation}
where \(\Delta\sigma_t(a;u):=\sigma(t,X_t^u,a)-\sigma(t,X_t^u,u_t)\).
Here \(P^u\) is the second-order adjoint generated by the control problem itself. Unlike the fixed positive weight \(\rho I\), the induced quadratic correction varies with time, state, and control direction and may have either sign. It therefore provides a direction-dependent assessment of the diffusion risk generated by a candidate perturbation, which can be viewed as a risk-sensitive weighting mechanism; see Hu et al.~\cite{HuJiXuXue2026}.

The contribution of the second-order adjoint is not merely to modify the update size, but to change the candidate information available to the optimizer. 
The second-order correction preserves the first-order descent information encoded by the Hamiltonian gradient, while modifying the local quadratic structure of the candidate objective.
A sufficiently negative correction can turn an ordinary-Hamiltonian local minimizer into a local maximizer of the corrected Hamiltonian, producing a local convexity--concavity reversal. Finite perturbations around such a point can therefore reveal favorable corrected-Hamiltonian candidates that are invisible to first-order optimization of the ordinary Hamiltonian. Moreover, the second-order correction removes the strict descent restriction of proximal MSA: unlike the proximal criterion requiring \(\Delta H_t^u(a)\leq-\frac{\rho}{2}|a-u_t|^2<0\), the endogenous sign-indefinite term \(\frac12\Delta\sigma_t(a;u)^\top P_t^u\Delta\sigma_t(a;u)\) allows candidates with positive ordinary-Hamiltonian variation.

To convert this candidate improvement into an actual decrease of the cost, we perform the soft update
\begin{equation}\label{eq:intro-soft-update}
    u_t^\eta
    =
    u_t+\eta(v_t-u_t),
    \qquad
    0<\eta\leq1,
\end{equation}
where \(v\) is a corrected-Hamiltonian candidate. We prove
\begin{equation}\label{eq:intro-descent}
    J(u^\eta)-J(u)
    \leq
    \mathbb E\!\int_0^T
    \left[
        \mathcal H_t^u(u_t^\eta)
        -
        \mathcal H_t^u(u_t)
    \right]dt
    +
    C\eta^2.
\end{equation}
Thus, the corrected-Hamiltonian improvement enters at order \(O(\eta)\), whereas the remaining perturbation is of order \(O(\eta^2)\). The second-order correction determines the candidate, while the soft update controls the realized displacement. Under uniform semiconvexity of the corrected Hamiltonian, this balance yields an explicit stepsize rule that guarantees a fixed proportion of the corrected-Hamiltonian improvement is converted into actual cost decrease.

This distinction is also important in numerical implementation. Exact pointwise Hamiltonian minimization can be difficult in high-dimensional problems; Li et al.~\cite{LiChenTaiE2018} used L-BFGS for the inner minimization, which incurs substantial computational costs, while replacing the minimization by a local gradient step again restricts the update to first-order information at the current control. 
We therefore develop DeepSMP, which samples perturbed controls around the current control, evaluates their corrected-Hamiltonian values, and optimizes a smooth soft-min of these values in the network parameter space. The perturbations expose the candidate geometry created by the second-order correction, while gradient descent remains the computational engine. Hence, the essential difference from direct neural control is not whether gradient optimization is used, but what information is supplied to it.

Our main contributions are as follows. First, we propose a soft-update MSA (SU-MSA) for stochastic control problems with control-dependent diffusion. The second-order correction is endogenous, signed, and direction-dependent, while the soft update independently controls the realized displacement. In particular, the correction can alter the local geometry of the candidate objective and produce a convexity--concavity reversal at an ordinary-Hamiltonian local minimizer, revealing favorable finite perturbations where ordinary-Hamiltonian gradient updates stagnate and positive proximal regularization further strengthens the local minimum.

Second, we establish the cost upper bound \eqref{eq:intro-descent}, which separates the corrected-Hamiltonian improvement from the quadratic soft-update remainder. Combined with uniform semiconvexity of the corrected Hamiltonian, this estimate yields an explicit stepsize rule guaranteeing proportional cost descent. The resulting cost values converge, while both the accepted corrected-Hamiltonian improvement and the control increments vanish asymptotically. When the corrected Hamiltonian is convex, the corrected-SMP residual further converges to zero in \(L^1([0,T]\times\Omega)\). For a class of nonconvex stochastic control problems admitting an exact It\^o reduction, this residual coincides with the global optimality gap, so its convergence yields convergence of the cost to the global optimum.

Third, we develop DeepSMP as a neural realization of SU-MSA. DeepSMP combines perturbation-based exploration, corrected-Hamiltonian evaluation, a differentiable smooth-min, and parameter-space soft updates, allowing a gradient-based neural solver to exploit second-order diffusion information without solving the pointwise corrected-Hamiltonian minimization exactly. We further prove that both the sampled smooth-min objective and its actor-parameter gradient converge in probability to their corrected-Hamiltonian counterparts as the action perturbations localize.
Illustrative examples demonstrate ordinary-Hamiltonian stagnation, convexity--concavity reversal, and the time- and direction-dependent adjustment induced by the second-order adjoint, while numerical experiments compare DeepSMP with the direct deep-learning method of Han and E \cite{HanE2016} and modified-MSA baselines.

The remainder of the paper is organized as follows. Section~2 formulates the stochastic control problem and introduces the adjoint equations and the corrected Hamiltonian. Section~3 presents SU-MSA and establishes its descent and convergence properties. Section~4 develops DeepSMP and its theoretical analysis. Section~5 reports the numerical experiments.

\section{Preliminaries}
\label{sec:preliminaries}

\subsection{Basic settings}
\label{subsec:basic-settings}

Fix a finite time horizon \(T>0\). Let \((\Omega,\mathcal F,\mathbb P)\) be a complete probability space on which a \(d\)-dimensional standard
Brownian motion \(W=(W^1,\ldots,W^d)\) is defined. Let
\(\mathbb F=(\mathcal F_t)_{0\le t\le T}\)
be the usual augmentation of the natural filtration generated by \(W\).
All stochastic processes considered in this paper are defined on this
filtered probability space. Unless otherwise specified, adaptedness,
progressive measurability, and stopping times are understood with respect
to \(\mathbb F\).

For positive integers \(n\) and \(m\), let \(\mathbb R^n\) denote the
\(n\)-dimensional Euclidean space, \(\mathbb R^{n\times m}\) the space of
real \(n\times m\) matrices, and \(\mathbb S^n\) the space of real
symmetric \(n\times n\) matrices. For \(x,y\in\mathbb R^n\), we write
\(x\cdot y=x^\top y\). For \(A,B\in\mathbb R^{n\times m}\), set$A:B:=\operatorname{Tr}(AB^\top)$, $|A|:=\sqrt{\operatorname{Tr}(AA^\top)}$.
For higher-order derivative arrays, \(|\cdot|\) denotes the Euclidean
norm of the array of components.
If \(\sigma\in\mathbb R^{n\times d}\), we write it columnwise as $\sigma=(\sigma^1,\ldots,\sigma^d)$, $\sigma^i\in\mathbb R^n, i=1,\ldots,d$.
For \(q=(q^1,\ldots,q^d)\in\mathbb R^{n\times d}\), we use the
notation
\(
q:\sigma:=\sum_{i=1}^d(q^i)^\top\sigma^i
\).
For \(p\geq1\), let
\(L^p_{\mathcal F_T}(\Omega;\mathbb R^n)\) denote the space of all
\(\mathcal F_T\)-measurable random variables
\(\xi:\Omega\to\mathbb R^n\) such that
\(\mathbb E|\xi|^p<\infty\), endowed with the norm
\(\|\xi\|_{L^p}:=(\mathbb E|\xi|^p)^{1/p}\).
We denote by
\(\mathcal S^p_{\mathbb F}(0,T;\mathbb R^n)\) the space of continuous
\(\mathbb F\)-adapted processes \(X\) satisfying
\[
    \|X\|_{\mathcal S^p}
    :=
    \left(
        \mathbb E\left[\sup_{0\leq t\leq T}|X_t|^p\right]
    \right)^{1/p}
    <\infty.
\]
We denote by
\(\mathbb H_{\mathbb F}^{p}(0,T;\mathbb R^m)\), \(p\ge 1\),
the space of \(\mathbb F\)-progressively measurable processes
\(Z:[0,T]\times\Omega\to\mathbb R^m\) satisfying
\[
    \|Z\|_{\mathbb H^p}
    :=
    \left\{
        \mathbb E
        \left[
            \left(
                \int_0^T |Z_t|^2\,dt
            \right)^{p/2}
        \right]
    \right\}^{1/p}
    <\infty.
\]

Throughout the paper, \(C\) denotes a positive constant whose value may
change from line to line. Unless explicitly stated otherwise, \(C\) is
independent of the iteration index, the soft-update parameter \(\eta\), and
the particular admissible controls under consideration. 
When necessary, the
dependence of \(C\) on specific parameters will be indicated explicitly.

\subsection{Problem formulation}
\label{subsec:problem-formulation}

Let \(n\) be the state dimension and \(k\) the control dimension. The
control domain is denoted by $U\subset\mathbb R^k$.
In the theoretical part of this paper, \(U\) is assumed to be a nonempty
compact convex set. Compactness provides uniform boundedness of control
processes and control increments, while convexity guarantees the feasibility
of the convex interpolation used in the soft-update step. We do not impose any convexity assumptions on the
cost functional, the Hamiltonian, or the coefficients with respect to the control variable. 

The admissible control set is defined by
\[
    \mathcal U
    :=
    \left\{
        u:[0,T]\times\Omega\to U
        \ \middle|\
        u \text{ is } \mathbb F\text{-progressively measurable}
    \right\}.
\]
Since \(U\) is compact, there exists a constant \(D_U>0\) such that
\(|a-b|\leq D_U\) for all \(a,b\in U\). Consequently, for any
\(u,v\in\mathcal U\), upon setting \(d_t:=v_t-u_t\), we have
\(|d_t|\leq D_U\), \(dt\otimes d\mathbb P\)-a.e. In particular, this
uniform bound implies that
\(\mathbb E\int_0^T |d_t|^p\,dt<\infty\) for every \(p\geq1\).

For a given \(u\in\mathcal U\), the corresponding state process
\(X^u=(X_t^u)_{0\leq t\leq T}\) is governed by
\begin{equation}
\label{eq:state-equation}
\begin{cases}
    dX_t^u
    =
    b(t,X_t^u,u_t)\,dt
    +
    \sigma(t,X_t^u,u_t)\,dW_t,
    \qquad 0\leq t\leq T,\\[0.15cm]
    X_0^u=x_0\in\mathbb R^n.
\end{cases}
\end{equation}
Here,
\(b:[0,T]\times\mathbb R^n\times U\to\mathbb R^n\) and
\(\sigma:[0,T]\times\mathbb R^n\times U\to\mathbb R^{n\times d}\)
are deterministic measurable functions. Using the column notation
\(\sigma(t,x,a)
=(\sigma^1(t,x,a),\ldots,\sigma^d(t,x,a))\), where
\(\sigma^i(t,x,a)\in\mathbb R^n\) for \(i=1,\ldots,d\), the state equation
can equivalently be written as
\[
    dX_t^u
    =
    b(t,X_t^u,u_t)\,dt
    +
    \sum_{i=1}^d
    \sigma^i(t,X_t^u,u_t)\,dW_t^i.
\]
The present paper focuses on the controlled-diffusion setting, in which
\(\sigma(t,x,a)\) may depend nontrivially on the control variable \(a\).

The running and terminal cost functions are denoted by
\(f:[0,T]\times\mathbb R^n\times U\to\mathbb R\) and
\(\Phi:\mathbb R^n\to\mathbb R\), respectively. For each
\(u\in\mathcal U\), the associated cost functional is defined by
\begin{equation}
\label{eq:cost-functional}
    J(u)
    :=
    \mathbb E\left[
        \Phi(X_T^u)
        +
        \int_0^T f(t,X_t^u,u_t)\,dt
    \right].
\end{equation}
The stochastic optimal control problem is to determine
\(\inf_{u\in\mathcal U}J(u)\). Throughout the paper, we assume that
\(\inf_{u\in\mathcal U}J(u)>-\infty\). If the infimum is attained, we denote by \(u^\ast\in\mathcal U\) an optimal
control satisfying \(J(u^\ast)=\inf_{u\in\mathcal U}J(u)\).

We now introduce the notation for the so-called soft update. Let \(u\in\mathcal U\)
be a reference control and \(v\in\mathcal U\) a candidate control. For
\(\eta\in[0,1]\), define
\begin{equation}
\label{eq:soft-update}
    u_t^\eta
    :=
    u_t+\eta(v_t-u_t)
    =
    (1-\eta)u_t+\eta v_t.
\end{equation}
Let \(X^\eta:=X^{u^\eta}\) be the state process associated with
\(u^\eta\):
\begin{equation}
\label{eq:soft-state-equation}
\begin{cases}
    dX_t^\eta
    =
    b(t,X_t^\eta,u_t^\eta)\,dt
    +
    \sigma(t,X_t^\eta,u_t^\eta)\,dW_t,
    \qquad 0\le t\le T,\\[0.15cm]
    X_0^\eta=x_0.
\end{cases}
\end{equation}
We denote the state perturbation by
\[
    \delta X_t:=X_t^\eta-X_t^u.
\]

For fixed \(u\) and \(v\), the family
\(\{u^\eta\}_{\eta\in[0,1]}\) defines an admissible path in the control
space connecting the reference control and the candidate control. Since
\[
    u_t^\eta-u_t=\eta(v_t-u_t)=\eta d_t,
\]
the control perturbation is proportional to the scalar parameter \(\eta\).
Accordingly, if \(b\) and \(\sigma\) are Lipschitz continuous in the control
variable, then for fixed \(x\),
\[
    |b(t,x,u_t^\eta)-b(t,x,u_t)|
    +
    |\sigma(t,x,u_t^\eta)-\sigma(t,x,u_t)|
    \le
    C\eta |d_t|.
\]
By compactness of \(U\), the right-hand side is uniformly bounded by
\(C\eta D_U\). This observation is used later to separate the first-order
control improvement from higher-order stability remainders in the cost
difference estimate.

We impose the following basic regularity assumptions.

\begin{assumption}[Basic regularity]
\label{ass:basic-regularity}
The following conditions hold.

\begin{enumerate}[label=\textup{(\roman*)},leftmargin=2.5em]
\item
The functions \(b,\sigma,f\), and \(\Phi\) are twice continuously
differentiable with respect to \(x\). Moreover,
\(b,\sigma,f,b_x,\sigma_x,f_x,b_{xx},\sigma_{xx}\), and \(f_{xx}\)
are jointly continuous in \((t,x,a)\), while
\(\Phi,\Phi_x\), and \(\Phi_{xx}\) are continuous in \(x\).

\item
There exists a constant \(L>0\) such that, for every
\((t,x,a)\in[0,T]\times\mathbb R^n\times U\),
\[
\begin{aligned}
    |b(t,x,a)|+|\sigma(t,x,a)|
    &\leq L\bigl(1+|x|+|a|\bigr),\\
    |f(t,x,a)|
    &\leq L\bigl(1+|x|^2+|a|^2\bigr),\\
    |\Phi(x)|
    &\leq L\bigl(1+|x|^2\bigr),\\
    |f_x(t,x,a)|+|\Phi_x(x)|
    &\leq L\bigl(1+|x|\bigr),
\end{aligned}
\]
and
\[
\begin{aligned}
    &|b_x(t,x,a)|+|\sigma_x(t,x,a)|
    +|b_{xx}(t,x,a)|+|\sigma_{xx}(t,x,a)|\\
    &\qquad
    +|f_{xx}(t,x,a)|+|\Phi_{xx}(x)|
    \leq L.
\end{aligned}
\]

\item $b_{xx}$, $\sigma_{xx}$, $f_{xx}$, and $\Phi_{xx}$ are
$L$-Lipschitz continuous in $x$, uniformly in $(t,a)$, i.e.,
\[
\left|
\psi_{xx}(t,x_1,a)-\psi_{xx}(t,x_2,a)
\right|
\leq
L|x_1-x_2|,
\quad
\forall (t,a)\in[0,T]\times U,\quad
x_1,x_2\in\mathbb{R}^n,
\]
where $\psi=b,\sigma,f$,
and
\[
\left|
\Phi_{xx}(x_1)-\Phi_{xx}(x_2)
\right|
\le
L|x_1-x_2|.
\]

\item The coefficients and their relevant first-order state derivatives satisfy the following Lipschitz
conditions with respect to the control variable. More precisely, for every $(t,x)\in[0,T]\times\mathbb{R}^n$
and $a_1,a_2\in U$,

\[
|b(t,x,a_1)-b(t,x,a_2)|+|\sigma(t,x,a_1)-\sigma(t,x,a_2)|
\leq L|a_1-a_2|,
\]
\[
|f(t,x,a_1)-f(t,x,a_2)|
\leq L(1+|x|^2)|a_1-a_2|.
\]
Moreover,
\[
|b_x(t,x,a_1)-b_x(t,x,a_2)|
+
|\sigma_x(t,x,a_1)-\sigma_x(t,x,a_2)|
\leq L|a_1-a_2|,
\]
\[
|f_x(t,x,a_1)-f_x(t,x,a_2)|
\leq L(1+|x|)|a_1-a_2|.
\]

\end{enumerate}
\end{assumption}

\begin{remark}
Since $U$ is compact, $\sup_{a\in U}|a|<\infty$.
Thus, after possibly enlarging $L$, Assumption~\ref{ass:basic-regularity} implies that,
uniformly over $(t,a)\in[0,T]\times U$,
\[
|b(t,x,a)|+|\sigma(t,x,a)|+|f_x(t,x,a)|
\le L(1+|x|),
\]
and
\[
|f(t,x,a)|\le L(1+|x|^2).
\]
Moreover, the boundedness of $b_x$ and $\sigma_x$ yields the
global Lipschitz continuity of $b$ and $\sigma$ with respect to $x$:
\[
|b(t,x_1,a)-b(t,x_2,a)|
+
|\sigma(t,x_1,a)-\sigma(t,x_2,a)|
\le L|x_1-x_2|.
\]
\end{remark}

The following lemma establishes the well-posedness of the controlled state
equation and a uniform moment estimate for its solution.
The proof is given below.

\begin{lemma}[Well-posedness and uniform moment estimate for the state process]
\label{lem:state-moment-estimate}
Let Assumption~\ref{ass:basic-regularity} hold. Then, for any given
\(u\in\mathcal U\), the state equation~\eqref{eq:state-equation} admits a
unique strong solution \(X^u\). Moreover, for every \(m\geq2\), the family
\(\{X^u:u\in\mathcal U\}\) is uniformly bounded in
\(\mathcal S^m_{\mathbb F}(0,T;\mathbb R^n)\); more precisely, there exists
a constant \(C_m>0\), independent of \(u\), such that
\[
    \sup_{u\in\mathcal U}
    \mathbb E\left[
        \sup_{0\leq t\leq T}|X_t^u|^m
    \right]
    \leq
    C_m.
\]
Equivalently,
\[
    \sup_{u\in\mathcal U}
    \|X^u\|_{\mathcal S^m}
    <\infty,
    \qquad m\geq2.
\]
\end{lemma}
\begin{proof}
Since \(U\) is compact, Assumption~\ref{ass:basic-regularity} implies,
after enlarging \(L\) if necessary, that
\[
    |b(t,x,a)|+|\sigma(t,x,a)|
    \leq
    L(1+|x|),
    \qquad
    (t,x,a)\in[0,T]\times\mathbb R^n\times U.
\]
Moreover, the uniform boundedness of \(b_x\) and \(\sigma_x\) yields
\[
\begin{aligned}
    &|b(t,x,a)-b(t,\bar x,a)|
    +
    |\sigma(t,x,a)-\sigma(t,\bar x,a)| 
    \leq
    L|x-\bar x|,
\end{aligned}
\]
uniformly in \((t,a)\in[0,T]\times U\).

Fix \(u\in\mathcal U\). Since \(u\) is progressively measurable, the
coefficients
\[
    (t,\omega,x)\longmapsto b(t,x,u_t(\omega)),
    \qquad
    (t,\omega,x)\longmapsto \sigma(t,x,u_t(\omega))
\]
are progressively measurable and satisfy the preceding global Lipschitz
and linear-growth conditions. Hence, the standard existence and uniqueness
theorem for stochastic differential equations yields a unique strong
solution \(X^u\) of~\eqref{eq:state-equation}.

For completeness, after localization, H\"older's inequality for the
drift and the Burkholder--Davis--Gundy inequality for the diffusion give
\[
\begin{aligned}
 \mathbb E\sup_{s\leq t}|X_s^u|^m
 &\leq C_m|x_0|^m
   +C_m\mathbb E\left(\int_0^t|b(s,X_s^u,u_s)|\,ds\right)^m\\
 &\quad+C_m\mathbb E\left(\int_0^t|\sigma(s,X_s^u,u_s)|^2\,ds\right)^{m/2}\\
 &\leq C_m(1+|x_0|^m)
   +C_m\int_0^t\mathbb E\sup_{r\leq s}|X_r^u|^m\,ds.
\end{aligned}
\]
The stopped estimates are uniform in the localization level; Gronwall's
inequality and Fatou's lemma remove the localization. Equivalently,
retaining the values of the coefficients at the origin yields
\[
\begin{aligned}
    \mathbb E\left[
        \sup_{0\leq t\leq T}|X_t^u|^m
    \right]
    &\leq
    C_m
    \left(
        1+|x_0|^m
        +
        \mathbb E\int_0^T
        \bigl[
            |b(t,0,u_t)|^m
            +
            |\sigma(t,0,u_t)|^m
        \bigr]dt
    \right).
\end{aligned}
\]
By the compactness of \(U\) and the uniform linear-growth condition,
\[
    \sup_{a\in U}
    \bigl(
        |b(t,0,a)|+|\sigma(t,0,a)|
    \bigr)
    \leq L,
\]
uniformly in \(t\in[0,T]\). Consequently,
\[
    \sup_{u\in\mathcal U}
    \mathbb E\left[
        \sup_{0\leq t\leq T}|X_t^u|^m
    \right]
    \leq
    C_m(1+|x_0|^m),
\]
where \(C_m\) depends only on \(m,T,L\) and the dimensions, and is
independent of \(u\). Since \(x_0\) is fixed, its contribution may be
absorbed into \(C_m\), which proves the assertion.
\end{proof}

The preceding uniform moment estimate, together with the growth conditions on
\(f\) and \(\Phi\), immediately yields the following well-posedness and
boundedness result for the cost functional.

\begin{corollary}[Well-posedness and uniform lower boundedness of the cost]
\label{cor:cost-well-posed}
Let Assumption~\ref{ass:basic-regularity} hold. Then the cost functional
\(J\) is well defined on \(\mathcal U\). Moreover, there exists a constant
\(C_J>0\), independent of \(u\), such that
\[
    \sup_{u\in\mathcal U}|J(u)|
    \leq C_J.
\]
In particular,
\[
    \inf_{u\in\mathcal U}J(u)
    \geq -C_J>-\infty.
\]
\end{corollary}

\begin{proof}
By the compactness of \(U\), there exists \(M_U>0\) such that
\(|a|\leq M_U\) for all \(a\in U\). Hence, by the growth assumptions on
\(f\) and \(\Phi\),
\[
\begin{aligned}
    |J(u)|
    &\leq
    \mathbb E|\Phi(X_T^u)|
    +
    \mathbb E\int_0^T|f(t,X_t^u,u_t)|\,dt\\
    &\leq
    L\mathbb E\bigl[1+|X_T^u|^2\bigr]
    +
    L\mathbb E\int_0^T
    \bigl(1+|X_t^u|^2+|u_t|^2\bigr)\,dt\\
    &\leq
    L\bigl(1+T+TM_U^2\bigr)
    +
    L(1+T)
    \mathbb E\left[
        \sup_{0\leq t\leq T}|X_t^u|^2
    \right].
\end{aligned}
\]
Applying Lemma~\ref{lem:state-moment-estimate} with \(m=2\) yields a
constant \(C_J>0\), independent of \(u\), such that
\(|J(u)|\leq C_J\). The asserted well-posedness and uniform lower bound
follow immediately.
\end{proof}

Next, we introduce the first- and second-order adjoint
equations associated with a given admissible control \(u\in\mathcal U\).
For notational convenience, we use the following notation:
\[
    b^u(t):=b(t,X_t^u,u_t),
    \qquad
    \sigma^{u,i}(t):=\sigma^i(t,X_t^u,u_t),
    \qquad i=1,\ldots,d.
\]
For \(\psi=b,\sigma^i,f\), we denote
\[
    \psi_x^u(t):=\psi_x(t,X_t^u,u_t),
    \qquad
    \psi_{xx}^u(t):=\psi_{xx}(t,X_t^u,u_t).
\]
Under the minimization convention, the Hamiltonian is defined by
\begin{equation}
\label{eq:Hamiltonian-adjoint}
    H(t,x,p,q,a)
    =
    p^\top b(t,x,a)
    +
    \sum_{i=1}^d (q^i)^\top \sigma^i(t,x,a)
    +
    f(t,x,a),
\end{equation}
where $p\in\mathbb R^n$, $q=(q^1,\ldots,q^d)\in\mathbb R^{n\times d}$.
For \(u\in\mathcal U\), the first-order adjoint process
\((p^u,q^u)\) is defined as the solution of the BSDE
\begin{equation}
\label{eq:first-adjoint}
\begin{cases}
    dp_t^u
    =
    -H_x(t,X_t^u,p_t^u,q_t^u,u_t)\,dt
    +
    \displaystyle\sum_{i=1}^d q_t^{u,i}\,dW_t^i,
    \qquad 0\le t\le T,\\[0.25cm]
    p_T^u=\Phi_x(X_T^u).
\end{cases}
\end{equation}

We next define the second-order adjoint process. Let $H_{xx}^u(t):= H_{xx}(t,X_t^u,p_t^u,q_t^u,u_t)$, 
More explicitly,
\[
    H_{xx}(t,x,p,q,a)
    =
    \sum_{\ell=1}^n p_\ell\, b_{xx}^{\ell}(t,x,a)
    +
    \sum_{i=1}^d\sum_{\ell=1}^n
    q_\ell^i\,\sigma_{xx}^{\ell i}(t,x,a)
    +
    f_{xx}(t,x,a),
\]

The second-order adjoint process \((P^u,Q^u)\) is defined as the solution of
the matrix-valued BSDE
\begin{equation}
\label{eq:second-adjoint}
\begin{cases}
    dP_t^u
    =
    -\Bigg[
        (b_x^u(t))^\top P_t^u
        +
        P_t^u b_x^u(t)
        +
        \displaystyle\sum_{i=1}^d
        (\sigma_x^{u,i}(t))^\top P_t^u\sigma_x^{u,i}(t)
        \\[0.15cm]\qquad\qquad
        +
        \displaystyle\sum_{i=1}^d
        \left\{
            (\sigma_x^{u,i}(t))^\top Q_t^{u,i}
            +
            (Q_t^{u,i})^\top \sigma_x^{u,i}(t)
        \right\}
        +
        H_{xx}^u(t)
    \Bigg]dt
    +
    \displaystyle\sum_{i=1}^d Q_t^{u,i}\,dW_t^i,\\[0.35cm]
    P_T^u=\Phi_{xx}(X_T^u).
\end{cases}
\end{equation}
Here
\[
    P_t^u\in\mathbb S^n,
    \qquad
    Q_t^u=(Q_t^{u,1},\ldots,Q_t^{u,d}),
    \qquad
    Q_t^{u,i}\in \mathbb S^n.
\]

\begin{lemma}[Uniform moment estimate for the first-order adjoint process]
\label{lem:first-adjoint-moment}
Let Assumption~\ref{ass:basic-regularity} hold. For every
\(u\in\mathcal U\), equation~\eqref{eq:first-adjoint} admits a unique solution
\[
    (p^u,q^u)\in
    \mathcal S^4_{\mathbb F}(0,T;\mathbb R^n)
    \times\mathbb H^4_{\mathbb F}(0,T;\mathbb R^{n\times d}).
\]
There exists a constant \(C>0\), independent of \(u\), such that
\[
    \sup_{u\in\mathcal U}
    \left(\|p^u\|_{\mathcal S^4}+\|q^u\|_{\mathbb H^4}\right)
    \leq C.
\]
\end{lemma}
\begin{proof}
For fixed \(u\), denote
\[
    F^u(t,y,z):=(b_x^u(t))^\top y
       +\sum_{i=1}^d(\sigma_x^{u,i}(t))^\top z^i+f_x^u(t).
\]
The coefficients of \(y,z\) are uniformly bounded, so \(F^u\) is
progressively measurable and uniformly Lipschitz in \((y,z)\),
independently of \(u\). The growth assumptions imply
\[
\begin{aligned}
 &\mathbb E\left[|\Phi_x(X_T^u)|^4
       +\left(\int_0^T|F^u(t,0,0)|^2\,dt\right)^2\right]\\
 &\qquad\leq C\left(1+\mathbb E\sup_{t\leq T}|X_t^u|^4\right)\leq C,
\end{aligned}
\]
where Lemma~\ref{lem:state-moment-estimate} gives the last bound.
The Lipschitz BSDE theorem and its a priori estimate
\cite[Theorem~5.1 and Proposition~5.1]{ElKarouiPengQuenez1997}
therefore give a unique solution and
\[
\begin{aligned}
 \|p^u\|_{\mathcal S^4}^4+\|q^u\|_{\mathbb H^4}^4
 &\leq C\mathbb E\left[|\Phi_x(X_T^u)|^4
       +\left(\int_0^T|f_x^u(t)|^2\,dt\right)^2\right]
 \leq C.
\end{aligned}
\]
The constants depend only on \(T\), the coefficient bounds, the dimensions,
the control bound, and \(x_0\). The Brownian filtration has the martingale
representation property, so no additional orthogonal martingale occurs.
\end{proof}

\begin{lemma}[Uniform moment estimate for the second-order adjoint process]
\label{lem:second-adjoint-moment}
Let Assumption~\ref{ass:basic-regularity} hold. For every
\(u\in\mathcal U\), equation~\eqref{eq:second-adjoint} admits a unique solution
\[
    (P^u,Q^u)\in
    \mathcal S^4_{\mathbb F}(0,T;\mathbb S^n)
    \times\mathbb H^4_{\mathbb F}(0,T;(\mathbb S^n)^d).
\]
There exists a constant \(C>0\), independent of \(u\), such that
\[
    \sup_{u\in\mathcal U}
    \left(\|P^u\|_{\mathcal S^4}+\|Q^u\|_{\mathbb H^4}\right)
    \leq C.
\]
\end{lemma}
\begin{proof}
On the finite-dimensional space \(\mathbb S^n\), the generator in
\eqref{eq:second-adjoint} is affine in \((P,Q)\), with uniformly bounded
linear coefficients. Its inhomogeneous term satisfies
\[
    |H_{xx}^u(t)|\leq C(1+|p_t^u|+|q_t^u|),
    \qquad |\Phi_{xx}(X_T^u)|\leq C.
\]
In particular,
\[
    \left(\int_0^T|p_t^u|^2\,dt\right)^2
    \leq T^2\sup_{t\leq T}|p_t^u|^4,
    \qquad
    \mathbb E\left(\int_0^T|q_t^u|^2\,dt\right)^2
    =\|q^u\|_{\mathbb H^4}^4.
\]
Consequently, Lemma~\ref{lem:first-adjoint-moment} gives
\[
\begin{aligned}
 \mathbb E\left(\int_0^T|H_{xx}^u(t)|^2\,dt\right)^2
 &\leq C\left(1+T^2\|p^u\|_{\mathcal S^4}^4
                     +\|q^u\|_{\mathbb H^4}^4\right)\leq C.
\end{aligned}
\]
The Lipschitz BSDE estimate
\cite[Theorem~5.1 and Proposition~5.1]{ElKarouiPengQuenez1997} now yields
\[
\begin{aligned}
 \|P^u\|_{\mathcal S^4}^4+\|Q^u\|_{\mathbb H^4}^4
 &\leq C\mathbb E\left[|\Phi_{xx}(X_T^u)|^4
            +\left(\int_0^T|H_{xx}^u(t)|^2\,dt\right)^2\right]
 \leq C.
\end{aligned}
\]
This application on \(\mathbb S^n\) is legitimate: the terminal Hessian
and \(H_{xx}^u\) are symmetric, and each term of the generator preserves
symmetry when \(P,Q^1,\ldots,Q^d\) are symmetric. Uniqueness in the
ambient matrix space shows that this symmetric solution is the unique
matrix-valued solution. All coefficient bounds are uniform in \(u\).
\end{proof}

\subsection{The corrected \texorpdfstring{\(\mathcal H\)}{H}-function and soft-update stationarity residual}
\label{subsec:H-function-soft-residual}
We introduce the second-order adjoint corrected Hamiltonian used in this
paper and denote it by the \(\mathcal H\)-function. This function contains
two control arguments: a candidate control \(v\) and a reference control
\(u\). The reference control determines the state, the adjoint processes,
and the anchor point for the diffusion variation, whereas the candidate
control is the variable compared or optimized in the Hamiltonian update
step.

For
\[
    (t,x,p,q,P,v,u)
    \in
    [0,T]\times\mathbb R^n\times\mathbb R^n
    \times\mathbb R^{n\times d}\times\mathbb S^n
    \times U\times U,
\]
define the corrected \(\mathcal H\)-function by
\begin{equation}
\label{eq:corrected-H-function}
\begin{aligned}
    \mathcal H(t,x,p,q,P,v,u)
    &:=
    H(t,x,p,q,v)                                                   \\
    &\quad
    +
    \frac12
    \sum_{i=1}^d
    \left[
        \sigma^i(t,x,v)-\sigma^i(t,x,u)
    \right]^\top
    P
    \left[
        \sigma^i(t,x,v)-\sigma^i(t,x,u)
    \right].
\end{aligned}
\end{equation}
In particular,
\[
    \mathcal H(t,x,p,q,P,u,u)
    =
    H(t,x,p,q,u).
\]
Then the stochastic maximum principle
\cite{Peng1990,yong1999stochastic}
can be equivalently rewritten in terms of \(\mathcal H\)-function.

\begin{theorem}[\(\mathcal H\)-form of the stochastic maximum principle]
\label{thm:H-function-SMP}
Let Assumption~\ref{ass:basic-regularity} hold. Let $(X^\ast,u^\ast)$ 
be an optimal state-control pair, and let $(p^\ast,q^\ast)$, $P^\ast$
be the corresponding first- and second-order adjoint processes. Then, for
\(dt\otimes d\mathbb P\)-a.e. \((t,\omega)\),
\begin{equation}
\label{eq:H-function-SMP}
    \mathcal H
    \left(
        t,
        X_t^\ast,
        p_t^\ast,
        q_t^\ast,
        P_t^\ast,
        v,
        u_t^\ast
    \right)
    \ge
    \mathcal H
    \left(
        t,
        X_t^\ast,
        p_t^\ast,
        q_t^\ast,
        P_t^\ast,
        u_t^\ast,
        u_t^\ast
    \right),
    \qquad
    \forall v\in U.
\end{equation}
Equivalently,
\[
    u_t^\ast
    \in
    \arg\min_{v\in U}
    \mathcal H
    \left(
        t,
        X_t^\ast,
        p_t^\ast,
        q_t^\ast,
        P_t^\ast,
        v,
        u_t^\ast
    \right),
    \qquad
    dt\otimes d\mathbb P\text{-a.e.}
\]
\end{theorem}
\begin{proof}
By the global stochastic maximum principle under the minimization
convention \cite{Peng1990,yong1999stochastic}, the optimal pair
\((X^\ast,u^\ast)\) satisfies, for
\(dt\otimes d\mathbb P\)-a.e. \((t,\omega)\) and every \(v\in U\),
\[
\begin{aligned}
    &H(t,X_t^\ast,p_t^\ast,q_t^\ast,v)
    -
    H(t,X_t^\ast,p_t^\ast,q_t^\ast,u_t^\ast)\\
    &\quad+
    \frac12\sum_{i=1}^d
    \Big[
        \sigma^i(t,X_t^\ast,v)
        -
        \sigma^i(t,X_t^\ast,u_t^\ast)
    \Big]^\top
    P_t^\ast
    \Big[
        \sigma^i(t,X_t^\ast,v)
        -
        \sigma^i(t,X_t^\ast,u_t^\ast)
    \Big]
    \geq0.
\end{aligned}
\]
By the definition of the corrected Hamiltonian
in~\eqref{eq:corrected-H-function}, the left-hand side equals
\[
\begin{aligned}
    &\mathcal H
    \left(
        t,X_t^\ast,p_t^\ast,q_t^\ast,P_t^\ast,
        v,u_t^\ast
    \right)
    -
    \mathcal H
    \left(
        t,X_t^\ast,p_t^\ast,q_t^\ast,P_t^\ast,
        u_t^\ast,u_t^\ast
    \right),
\end{aligned}
\]
where the second-order correction vanishes at
\(v=u_t^\ast\). Hence~\eqref{eq:H-function-SMP} follows. Since the
inequality holds for every \(v\in U\), it is equivalent to
\[
    u_t^\ast
    \in
    \arg\min_{v\in U}
    \mathcal H
    \left(
        t,X_t^\ast,p_t^\ast,q_t^\ast,P_t^\ast,
        v,u_t^\ast
    \right),
\]
for \(dt\otimes d\mathbb P\)-a.e. \((t,\omega)\).
\end{proof}
This formulation shows that, when the diffusion coefficient depends on the
control, an optimal control satisfies the pointwise minimization condition
for the second-order corrected Hamiltonian. The \(\mathcal H\)-function will
serve as the basis for the control update and the stationarity residuals
introduced below.

Fix a reference control \(u\in\mathcal U\) and define the corrected
Hamiltonian along \(u\) by
\begin{equation}
\label{eq:Hu-density}
    \mathcal H_t^u(a)
    :=
    \mathcal H
    \left(
        t,
        X_t^u,
        p_t^u,
        q_t^u,
        P_t^u,
        a,
        u_t
    \right),
    \qquad a\in U.
\end{equation}
For \(u,v\in\mathcal U\) and \(\eta\in[0,1]\), define the soft-update map by
\begin{equation}
\label{eq:soft-update-map}
    S_\eta(u,v):=u+\eta(v-u).
\end{equation}

\section{Soft-Update MSA}
\label{sec:su-msa}
In this section, we introduce a soft-update MSA for nonconvex stochastic
control problems with control-dependent diffusion, hereafter referred to as the
soft-update MSA (SU-MSA). The method incorporates the second-order
adjoint correction into the Hamiltonian update and controls the resulting control
perturbation through a soft-update step. 

\subsection{The SU-MSA algorithm}
\label{subsec:su-algorithm}

At iteration \(n\), solve the state equation and the two adjoint
equations associated with \(u^n\), and choose a progressively measurable
candidate satisfying
\[
    v_t^n\in\arg\min_{a\in U}\mathcal H_t^{u^n}(a),
    \qquad dt\otimes d\mathbb P\text{-a.e.}
\]
Such a candidate exists: \(\mathcal H_t^{u^n}(a)\) is progressively
measurable in \((t,\omega)\), continuous in \(a\), and \(U\) is compact.
A measurable minimizer map for the continuous coefficient functions on
this finite-dimensional parameter space, composed with
\((t,X_t^{u^n},p_t^{u^n},q_t^{u^n},P_t^{u^n},u_t^n)\), gives a progressively
measurable selector. The moment estimates above and compactness of \(U\)
also ensure that all the Hamiltonian integrals below are finite.
The candidate is determined by the corrected Hamiltonian
\eqref{eq:Hu-density}. Its displacement is controlled independently
through the soft update \eqref{eq:soft-update-map}.

\begin{assumption}[Uniform semiconvexity of the corrected Hamiltonian]
\label{ass:corrected-semiconvexity}
There exists a constant \(A\geq0\), independent of the reference control,
such that, for every \(u\in\mathcal U\), the function
\[
    a\longmapsto\mathcal H_t^u(a)+A|a|^2
\]
is convex on \(U\), for \(dt\otimes d\mathbb P\)-almost every
\((t,\omega)\).
\end{assumption}

This is an additional condition on the corrected-Hamiltonian family.
When the action Hessian exists, a sufficient condition is
\(\nabla_{aa}^2\mathcal H_t^u(a)\succeq-2AI\).
The value \(A=0\) corresponds to convexity of the corrected Hamiltonian;
convexity of the original cost functional is not required.

Fix a valid uniform remainder constant \(C_{\mathrm{ub}}>0\) in
Theorem~\ref{thm:cost-upper-bound-soft-update}, and retain the acceptance
parameter \(\vartheta\in(0,1)\). Denote
\begin{equation}
\label{eq:eta-candidate-statistics}
\begin{aligned}
    r_n&:=\mathbb E\int_0^T
        [\mathcal H_t^{u^n}(u_t^n)-\mathcal H_t^{u^n}(v_t^n)]\,dt,\\
    s_n&:=\mathbb E\int_0^T|v_t^n-u_t^n|^2\,dt,
    &\Delta_4(u,v)&:=\|v-u\|_{\mathbb H^4}^2,\\
    \chi_n&:=\frac{\Delta_4(u^n,v^n)}{s_n},
    &\zeta_n&:=\frac{r_n}{s_n}\qquad(s_n>0).
\end{aligned}
\end{equation}
Thus \(r_n\geq0\) is the corrected-SMP residual, \(s_n\) is the
squared \(L^2(\Omega\times[0,T])\)-distance to the candidate, and
\(\zeta_n\) is the corrected-Hamiltonian gain per unit squared
displacement. The statistic \(\chi_n\geq1\) compares the
\(\mathbb H^4\) and \(L^2\) displacement scales.
The ratios are evaluated only after checking that \(r_n>0\), which
implies \(s_n>0\).

\begin{remark}[Interpretation of the displacement statistic]
\label{rem:eta-path-concentration}
For \(s_n>0\), denote the time-integrated displacement energy by
\[
    Z_n:=\int_0^T|v_t^n-u_t^n|^2\,dt.
\]
Then \(s_n=\mathbb{E}[Z_n]\) and
\[
\begin{aligned}
    \chi_n^2
    &=
    \frac{\mathbb{E}[Z_n^2]}{s_n^2}\\
    &=
    1+\frac{\operatorname{Var}(Z_n)}{s_n^2}
    =
    1+\mathbb{E}\left[
        \left(\frac{Z_n}{s_n}-1\right)^2
    \right].
\end{aligned}
\]
Thus \(s_n\) measures the mean displacement energy, whereas
\(\chi_n^2-1\) measures its variance after normalization to unit
mean. Equivalently, \(\sqrt{\chi_n^2-1}\) is the root-mean-square
deviation from the mean, expressed relative to that mean.
In particular,
\[
    \chi_n\geq1,
    \qquad
    \chi_n=1
    \quad\Longleftrightarrow\quad
    Z_n=s_n\quad\text{almost surely}.
\]
A larger \(\chi_n\) therefore indicates greater relative
variation of the displacement energy across sample paths.
Moreover,
\[
    C_{\mathrm{ub}}\eta^2\Delta_4(u^n,v^n)
    =
    C_{\mathrm{ub}}\eta^2\chi_ns_n,
\]
so \(\chi_n\) quantifies how these relative fluctuations
increase the remainder bound at a fixed mean displacement
energy. This statistic concerns variation across sample paths;
it does not describe the temporal distribution of displacement
within an individual path.
\end{remark}

The stepsize is computed directly as
\begin{equation}
\label{eq:explicit-eta-rule}
\eta_n=
\begin{cases}
\displaystyle
\min\!\left\{1,
\frac{\vartheta(\zeta_n-A)}
     {C_{\mathrm{ub}}\chi_n-\vartheta A}\right\},
& C_{\mathrm{ub}}\chi_n>\vartheta A,\quad\zeta_n>A,\\[1.2em]
1,
& C_{\mathrm{ub}}\chi_n\leq\vartheta A,\quad
  \displaystyle\zeta_n\geq\frac{C_{\mathrm{ub}}\chi_n}{\vartheta},\\[0.7em]
0,&\text{otherwise}.
\end{cases}
\end{equation}
The rule is designed to retain a fixed proportion \(1-\vartheta\) of the
corrected-Hamiltonian improvement as actual cost decrease. Combining
the semiconvex interpolation bound with the cost upper bound gives
the sufficient condition \eqref{eq:eta-linear-certificate}.
If \(C_{\mathrm{ub}}\chi_n>\vartheta A\), this condition imposes an
upper bound on \(\eta\), giving the first branch. If
\(C_{\mathrm{ub}}\chi_n\leq\vartheta A\), every nonempty set of
certified positive steps contains \(\eta=1\), giving the second
branch. Thus, whenever a positive step is certified, the rule chooses
the largest one in \((0,1]\), including regimes with a positive
lower stepsize bound. Otherwise, it returns zero. The proof is given
in Proposition~\ref{prop:explicit-eta-certificate}.

\begin{algorithm}[H]
\caption{SU-MSA with an explicit semiconvexity-based stepsize}
\label{alg:su-msa}
\begin{algorithmic}[1]
\Require \(u^0\in\mathcal U\), \(\vartheta\in(0,1)\),
\(\varepsilon_J\geq0\), and valid constants
\(A\geq0\), \(C_{\mathrm{ub}}>0\).
\For{\(n=0,1,2,\ldots\)}
\State Solve the state and adjoint equations under \(u^n\).
\State Choose a progressively measurable pointwise minimizer \(v^n\)
of \(\mathcal H^{u^n}\).
\State Compute \(r_n\) from \eqref{eq:eta-candidate-statistics}.
\If{\(r_n=0\)}
\State \Return \(u^n\) \Comment{The corrected-SMP minimum condition holds}
\EndIf
\State Compute \(s_n,\Delta_4(u^n,v^n),\chi_n,\zeta_n\), and
\(\eta_n\) from \eqref{eq:explicit-eta-rule}.
\If{\(\eta_n=0\)}
\State \Return \(u^n\) \Comment{No positive step is certified}
\EndIf
\State Form \(\widetilde u^{\,n+1}:=u^n+\eta_n(v^n-u^n)\).
\State Compute \(\mathscr A_n:=J(u^n)-J(\widetilde u^{\,n+1})\).
\If{\(0\leq\mathscr A_n\leq\varepsilon_J\)}
\State \Return \(u^n\) \Comment{Cost improvement is sufficiently small}
\EndIf
\State Accept \(u^{n+1}:=\widetilde u^{\,n+1}\).
\EndFor
\end{algorithmic}
\end{algorithm}

\subsection{Analytical comparisons with proximal MSA}
\label{subsec:analytical-comparisons-proximal-msa}

The following examples demonstrate three advantages of SU-MSA
over the scalar proximal update in \cite{SethiSiska2024}.
First, every control that pointwise minimizes its ordinary
Hamiltonian is a fixed point of the proximal iteration,
including nonoptimal controls. The first example shows that the
second-order-adjoint correction can reverse the local geometry of
such a stagnation point: a local minimum of the ordinary
Hamiltonian becomes a nonoptimal point for the corrected
Hamiltonian, thereby generating an improving direction.
Second, the proximal parameter \(\rho\) assigns
the same weight at every time, state, and sample, whereas the
second-order adjoint adapts the evaluation of diffusion
perturbations to the current trajectory and their future
cost consequences. The second example illustrates this
adaptation through an explicit time-dependent weight.
Third, the correction provides risk-sensitive weighting
across diffusion directions, as demonstrated by the
multidimensional example.
In each case, the corrected-Hamiltonian candidate equals
the optimal control, and the stepsize rule accepts the
full update under the common parameter choice in
Remark~\ref{rem:analytical-examples-acceptance}.

\subsubsection{Escaping an ordinary-Hamiltonian stagnation point}

The stagnation mechanism follows directly from the proximal
candidate rule. Suppose that a reference control \(u\) satisfies
\[
    u_t\in\arg\min_{a\in U}H_t^u(a),
    \qquad dt\otimes d\mathbb P\text{-a.e.}
\]
For every \(\rho>0\),
\[
    H_t^u(a)+\frac{\rho}{2}|a-u_t|^2
    \geq H_t^u(u_t),
\]
with equality only at \(a=u_t\).
Hence the proximal candidate equals \(u\), and every
subsequent soft update leaves the control unchanged.
This holds for every such reference control, including
those that are nonoptimal for \(J\).
The following example shows that the second-order-adjoint
correction identifies an improving candidate at such a control.

Let \(T=1\), \(X_0=0\), and \(U=[-2,2]\). Consider
\[
    dX_t^u=\sigma(u_t)\,dW_t,
    \qquad
    \sigma(a):=\frac{a}{a^2-a+1},
\]
with terminal and running costs
\[
    \Phi(x):=-\sqrt{1+x^2},
    \qquad
    f(x,a):=
    \sigma(a)^2
    \left[
        \frac{1}{2(1+x^2)^{3/2}}-\frac14
    \right].
\]
These coefficients satisfy
Assumption~\ref{ass:basic-regularity}. Since
\[
    \Phi_{xx}(x)=-\frac{1}{(1+x^2)^{3/2}},
\]
the running cost satisfies
\[
    f(x,a)
    =
    -\frac14\sigma(a)^2
    -\frac12\Phi_{xx}(x)\sigma(a)^2.
\]
Itô's formula therefore gives
\[
    J(u)
    =
    -1-\frac14
    \mathbb E\int_0^1\sigma(u_t)^2\,dt.
\]
Moreover,
\[
    1-\sigma(a)^2
    =
    \frac{(a-1)^2(a^2+1)}{(a^2-a+1)^2}
    \geq0,
\]
with equality if and only if \(a=1\).
Thus the unique optimal control is
\[
    u_t^\ast\equiv1,
    \qquad
    J(u^\ast)=-\frac54.
\]

Starting from \(u_t^0\equiv0\), we have
\[
    X_t^0=p_t^0=q_t^0=0,
    \qquad
    P_t^0=\Phi_{xx}(0)=-1.
\]
The ordinary Hamiltonian becomes
\[
    H_t^{u^0}(a)=\frac14\sigma(a)^2,
\]
and hence
\[
    \arg\min_{a\in U}H_t^{u^0}(a)=\{0\}.
\]
For every \(\rho>0\),
\[
    \arg\min_{a\in U}
    \left\{
        \frac14\sigma(a)^2+\frac{\rho}{2}a^2
    \right\}
    =
    \{0\}.
\]
The proximal MSA therefore remains at the nonoptimal control
\(u^0=0\) for every \(\rho>0\) and every subsequent scalar
soft-update stepsize.

The second-order-adjoint correction changes the candidate:
\[
\begin{aligned}
    \mathcal H_t^{u^0}(a)
    &=
    H_t^{u^0}(a)
    +\frac12P_t^0
    \bigl[\sigma(a)-\sigma(0)\bigr]^2\\
    &=-\frac14\sigma(a)^2.
\end{aligned}
\]
Consequently,
\[
    \arg\min_{a\in U}\mathcal H_t^{u^0}(a)=\{1\}.
\]
At this candidate, the ordinary Hamiltonian increases by
\(1/4\), while the diffusion correction contributes \(-1/2\):
\[
\begin{aligned}
    H_t^{u^0}(1)-H_t^{u^0}(0)&=\frac14,\\
    \frac12P_t^0
    \bigl[\sigma(1)-\sigma(0)\bigr]^2&=-\frac12.
\end{aligned}
\]
Their sum gives the exact cost improvement,
\[
    \mathcal H_t^{u^0}(1)-\mathcal H_t^{u^0}(0)
    =
    -\frac14
    =
    J(1)-J(0).
\]
Thus the second-order term identifies an improving control
that the ordinary Hamiltonian and every positive proximal
penalty reject.

The same example exhibits a local curvature reversal:
\[
    \frac{d^2}{da^2}H_t^{u^0}(0)=\frac12,
    \qquad
    \frac{d^2}{da^2}\mathcal H_t^{u^0}(0)=-\frac12.
\]
The reference action is a strict local minimizer of the
ordinary Hamiltonian and a strict local maximizer of the
corrected Hamiltonian. The proximal term instead increases
the curvature at the reference action to \(1/2+\rho\).

For an arbitrary reference control \(u\), the adjoint processes
satisfy
\[
    p_t^u=\Phi_x(X_t^u),
    \qquad
    q_t^u=\Phi_{xx}(X_t^u)\sigma(u_t),
    \qquad
    P_t^u=\Phi_{xx}(X_t^u).
\]
Substitution gives
\[
    \mathcal H_t^u(a)
    =
    -\frac14\sigma(a)^2
    +\frac12\Phi_{xx}(X_t^u)\sigma(u_t)^2.
\]
The second term is independent of \(a\), so the candidate is
\(v_t^u=1\) at every reference control.
With the parameter choice in
Remark~\ref{rem:analytical-examples-acceptance},
SU-MSA accepts the full step from \(u^0=0\) and reaches
the unique optimal control after one update.

\subsubsection{Adaptation to time-dependent diffusion sensitivity}

A constant proximal parameter imposes the same quadratic
weight on control displacements throughout the time horizon
and across all states and samples. The second-order adjoint
provides an adapted weight determined by the state trajectory
and the propagation of future costs.
The following example makes its temporal adaptation explicit.

Let \(T=1\), \(X_0=0\), \(U=[-2,2]\), and denote
\[
    k(t):=5-4t,
    \qquad 1\leq k(t)\leq5.
\]
Consider
\[
    dX_t^u=\tanh(u_t)\,dW_t,
\]
with
\[
    \Phi(x)=\log\cosh x
\]
and
\[
    f(t,x,a)
    =
    4\log\cosh x
    +\frac{k(t)}{2}\tanh^2x\,\tanh^2a
    -\frac12\tanh a.
\]
Denote
\[
    V(t,x):=k(t)\log\cosh x,
    \qquad
    \ell_t(a):=
    \frac{k(t)}{2}\tanh^2a-\frac12\tanh a.
\]
Then \(V(1,x)=\Phi(x)\) and
\[
    f(t,x,a)
    =
    \ell_t(a)-V_t(t,x)
    -\frac12V_{xx}(t,x)\tanh^2a.
\]
Itô's formula yields
\[
    J(u)=\mathbb E\int_0^1\ell_t(u_t)\,dt.
\]

For \(y=\tanh a\), completing the square gives
\[
    \ell_t(a)
    =
    \frac{k(t)}{2}
    \left(y-\frac{1}{2k(t)}\right)^2
    -\frac{1}{8k(t)}.
\]
Since
\[
    \frac{1}{2k(t)}
    \in[0.1,0.5]
    \subset(-\tanh2,\tanh2),
\]
the unique optimal control is
\[
    u_t^\ast
    =
    \operatorname{arctanh}
    \left(\frac{1}{2k(t)}\right),
\]
with
\[
    u_0^\ast\approx0.1003,
    \qquad
    u_1^\ast\approx0.5493,
    \qquad
    J(u^\ast)=-\frac{\log5}{32}.
\]

The adjoint processes are
\[
    p_t^u=V_x(t,X_t^u),
    \qquad
    q_t^u=V_{xx}(t,X_t^u)\tanh(u_t),
    \qquad
    P_t^u=k(t)\operatorname{sech}^2(X_t^u).
\]
Thus \(P_t^u\) depends on both time and the realized state.
At \(u_t^0\equiv0\),
\[
    X_t^0=p_t^0=q_t^0=0,
    \qquad
    P_t^0=k(t).
\]
The ordinary Hamiltonian is
\[
    H_t^{u^0}(a)=-\frac12\tanh a,
\]
whose action dependence is identical at every time.
The corrected Hamiltonian incorporates the time-dependent
weight \(P_t^0\):
\[
    \mathcal H_t^{u^0}(a)
    =
    -\frac12\tanh a+\frac{k(t)}{2}\tanh^2a
    =
    \ell_t(a).
\]
Hence
\[
    \arg\min_{a\in U}\mathcal H_t^{u^0}(a)
    =
    \{u_t^\ast\}.
\]
As the weight decreases from \(5\) to \(1\), the correction
allows progressively larger diffusion changes, recovering
the optimal control increasing from approximately
\(0.1003\) to \(0.5493\).

For a fixed \(\rho>0\), the proximal candidate instead solves
\[
    a_\rho
    =
    \arg\min_{a\in[-2,2]}
    \left\{
        -\frac12\tanh a+\frac{\rho}{2}a^2
    \right\}.
\]
The minimizer is unique and independent of time: negative
actions are inferior to their positive counterparts, and
the objective is strictly convex on \([0,2]\).
Every scalar soft update therefore produces a constant
control
\[
    u_t^{1,\mathrm{prox}}=\eta a_\rho,
    \qquad \eta\in(0,1].
\]

The uniform proximal scale can be quantified using the exact
cost-variation identity
\[
\begin{aligned}
    J(v)-J(u)
    =
    \mathbb E\int_0^1
    \left[
        H_t^u(v_t)-H_t^u(u_t)
        +\frac12P_t^u
        \bigl(\tanh v_t-\tanh u_t\bigr)^2
    \right]dt.
\end{aligned}
\]
Since \(0<P_t^u\leq5\) and \(\tanh\) is \(1\)-Lipschitz,
\[
    J(v)-J(u)
    \leq
    \mathbb E\int_0^1
    \left[
        H_t^u(v_t)-H_t^u(u_t)
        +\frac52|v_t-u_t|^2
    \right]dt.
\]
This estimate yields the uniform sufficient choice
\(\rho\geq5\) for proximal descent, with strict descent
for a nontrivial proximal update when \(\rho>5\).

For \(\rho\geq5\), the candidate is interior and satisfies
\[
    \rho a_\rho
    =
    \frac12\operatorname{sech}^2a_\rho,
    \qquad
    0<a_\rho<\frac{1}{2\rho}\leq0.1.
\]
The candidate decreases as \(\rho\) increases. At \(\rho=5\),
\[
    a_5\approx0.0990258.
\]
Thus every \(\rho\geq5\) and \(\eta\in(0,1]\) satisfy
\[
    0<u_t^{1,\mathrm{prox}}
    =\eta a_\rho\leq a_5.
\]

To compare the update amplitudes, denote
\[
    R_{\rho,\eta}(t)
    :=
    \frac{u_t^{1,\mathrm{prox}}-u_t^0}
         {u_t^\ast-u_t^0}
    =
    \frac{\eta a_\rho}{u_t^\ast}.
\]
Then
\[
    R_{\rho,\eta}(t)\leq R_{5,1}(t),
    \qquad \rho\geq5,\quad \eta\in(0,1],
\]
whereas the accepted SU-MSA update has ratio \(1\)
at every time.
Table~\ref{tab:time-adaptive-update} gives the comparison.

\begin{table}[htbp]
\centering
\caption{First-update comparison from \(u^0=0\).
The proximal values at \(\rho=5\), \(\eta=1\) give the
largest update amplitudes over all \(\rho\geq5\) and
\(\eta\in(0,1]\). SU-MSA reaches \(u_t^\ast\) at every time.}
\label{tab:time-adaptive-update}
\begin{tabular}{ccccc}
\hline
\(t\) &
\(P_t^0\) &
\(u_t^\ast=u_t^{1,\mathrm{SU}}\) &
\(u_t^{1,\mathrm{prox}}\) &
\(R_{5,1}(t)\)\\
\hline
\(0\)   & \(5\) & \(0.100335\) & \(0.099026\) & \(98.69\%\)\\
\(1/4\) & \(4\) & \(0.125657\) & \(0.099026\) & \(78.81\%\)\\
\(1/2\) & \(3\) & \(0.168236\) & \(0.099026\) & \(58.86\%\)\\
\(3/4\) & \(2\) & \(0.255413\) & \(0.099026\) & \(38.77\%\)\\
\(1\)   & \(1\) & \(0.549306\) & \(0.099026\) & \(18.03\%\)\\
\hline
\end{tabular}
\end{table}

The fixed proximal weight is well matched to the initial
scale but becomes increasingly restrictive as \(P_t^0\)
decreases. Even its largest update in this parameter range
covers only \(18.03\%\) of the optimal displacement at the
terminal time. Increasing \(\rho\) or reducing \(\eta\)
further decreases every ratio in the table.
The second-order-adjoint correction adjusts the weight
throughout the horizon and recovers the complete optimal
profile in one update.

This difference also yields an explicit cost comparison.
For any constant control \(w\), denote \(y:=\tanh w\).
Since \(\int_0^1k(t)\,dt=3\),
\[
    J(w)
    =
    \frac32y^2-\frac12y
    =
    \frac32\left(y-\frac16\right)^2-\frac1{24}.
\]
Every first proximal soft update is constant. Therefore,
for every fixed \(\rho>0\) and \(\eta\in(0,1]\),
\[
    J(u^{1,\mathrm{prox}})
    \geq-\frac1{24}
    >
    -\frac{\log5}{32}
    =
    J(u^{1,\mathrm{SU}}).
\]
For the uniform sufficient range \(\rho\geq5\),
\(0<\tanh(\eta a_\rho)<0.1\), so
\[
    J(u^{1,\mathrm{prox}})>-0.035.
\]
At \(\rho=5\) and \(\eta=1\), the costs are
\[
    J(u^{1,\mathrm{prox}})
    \approx-0.0347381,
    \qquad
    J(u^{1,\mathrm{SU}})
    =
    -\frac{\log5}{32}
    \approx-0.0502949.
\]
Thus SU-MSA achieves the optimal cost in one update,
while the fixed proximal weight leaves a substantial
cost gap through its suppression of updates in the
low-sensitivity portion of the time horizon.

\subsubsection{Risk-sensitive weighting across diffusion directions}

The second-order adjoint evaluates diffusion perturbations
according to their second-order contribution to future
expected costs. This induces risk-sensitive weighting:
directions with a larger positive cost contribution receive
a stronger penalty, while directions with a smaller
contribution permit larger adjustments.
The following example shows how this directional sensitivity
determines the optimal allocation between two controls.

Let \(T=1\), \(X_0=0\), and \(U=[-2,2]^2\).
Consider a two-dimensional Brownian motion and
\[
    dX_t^u=\Sigma(u_t)\,dW_t,
    \qquad
    \Sigma(a):=
    \begin{pmatrix}
        \tanh a^1 & 0\\
        0 & \tanh a^2
    \end{pmatrix}.
\]
Denote
\[
    M:=
    \begin{pmatrix}
        100 & 0\\
        0 & 1
    \end{pmatrix},
    \qquad
    c:=
    \begin{pmatrix}
        5\\[1mm]
        \frac12
    \end{pmatrix}.
\]
The terminal and running costs are
\[
    \Phi(x)
    =
    100\log\cosh x^1+\log\cosh x^2
\]
and
\[
\begin{aligned}
    f(x,a)
    ={}&
    50\tanh^2x^1\,\tanh^2a^1
    +\frac12\tanh^2x^2\,\tanh^2a^2\\
    &-5\tanh a^1-\frac12\tanh a^2.
\end{aligned}
\]
Denote
\[
    y(a):=
    \begin{pmatrix}
        \tanh a^1\\
        \tanh a^2
    \end{pmatrix},
    \qquad
    \ell(a):=\frac12y(a)^\top My(a)-c^\top y(a).
\]
Itô's formula gives
\[
    J(u)=\mathbb E\int_0^1\ell(u_t)\,dt.
\]
The quadratic objective in the diffusion variable is minimized
at
\[
    y^\ast=M^{-1}c
    =
    \begin{pmatrix}
        0.05\\
        0.5
    \end{pmatrix}
    \in(-\tanh2,\tanh2)^2.
\]
Hence the unique optimal control is
\[
    u^\ast
    =
    \begin{pmatrix}
        \operatorname{arctanh}(0.05)\\
        \operatorname{arctanh}(0.5)
    \end{pmatrix}
    \approx
    \begin{pmatrix}
        0.05004\\
        0.54931
    \end{pmatrix},
    \qquad
    J(u^\ast)=-\frac14.
\]

For an arbitrary reference control,
\[
    p_t^u=\Phi_x(X_t^u),
    \qquad
    q_t^u=\Phi_{xx}(X_t^u)\Sigma(u_t),
\]
and
\[
    P_t^u=\Phi_{xx}(X_t^u)
    =
    \begin{pmatrix}
        100\operatorname{sech}^2(X_t^{u,1}) & 0\\
        0 & \operatorname{sech}^2(X_t^{u,2})
    \end{pmatrix}.
\]
At \(u_t^0\equiv0\),
\[
    X_t^0=p_t^0=q_t^0=0,
    \qquad P_t^0=M.
\]
The ordinary and corrected Hamiltonians are
\[
    H_t^{u^0}(a)=-c^\top y(a),
    \qquad
    \mathcal H_t^{u^0}(a)
    =
    -c^\top y(a)+\frac12y(a)^\top My(a)
    =
    \ell(a).
\]
Consequently,
\[
    \arg\min_{a\in U}\mathcal H_t^{u^0}(a)=\{u^\ast\}.
\]
The matrix \(P_t^0\) recovers the optimal allocation between
the two control components in one candidate computation.

For state-space diffusion perturbations
\[
    d^{(1)}=(\delta,0)^\top,
    \qquad
    d^{(2)}=(0,\delta)^\top,
\]
the quadratic contributions are
\[
    \frac12(d^{(1)})^\top P_t^0d^{(1)}
    =
    50\delta^2,
    \qquad
    \frac12(d^{(2)})^\top P_t^0d^{(2)}
    =
    \frac12\delta^2.
\]
Equal-magnitude diffusion perturbations are weighted in
the ratio \(100:1\).
This is the risk-sensitive role of the second-order adjoint:
it evaluates diffusion changes according to their
direction-dependent contribution to future expected costs.
The first state component carries a much larger
terminal-cost weight, so SU-MSA restrains its diffusion
while allowing a larger change in the second component.
The resulting optimal diffusion levels are
\[
    y^{\ast,1}=0.05,
    \qquad
    y^{\ast,2}=0.5.
\]
These distinct weights are generated by the adjoint equation
from the original control problem.

The scalar proximal candidate instead minimizes
\[
    -5\tanh a^1-\frac12\tanh a^2
    +\frac{\rho}{2}\bigl[(a^1)^2+(a^2)^2\bigr].
\]
For every \(\rho>0\), its components satisfy
\[
    a_\rho^1\geq a_\rho^2>0.
\]
Indeed, each component minimizes
\[
    -c_j\tanh a+\frac{\rho}{2}a^2
    \qquad\text{over }[-2,2].
\]
The minimizer is positive and nondecreasing in \(c_j\).
An interior minimizer satisfies
\[
    \rho a\cosh^2a=c_j,
\]
whose left-hand side is strictly increasing for \(a\geq0\);
the constrained minimizer is obtained by truncation at \(2\).
Since \(c_1=5>c_2=1/2\), the ordering follows.

Every scalar soft update preserves this ordering,
\[
    \eta a_\rho^1\geq\eta a_\rho^2,
    \qquad \eta\in(0,1],
\]
whereas the optimal control satisfies
\[
    u^{\ast,1}<u^{\ast,2}.
\]
Thus the scalar proximal update preserves a component
ordering opposite to that of the optimal control.
The risk-sensitive second-order correction assigns the
larger penalty to the more costly diffusion direction
and recovers the optimal allocation in one update.

The scale mismatch can also be quantified.
The cost-variation identity and
\(\lambda_{\max}(P_t^u)\leq100\) imply
\[
    J(v)-J(u)
    \leq
    \mathbb E\int_0^1
    \left[
        H_t^u(v_t)-H_t^u(u_t)
        +50|v_t-u_t|^2
    \right]dt.
\]
This estimate yields the uniform sufficient choice
\(\rho\geq100\) for proximal descent, with strict descent
for a nontrivial proximal update when \(\rho>100\).
In this range, the candidate is interior and satisfies
\[
    \rho a_\rho^1
    =
    5\operatorname{sech}^2a_\rho^1,
    \qquad
    \rho a_\rho^2
    =
    \frac12\operatorname{sech}^2a_\rho^2.
\]
Consequently,
\[
    0<a_\rho^1<\frac5\rho\leq0.05,
    \qquad
    0<a_\rho^2<\frac1{2\rho}\leq0.005.
\]
The uniform proximal scale strongly suppresses the second
component, whose optimal value is approximately \(0.54931\).
SU-MSA uses the distinct weights in \(P_t^0\) and identifies
both optimal components simultaneously.

\begin{remark}[Stepsize acceptance in the analytical examples]
\label{rem:analytical-examples-acceptance}
In all three examples, direct calculation gives
\[
    J(u^\eta)-J(u)=-D_\eta(u,v),
    \qquad u^\eta=u+\eta(v-u),
\]
so any \(C_{\mathrm{ub}}>0\) is admissible. Their corrected Hamiltonians
have bounded action Hessians on the compact control domains, uniformly
in the reference control, so a finite semiconvexity constant \(A\) is
available. Set
\[
    \varepsilon_J=0,
    \qquad C_{\mathrm{ub}}=\frac{\vartheta}{8},
    \qquad \vartheta\in(0,1).
\]
Starting from \(u^0=0\), the candidate satisfies \(v^0=u^\ast\), and
\[
    r_0\geq\frac14\Delta_4(u^0,v^0)>0.
\]
In the first example equality holds. In the other two, the optimal
diffusion coordinates have absolute value at most \(1/2\), and
\(|\operatorname{arctanh}y|\leq4|y|/3\) on this interval. Since the
quadratic weights in the reduced objectives are at least one, this gives
the stronger bound \(r_0\geq(9/32)\Delta_4(u^0,v^0)\).
Hence
\[
    C_{\mathrm{ub}}\Delta_4(u^0,v^0)
    \leq\frac{\vartheta}{2}r_0
    <\vartheta r_0.
\]
By Remark~\ref{rem:explicit-full-step}, the explicit rule selects
\(\eta_0=1\), and the algorithm reaches the unique optimal control
after one update in each example.
\end{remark}

\begin{remark}[Compact and unbounded control domains]
\label{rem:examples-unbounded-control-domain}
The coefficients in these examples satisfy
Assumption~\ref{ass:basic-regularity} on the stated compact
control domains. The optimal controls and the proximal
candidates in the displayed sufficient parameter ranges
lie in the interiors of these domains.

The same optimal controls are obtained on
\(U=\mathbb R\) and \(U=\mathbb R^2\), respectively.
The proximal objectives are coercive for every \(\rho>0\).
The stagnation in the first example, the time-independent
proximal candidate in the second, and the ordering of the
proximal components in the third also hold on these
unbounded domains. The comparisons thus reflect the
candidate-selection mechanisms themselves.
\end{remark}

\subsection{Auxiliary estimates under soft update}
\label{subsec:auxiliary-soft-update}
We next collect several estimates for the soft-update perturbation.
Subtracting the state equations associated with \(u^\eta\) and \(u\), we
obtain
\begin{equation}
\label{eq:state-variation-sde}
\begin{cases}
\displaystyle
    d\delta X_t
    =
    \bigl[
        b(t,X_t^\eta,u_t^\eta)-b(t,X_t^u,u_t)
    \bigr]\,dt
    +
    \bigl[
        \sigma(t,X_t^\eta,u_t^\eta)-\sigma(t,X_t^u,u_t)
    \bigr]\,dW_t,\\[0.2cm]
    \delta X_0=0.
\end{cases}
\end{equation}

\begin{lemma}[State perturbation estimate under the soft update]
\label{lem:state-variation}
Let Assumption~\ref{ass:basic-regularity} hold. Then, for every \(p\in[2,\infty)\), there exists a
constant $C_p>0$, independent of $u$, $v$, and $\eta$, such that, for all
$u,v\in\mathcal U$ and $\eta\in(0,1]$,
\[
\mathbb E\left[
    \sup_{0\le t\le T}|\delta X_t|^p
\right]
\le
C_p\eta^p
\mathbb E\left[
    \left(
        \int_0^T |v_t-u_t|^2\,dt
    \right)^{p/2}
\right].
\]
Consequently, since \(U\) is compact,
\[
    \|\delta X\|_{\mathcal S^p}
    \leq
    C_p\eta.
\]
\end{lemma}
\begin{proof}
Set
\[
    d_t:=v_t-u_t,
    \qquad
    u_t^\eta=u_t+\eta d_t.
\]
Since \(U\) is convex, \(u^\eta\in\mathcal U\). Subtracting the state
equations associated with \(u^\eta\) and \(u\), we obtain
\[
    \delta X_t
    =
    \int_0^t \alpha_s\,ds
    +
    \int_0^t \beta_s\,dW_s,
\]
where
\[
\begin{aligned}
    \alpha_s
    &:=
    b(s,X_s^\eta,u_s^\eta)
    -
    b(s,X_s^u,u_s),\\
    \beta_s
    &:=
    \sigma(s,X_s^\eta,u_s^\eta)
    -
    \sigma(s,X_s^u,u_s).
\end{aligned}
\]
For example, the drift splits as
\[
 \alpha_s=
 [b(s,X_s^\eta,u_s^\eta)-b(s,X_s^u,u_s^\eta)]
 +[b(s,X_s^u,u_s^\eta)-b(s,X_s^u,u_s)],
\]
where the second bracket is bounded by \(L\eta|d_s|\).
The diffusion admits the same decomposition. By the uniform Lipschitz
continuity in the state and control variables,
\[
    |\alpha_s|+|\beta_s|
    \leq
    C\bigl(|\delta X_s|+\eta|d_s|\bigr),
\]
where \(C\) is independent of \(u\), \(v\), and \(\eta\).

Fix \(p\geq2\). Hölder's inequality and the
Burkholder--Davis--Gundy inequality give, for every \(t\in[0,T]\),
\[
\begin{aligned}
    \mathbb E\left[
        \sup_{0\leq s\leq t}|\delta X_s|^p
    \right]
    &\leq
    C_p\,
    \mathbb E\left[
        \left(\int_0^t|\alpha_s|\,ds\right)^p
    \right]
    +
    C_p\,
    \mathbb E\left[
        \left(\int_0^t|\beta_s|^2\,ds\right)^{p/2}
    \right]                                                   \\
    &\leq
    C_p\int_0^t
    \mathbb E\left[
        \sup_{0\leq r\leq s}|\delta X_r|^p
    \right]ds
    +
    C_p\eta^p
    \mathbb E\left[
        \left(
            \int_0^t|d_s|^2\,ds
        \right)^{p/2}
    \right].
\end{aligned}
\]
Since
\[
    \int_0^t|d_s|^2\,ds
    \leq
    \int_0^T|d_s|^2\,ds,
\]
Gronwall's inequality yields
\[
    \mathbb E\left[
        \sup_{0\leq t\leq T}|\delta X_t|^p
    \right]
    \leq
    C_p\eta^p
    \mathbb E\left[
        \left(
            \int_0^T|v_t-u_t|^2\,dt
        \right)^{p/2}
    \right],
\]
where \(C_p\) is independent of \(u\), \(v\), and \(\eta\).

Finally, the compactness of \(U\) implies
\[
    |v_t-u_t|
    \leq
    \operatorname{diam}(U),
    \qquad
    dt\otimes d\mathbb P\text{-a.e.}
\]
Therefore,
\[
    \mathbb E\left[
        \left(
            \int_0^T|v_t-u_t|^2\,dt
        \right)^{p/2}
    \right]
    \leq
    T^{p/2}\operatorname{diam}(U)^p,
\]
and hence
\[
    \|\delta X\|_{\mathcal S^p}
    \leq
    C_p\eta.
\]
\end{proof}

To linearize the state perturbation with respect to the state variable, define
\[
A_t:=b_x(t,X_t^u,u_t),
\qquad
C_t^i:=\sigma_x^i(t,X_t^u,u_t),
\qquad i=1,\ldots,d,
\]
\[
\Delta b_t^\eta
:=
b(t,X_t^u,u_t^\eta)-b(t,X_t^u,u_t),
\qquad
\Delta\sigma_t^{\eta,i}
:=
\sigma^i(t,X_t^u,u_t^\eta)-\sigma^i(t,X_t^u,u_t),
\quad i=1,\ldots,d.
\]
The linearized state perturbation $\xi$ is defined as the unique solution of
\[
\begin{cases}
d\xi_t
=
(A_t\xi_t+\Delta b_t^\eta)\,dt
+
\displaystyle\sum_{i=1}^d
(C_t^i\xi_t+\Delta\sigma_t^{\eta,i})\,dW_t^i,
\\[1mm]
\xi_0=0.
\end{cases}
\]
To quantify the linearization error, set
\[
r_t:=\delta X_t-\xi_t,
\]
and define the corresponding remainder terms by
\[
R_t^b
:=
\left[
\int_0^1
b_x(t,X_t^u+\theta\delta X_t,u_t^\eta)\,d\theta
-
b_x(t,X_t^u,u_t)
\right]\delta X_t,
\]
\[
R_t^{\sigma,i}
:=
\left[
\int_0^1
\sigma_x^i(t,X_t^u+\theta\delta X_t,u_t^\eta)\,d\theta
-
\sigma_x^i(t,X_t^u,u_t)
\right]\delta X_t,
\qquad i=1,\ldots,d.
\]

\begin{lemma}[Linearization and remainder estimates for the state perturbation]
\label{lem:linear-variation-residual}
Let Assumption~\ref{ass:basic-regularity} hold. Then, for every
\(p\in[2,\infty)\), there exists a constant \(C_p>0\), independent of
\(u\), \(v\), and \(\eta\), such that, for all \(u,v\in\mathcal U\) and
\(\eta\in(0,1]\),
\[
\begin{aligned}
    \|\xi\|_{\mathcal S^p}
    +
    \left\|
        \left(
            \int_0^T|\xi_t|^2\,dt
        \right)^{1/2}
    \right\|_{L^p}
    \leq
    C_p\eta
    \left\|
        \left(
            \int_0^T|v_t-u_t|^2\,dt
        \right)^{1/2}
    \right\|_{L^p}.
\end{aligned}
\]

Moreover, there exists a constant \(C>0\), independent of
\(u\), \(v\), and \(\eta\), such that
\[
\begin{aligned}
    \mathbb E\int_0^T
    \left(
        |R_t^b|^2
        +
        \sum_{i=1}^d|R_t^{\sigma,i}|^2
    \right)dt
    \leq
    C\eta^4
    \mathbb E\left[
        \left(
            \int_0^T|v_t-u_t|^2\,dt
        \right)^2
    \right],
\end{aligned}
\]
and
\[
    \|r\|_{\mathcal S^2}
    \leq
    C\eta^2\|v-u\|_{\mathbb H^4}^2.
\]
\end{lemma}
\begin{proof}
Denote
\[
    d_t:=v_t-u_t,
    \qquad
    D:=\int_0^T|d_t|^2\,dt.
\]
By Assumption~\ref{ass:basic-regularity},
\[
    |A_t|
    +
    \left(
        \sum_{i=1}^d|C_t^i|^2
    \right)^{1/2}
    \leq C,
\]
and
\[
    |\Delta b_t^\eta|
    +
    \left(
        \sum_{i=1}^d|\Delta\sigma_t^{\eta,i}|^2
    \right)^{1/2}
    \leq C\eta|d_t|.
\]
Hence, the standard BDG--Gronwall estimate for linear SDEs yields, for
every \(p\geq2\),
\[
    \|\xi\|_{\mathcal S^p}
    \leq
    C_p\eta\|D^{1/2}\|_{L^p}.
\]
Since
\[
    \left(\int_0^T|\xi_t|^2\,dt\right)^{1/2}
    \leq
    T^{1/2}\sup_{0\leq t\leq T}|\xi_t|,
\]
the first assertion follows.

The boundedness of \(b_{xx}\) and \(\sigma_{xx}\), together with the
Lipschitz continuity of \(b_x\) and \(\sigma_x\) in the control
variable, gives
\[
    |R_t^b|
    +
    \left(
        \sum_{i=1}^d|R_t^{\sigma,i}|^2
    \right)^{1/2}
    \leq
    C\left(
        |\delta X_t|^2
        +
        \eta|d_t|\,|\delta X_t|
    \right).
\]
It follows that
\[
\begin{aligned}
    &\mathbb E\int_0^T
    \left(
        |R_t^b|^2
        +
        \sum_{i=1}^d|R_t^{\sigma,i}|^2
    \right)dt
    \leq
    C\mathbb E\int_0^T|\delta X_t|^4\,dt
    +
    C\eta^2
    \mathbb E\int_0^T|d_t|^2|\delta X_t|^2\,dt.
\end{aligned}
\]
By Lemma~\ref{lem:state-variation} with \(p=4\),
\[
    \mathbb E\int_0^T|\delta X_t|^4\,dt
    \leq
    C\eta^4\mathbb E[D^2].
\]
Moreover, by the Cauchy--Schwarz inequality and the same estimate,
\[
\begin{aligned}
    \mathbb E\int_0^T|d_t|^2|\delta X_t|^2\,dt
    &\leq
    \mathbb E\left[
        D\sup_{0\leq t\leq T}|\delta X_t|^2
    \right]\\
    &\leq
    \bigl(\mathbb E[D^2]\bigr)^{1/2}
    \left(
        \mathbb E\sup_{0\leq t\leq T}|\delta X_t|^4
    \right)^{1/2}\\
    &\leq
    C\eta^2\mathbb E[D^2].
\end{aligned}
\]
Therefore,
\[
    \mathbb E\int_0^T
    \left(
        |R_t^b|^2
        +
        \sum_{i=1}^d|R_t^{\sigma,i}|^2
    \right)dt
    \leq
    C\eta^4\mathbb E[D^2].
\]

Subtracting the equation for \(\xi\) from that for \(\delta X\), we
obtain
\[
\begin{cases}
    dr_t
    =
    \bigl(A_t r_t+R_t^b\bigr)\,dt
    +
    \displaystyle\sum_{i=1}^d
    \bigl(C_t^i r_t+R_t^{\sigma,i}\bigr)\,dW_t^i,\\
    r_0=0.
\end{cases}
\]
Since \(A\) and \(C^i\) are uniformly bounded, the standard
\(L^2\) estimate for this linear SDE gives
\[
\begin{aligned}
    \|r\|_{\mathcal S^2}^2
    &\leq
    C\mathbb E\left[
        \left(\int_0^T|R_t^b|\,dt\right)^2
        +
        \int_0^T\sum_{i=1}^d|R_t^{\sigma,i}|^2\,dt
    \right]\\
    &\leq
    C\mathbb E\int_0^T
    \left(
        |R_t^b|^2
        +
        \sum_{i=1}^d|R_t^{\sigma,i}|^2
    \right)dt\\
    &\leq
    C\eta^4\mathbb E[D^2].
\end{aligned}
\]
Here the second inequality follows from
\[
    \left(\int_0^T|R_t^b|\,dt\right)^2
    \leq T\int_0^T|R_t^b|^2\,dt,
\]
and the last inequality is the squared-forcing estimate proved above.
Taking square roots and using
\[
    \bigl(\mathbb E[D^2]\bigr)^{1/2}
    =
    \|v-u\|_{\mathbb H^4}^2
\]
completes the proof.
\end{proof}
The estimates above show that the soft-update perturbation admits the
decomposition $X_t^\eta-X_t^u=\xi_t+r_t$ 
where \(\xi=O(\eta)\) and \(r=O(\eta^2)\) in the norms stated above.

\subsection{A cost upper bound under soft update}
\label{subsec:cost-upper-bound-soft-update}

We now derive the upper bound for the cost variation induced by a soft
update. This estimate is the key link between the pointwise minimization of
the second-order corrected Hamiltonian and the monotonicity of the original
cost functional.

Let \(u,v\in\mathcal U\) and \(\eta\in(0,1]\). Set
\[
    u_t^\eta:=S_\eta(u,v)_t=u_t+\eta(v_t-u_t).
\]
For notational brevity, define, along the reference trajectory \(X^u\),
\[
    \Delta f_t^\eta
    :=
    f(t,X_t^u,u_t^\eta)-f(t,X_t^u,u_t).
\]
Recall that \((p^u,q^u)\) and \((P^u,Q^u)\) denote the first- and
second-order adjoint processes associated with the reference control \(u\). For later use, we define the corrected-Hamiltonian improvement generated
by the soft update \(u^\eta=u+\eta(v-u)\) as
\[
    D_\eta(u,v)
    :=
    \mathbb E\int_0^T
    \left[
        \mathcal H_t^u(u_t)
        -
        \mathcal H_t^u(u_t^\eta)
    \right]dt.
\]

\begin{theorem}[Cost upper bound under soft update]
\label{thm:cost-upper-bound-soft-update}
Let \(U\subset\mathbb R^k\) be the nonempty compact convex control domain specified in Section~\ref{subsec:problem-formulation}, and let
Assumption~\ref{ass:basic-regularity} hold. Then there exists a constant \(C>0\), independent of \(u\), \(v\), and \(\eta\), such that
\[
\begin{aligned}
    J(u^\eta)-J(u)
    \le
    &\;
    \mathbb E\int_0^T
    \Bigg[
        (p_t^u)^\top \Delta b_t^\eta
        +
        \sum_{i=1}^d
        (q_t^{u,i})^\top \Delta\sigma_t^{\eta,i}
        +
        \Delta f_t^\eta
    \Bigg]dt
    \\
    &\quad
    +
    \frac12
    \mathbb E\int_0^T
    \sum_{i=1}^d
    (\Delta\sigma_t^{\eta,i})^\top
    P_t^u
    \Delta\sigma_t^{\eta,i}\,dt
    +
    C\eta^2\|v-u\|_{\mathbb H^4}^2 .
\end{aligned}
\]
Equivalently,
\[
\begin{aligned}
    J(u^\eta)-J(u)
    \le
    \mathbb E\int_0^T
    \Big[
        \mathcal H_t^u(u_t^\eta)
        -
        \mathcal H_t^u(u_t)
    \Big]dt
    +
    C\eta^2\|v-u\|_{\mathbb H^4}^2,
\end{aligned}
\]
\end{theorem}
\begin{proof}
Denote
\[
    d_t:=v_t-u_t,\qquad
    D:=\int_0^T|d_t|^2\,dt,\qquad
    \rho_4:=\bigl(\mathbb E[D^2]\bigr)^{1/2}
    =\|v-u\|_{\mathbb H^4}^2.
\]
The preceding state perturbation estimates imply
\[
    \|\delta X\|_{\mathcal S^2}+\|\xi\|_{\mathcal S^2}
    \leq C\eta\rho_4^{1/2},
    \qquad
    \|\delta X\|_{\mathcal S^4}^2+\|\xi\|_{\mathcal S^4}^2
    \leq C\eta^2\rho_4.
\]
Moreover, the standard \(L^2\)-estimate for the equation satisfied by
\(r=\delta X-\xi\), together with
Lemma~\ref{lem:linear-variation-residual}, gives
\[
    \|r\|_{\mathcal S^2}\leq C\eta^2\rho_4.
\]
Since \(U\) is compact, \(\rho_4\) is uniformly bounded over
\(u,v\in\mathcal U\).

Taylor expansion with respect to the state variable, followed by the
decomposition \(\delta X=\xi+r\), yields
\[
\begin{aligned}
    J(u^\eta)-J(u)
    &=
    \mathbb E\bigl[\Phi_x(X_T^u)^\top\xi_T\bigr]
    +\mathbb E\int_0^T f_x(t,X_t^u,u_t)^\top\xi_t\,dt
    +\mathbb E\int_0^T\Delta f_t^\eta\,dt\\
    &\quad+
    \frac12\mathbb E\bigl[
        \xi_T^\top\Phi_{xx}(X_T^u)\xi_T
    \bigr]
    +\frac12\mathbb E\int_0^T
        \xi_t^\top f_{xx}(t,X_t^u,u_t)\xi_t\,dt
    +R_0^\eta,
\end{aligned}
\]
where
\[
    |R_0^\eta|\leq C\eta^2\rho_4.
\]
To verify this bound, the linear terms involving \(r\) satisfy
\[
 \left|\mathbb E[\Phi_x(X_T^u)^\top r_T]
       +\mathbb E\int_0^T(f_x^u(t))^\top r_t\,dt\right|
 \leq C(1+\|X^u\|_{\mathcal S^2})\|r\|_{\mathcal S^2}
 \leq C\eta^2\rho_4.
\]
Replacing \(\delta X\) by \(\xi\) in bounded quadratic forms costs at most
\[
 C\bigl(\|\xi\|_{\mathcal S^2}\|r\|_{\mathcal S^2}
             +\|r\|_{\mathcal S^2}^2\bigr)
 \leq C(\eta^3\rho_4^{3/2}+\eta^4\rho_4^2)
 \leq C\eta^2\rho_4.
\]
The terminal and running third-order state Taylor remainders satisfy
\[
    C\mathbb E\left[\sup_{0\leq t\leq T}|\delta X_t|^3\right]
    \leq C\eta^3\rho_4^{3/2}
    \leq C\eta^2\rho_4.
\]
For the change of \(f_x\) in the action argument, H\"older's inequality
with exponents \(4,4,2\) gives
\[
\begin{aligned}
 &\mathbb E\int_0^T
  |f_x(t,X_t^u,u_t^\eta)-f_x(t,X_t^u,u_t)|\,|\delta X_t|\,dt\\
 &\qquad\leq C\eta\,
  \|1+\sup_{t\leq T}|X_t^u|\|_{L^4}
  \|D^{1/2}\|_{L^4}\|\delta X\|_{\mathcal S^2}
  \leq C\eta^2\rho_4.
\end{aligned}
\]
Finally, the change of \(f_{xx}\) in the action argument contributes at
most \(2L\mathbb E\int_0^T|\delta X_t|^2\,dt\leq C\eta^2\rho_4\).
Thus no control-Lipschitz condition on \(f_{xx}\) is needed.

Applying It\^o's formula to \((p_t^u)^\top\xi_t\) and using
the first-order adjoint equation, we obtain
\[
\begin{aligned}
    d\bigl[(p_t^u)^\top\xi_t\bigr]
    ={}&
    \Bigg[
        -(f_x^u(t))^\top\xi_t
        +(p_t^u)^\top\Delta b_t^\eta
        +\sum_{i=1}^d
        (q_t^{u,i})^\top\Delta\sigma_t^{\eta,i}
    \Bigg]dt
    +dM_t^{(1)},
\end{aligned}
\]
where
\[
\begin{aligned}
    M_t^{(1)}
    :=
    \sum_{i=1}^d\int_0^t
    \Big[
        \xi_s^\top q_s^{u,i}
        +(p_s^u)^\top
        \bigl(C_s^i\xi_s+\Delta\sigma_s^{\eta,i}\bigr)
    \Big]\,dW_s^i.
\end{aligned}
\]
Here \(f_x^u(t)=f_x(t,X_t^u,u_t)\). Boundedness of \(C^i\),
the estimate
\[
    \left(
        \sum_{i=1}^d|\Delta\sigma_t^{\eta,i}|^2
    \right)^{1/2}
    \leq C\eta|d_t|,
\]
and the Cauchy--Schwarz inequality imply
\[
\begin{aligned}
    \mathbb E\!\left[[M^{(1)}]_T^{1/2}\right]
    \leq C\Big(
        &\|\xi\|_{\mathcal S^2}\|q^u\|_{\mathbb H^2}\\
        &+\|p^u\|_{\mathcal S^2}
        \bigl[
            \|\xi\|_{\mathcal S^2}
            +\eta\|D^{1/2}\|_{L^2}
        \bigr]
    \Big)
    <\infty,
\end{aligned}
\]
where \([M^{(1)}]\) denotes the quadratic variation.
The BDG inequality therefore gives
\[
    \mathbb E\sup_{t\leq T}|M_t^{(1)}|<\infty,
\]
so
\(M^{(1)}\) is a true martingale with zero expectation.
All drift terms are integrable by the preceding moment estimates.
Integrating, taking expectations, and using
\(\xi_0=0\) and \(p_T^u=\Phi_x(X_T^u)\), we obtain
\[
\begin{aligned}
    &\mathbb E\bigl[\Phi_x(X_T^u)^\top\xi_T\bigr]
    +\mathbb E\int_0^T(f_x^u(t))^\top\xi_t\,dt\\
    &\qquad=
    \mathbb E\int_0^T
    \left[
        (p_t^u)^\top\Delta b_t^\eta
        +\sum_{i=1}^d
        (q_t^{u,i})^\top\Delta\sigma_t^{\eta,i}
    \right]dt.
\end{aligned}
\]

Next, applying It\^o's formula to
\(\frac12\xi_t^\top P_t^u\xi_t\), using the second-order
adjoint equation and the symmetry of \(P^u,Q^{u,i}\), yields
\[
\begin{aligned}
    d\left(\frac12\xi_t^\top P_t^u\xi_t\right)
    ={}&
    \Bigg[
        -\frac12\xi_t^\top H_{xx}^u(t)\xi_t
        +\xi_t^\top P_t^u\Delta b_t^\eta\\
        &\quad+
        \sum_{i=1}^d
        \xi_t^\top
        \bigl[(C_t^i)^\top P_t^u+Q_t^{u,i}\bigr]
        \Delta\sigma_t^{\eta,i}\\
        &\quad+
        \frac12\sum_{i=1}^d
        (\Delta\sigma_t^{\eta,i})^\top
        P_t^u\Delta\sigma_t^{\eta,i}
    \Bigg]dt
    +dM_t^{(2)},
\end{aligned}
\]
where
\[
\begin{aligned}
    M_t^{(2)}
    :=
    \sum_{i=1}^d\int_0^t
    \Bigg[
        \frac12\xi_s^\top Q_s^{u,i}\xi_s
        +\xi_s^\top P_s^u
        \bigl(C_s^i\xi_s+\Delta\sigma_s^{\eta,i}\bigr)
    \Bigg]\,dW_s^i.
\end{aligned}
\]
The cancellation of the terms containing \(A_t\), \(C_t^i\),
and \(Q_t^{u,i}C_t^i\) follows directly from the second-order
adjoint equation. By Cauchy--Schwarz in time and H\"older's
inequality with exponents \(2,2\) and \(2,4,4\),
\[
\begin{aligned}
    \mathbb E\!\left[[M^{(2)}]_T^{1/2}\right]
    \leq C\Big(
        &\bigl[
            \|P^u\|_{\mathcal S^2}
            +\|Q^u\|_{\mathbb H^2}
        \bigr]\|\xi\|_{\mathcal S^4}^2\\
        &+\eta\|P^u\|_{\mathcal S^2}
        \|\xi\|_{\mathcal S^4}
        \|D^{1/2}\|_{L^4}
    \Big)
    <\infty.
\end{aligned}
\]
Thus the BDG inequality also gives
\(\mathbb E\sup_{t\leq T}|M_t^{(2)}|<\infty\), and
\(M^{(2)}\) is a true martingale with zero expectation.

The boundedness of \(b_{xx}\) and \(\sigma_{xx}\) gives
\[
    |H_{xx}^u(t)-f_{xx}^u(t)|
    \leq C\bigl(|p_t^u|+|q_t^u|\bigr),
    \qquad
    f_{xx}^u(t):=f_{xx}(t,X_t^u,u_t).
\]
Together with the preceding moment estimates, this ensures
integrability of all drift terms in the quadratic identity.
Integrating, taking expectations, and using
\(P_T^u=\Phi_{xx}(X_T^u)\), we obtain
\[
\begin{aligned}
    &\frac12\mathbb E\bigl[
        \xi_T^\top\Phi_{xx}(X_T^u)\xi_T
    \bigr]
    +\frac12\mathbb E\int_0^T
        \xi_t^\top f_{xx}^u(t)\xi_t\,dt\\
    &\qquad=
    \frac12\mathbb E\int_0^T
        \sum_{i=1}^d
        (\Delta\sigma_t^{\eta,i})^\top
        P_t^u\Delta\sigma_t^{\eta,i}\,dt
    +R_P^\eta,
\end{aligned}
\]
where
\[
\begin{aligned}
    R_P^\eta
    :=
    \mathbb E\int_0^T\Bigg[
        &\xi_t^\top P_t^u\Delta b_t^\eta
        +\sum_{i=1}^d
        \xi_t^\top
        \bigl[(C_t^i)^\top P_t^u+Q_t^{u,i}\bigr]
        \Delta\sigma_t^{\eta,i}
        -\frac12\xi_t^\top
        \bigl[H_{xx}^u(t)-f_{xx}^u(t)\bigr]\xi_t
    \Bigg]dt.
\end{aligned}
\]

To estimate \(R_P^\eta\), the coefficient-increment bounds,
Cauchy--Schwarz in time, and H\"older's inequality give
\[
\begin{aligned}
    &\mathbb E\int_0^T
    |\xi_t||P_t^u|
    \left(
        |\Delta b_t^\eta|
        +\sum_{i=1}^d|C_t^i|
        |\Delta\sigma_t^{\eta,i}|
    \right)dt
    \leq
    C\eta\|P^u\|_{\mathcal S^2}
    \|\xi\|_{\mathcal S^4}
    \|D^{1/2}\|_{L^4},
\end{aligned}
\]
\[
\begin{aligned}
    &\mathbb E\int_0^T\sum_{i=1}^d
    |\xi_t||Q_t^{u,i}|
    |\Delta\sigma_t^{\eta,i}|\,dt
    \leq
    C\eta\|Q^u\|_{\mathbb H^2}
    \|\xi\|_{\mathcal S^4}
    \|D^{1/2}\|_{L^4},
\end{aligned}
\]
and
\[
\begin{aligned}
    &\mathbb E\int_0^T
    \bigl(|p_t^u|+|q_t^u|\bigr)|\xi_t|^2\,dt
    \leq
    C\bigl(
        \|p^u\|_{\mathcal S^2}
        +\|q^u\|_{\mathbb H^2}
    \bigr)\|\xi\|_{\mathcal S^4}^2.
\end{aligned}
\]
Since
\[
    \|\xi\|_{\mathcal S^4}
    \leq C\eta\rho_4^{1/2},
    \qquad
    \|D^{1/2}\|_{L^4}=\rho_4^{1/2},
\]
the uniform adjoint moment estimates imply
\[
    |R_P^\eta|\leq C\eta^2\rho_4.
\]
Combining the preceding identities, we obtain
\[
\begin{aligned}
    J(u^\eta)-J(u)
    &=
    \mathbb E\int_0^T
    \left[
        (p_t^u)^\top\Delta b_t^\eta
        +\sum_{i=1}^d
        (q_t^{u,i})^\top\Delta\sigma_t^{\eta,i}
        +\Delta f_t^\eta
    \right]dt\\
    &\quad+
    \frac12\mathbb E\int_0^T
        \sum_{i=1}^d
        (\Delta\sigma_t^{\eta,i})^\top
        P_t^u\Delta\sigma_t^{\eta,i}\,dt
    +R_\eta(u,v),
\end{aligned}
\]
where
\[
    |R_\eta(u,v)|
    \leq C\eta^2\rho_4
    =C\eta^2\|v-u\|_{\mathbb H^4}^2.
\]
Taking the upper bound for the remainder proves the first assertion.
The equivalent formulation follows directly from the definition of
\(\mathcal H_t^u\). Every constant used above depends only on \(T,L\),
the dimensions, \(x_0\), and the compact control bound, through the uniform
state and adjoint moment estimates. Thus a single remainder constant is
valid for all admissible \(u,v\) and \(\eta\in[0,1]\).
\end{proof}

The preceding estimate separates the corrected-Hamiltonian variation
from a uniform stability remainder. The following result quantifies
their respective orders under Assumption~\ref{ass:basic-regularity}.

\begin{corollary}[Order separation under the soft update]
\label{cor:soft-update-order-separation}
Let Assumption~\ref{ass:basic-regularity} hold. Then there exists a
constant \(C>0\), independent of \(u\), \(v\), and \(\eta\), such that
\[
    |D_\eta(u,v)|
    \leq
    C\eta\|v-u\|_{L^2(\Omega\times[0,T])}
    +C\eta^2\|v-u\|_{\mathbb H^4}^2.
\]
Moreover,
\[
    J(u^\eta)-J(u)
    \leq
    -D_\eta(u,v)
    +C\eta^2\|v-u\|_{\mathbb H^4}^2.
\]
Thus, the corrected-Hamiltonian variation is bounded by a first-order
term weighted by the \(L^2(\Omega\times[0,T])\)-distance between \(u\)
and \(v\), together with a second-order term weighted by
\(\|v-u\|_{\mathbb H^4}^2\). The stability remainder is of second order
and carries the same \(\mathbb H^4\)-distance factor.
\end{corollary}

\begin{proof}
By the definition of \(\mathcal H_t^u\),
\[
\begin{aligned}
    -D_\eta(u,v)
    &=
    \mathbb E\int_0^T
    \left[
        (p_t^u)^\top\Delta b_t^\eta
        +\sum_{i=1}^d(q_t^{u,i})^\top\Delta\sigma_t^{\eta,i}
        +\Delta f_t^\eta
    \right]dt\\
    &\quad+
    \frac12\mathbb E\int_0^T
    \sum_{i=1}^d
    (\Delta\sigma_t^{\eta,i})^\top
    P_t^u\Delta\sigma_t^{\eta,i}\,dt.
\end{aligned}
\]
The Lipschitz assumptions in the control variable give
\[
    |\Delta b_t^\eta|
    +\left(\sum_{i=1}^d|\Delta\sigma_t^{\eta,i}|^2\right)^{1/2}
    \leq C\eta|v_t-u_t|,
    \qquad
    |\Delta f_t^\eta|
    \leq C\eta(1+|X_t^u|^2)|v_t-u_t|.
\]
Hence, by the Cauchy--Schwarz inequality and the uniform moment estimates
for \(X^u\), \(p^u\), and \(q^u\),
\[
\begin{aligned}
    &\left|
        \mathbb E\int_0^T
        \left[
            (p_t^u)^\top\Delta b_t^\eta
            +\sum_{i=1}^d(q_t^{u,i})^\top\Delta\sigma_t^{\eta,i}
            +\Delta f_t^\eta
        \right]dt
    \right|
    \leq
    C\eta\|v-u\|_{L^2(\Omega\times[0,T])}.
\end{aligned}
\]
Furthermore, the Lipschitz continuity of \(\sigma\) in the control
variable and the uniform \(\mathcal S^2\)-estimate for \(P^u\) yield
\[
\begin{aligned}
    &\left|
        \frac12\mathbb E\int_0^T
        \sum_{i=1}^d
        (\Delta\sigma_t^{\eta,i})^\top
        P_t^u\Delta\sigma_t^{\eta,i}\,dt
    \right|
   \leq
    C\eta^2
    \mathbb E\left[
        \sup_{0\leq t\leq T}|P_t^u|
        \int_0^T|v_t-u_t|^2\,dt
    \right]
    \leq
    C\eta^2\|v-u\|_{\mathbb H^4}^2.
\end{aligned}
\]
Combining the two estimates proves the bound for \(D_\eta(u,v)\). The
cost estimate follows from
Theorem~\ref{thm:cost-upper-bound-soft-update}.
\end{proof}

\subsection{Explicit stepsizes under corrected-Hamiltonian semiconvexity}
\label{subsec:explicit-eta-analysis}

For the candidate pair at iteration \(n\), denote
\[
    u^{n,\eta}:=S_\eta(u^n,v^n),\qquad
    D_n(\eta):=D_\eta(u^n,v^n),\qquad 0\leq\eta\leq1.
\]
The improvement \(D_\eta\) is defined in
Subsection~\ref{subsec:cost-upper-bound-soft-update}, and
\(D_n(1)=r_n\). The cost upper bound becomes
\begin{equation}
\label{eq:conv-cost-upper-bound}
    J(u^{n,\eta})-J(u^n)
    \leq-D_n(\eta)
    +C_{\mathrm{ub}}\eta^2\Delta_4(u^n,v^n).
\end{equation}

\begin{lemma}[Semiconvex interpolation bound]
\label{lem:semiconvex-interpolation}
Under Assumption~\ref{ass:corrected-semiconvexity}, for every
\(\eta\in[0,1]\),
\begin{equation}
\label{eq:semiconvex-improvement-bound}
    D_n(\eta)
    \geq\eta r_n-A\eta(1-\eta)s_n.
\end{equation}
Consequently, when \(s_n>0\),
\begin{equation}
\label{eq:semiconvex-cost-bound}
    J(u^{n,\eta})-J(u^n)
    \leq-\eta s_n
    \bigl[\zeta_n-A-(C_{\mathrm{ub}}\chi_n-A)\eta\bigr].
\end{equation}
\end{lemma}

\begin{proof}
Convexity of \(a\mapsto\mathcal H_t^{u^n}(a)+A|a|^2\) gives
\[
\begin{aligned}
\mathcal H_t^{u^n}(u_t^{n,\eta})+A|u_t^{n,\eta}|^2
&\leq(1-\eta)
  [\mathcal H_t^{u^n}(u_t^n)+A|u_t^n|^2]\\
&\quad+\eta[\mathcal H_t^{u^n}(v_t^n)+A|v_t^n|^2].
\end{aligned}
\]
Using
\[
(1-\eta)|u_t^n|^2+\eta|v_t^n|^2-|u_t^{n,\eta}|^2
=\eta(1-\eta)|v_t^n-u_t^n|^2
\]
and integrating yields \eqref{eq:semiconvex-improvement-bound}.
Substitution into \eqref{eq:conv-cost-upper-bound}, together with
\(r_n=\zeta_ns_n\) and \(\Delta_4(u^n,v^n)=\chi_ns_n\), proves
\eqref{eq:semiconvex-cost-bound}.
\end{proof}

\begin{proposition}[Explicit stepsize and proportional descent]
\label{prop:explicit-eta-certificate}
Let Assumptions~\ref{ass:basic-regularity} and
\ref{ass:corrected-semiconvexity} hold, and suppose \(r_n>0\).
Every positive step selected by \eqref{eq:explicit-eta-rule} satisfies
\begin{equation}
\label{eq:conv-proportional-acceptance}
    C_{\mathrm{ub}}\eta_n^2\Delta_4(u^n,v^n)
    \leq\vartheta D_n(\eta_n),
\end{equation}
and hence
\begin{equation}
\label{eq:explicit-proportional-descent}
    J(u^n)-J(u^{n,\eta_n})
    \geq(1-\vartheta)D_n(\eta_n)>0.
\end{equation}
Moreover, \(\eta_n\) is the largest step satisfying the sufficient
condition
\begin{equation}
\label{eq:eta-linear-certificate}
    [C_{\mathrm{ub}}\chi_n-\vartheta A]\eta
    \leq\vartheta(\zeta_n-A),\qquad0<\eta\leq1.
\end{equation}
If \eqref{eq:explicit-eta-rule} gives \(\eta_n=0\), this sufficient
condition has no solution in \((0,1]\).
\end{proposition}

\begin{proof}
Since \(r_n>0\), one has \(s_n>0\) and
\(\Delta_4(u^n,v^n)>0\). By
\eqref{eq:semiconvex-improvement-bound},
\[
    D_n(\eta)
    \geq
    \eta s_n(\zeta_n-A+A\eta),
    \qquad 0<\eta\leq1.
\]
Hence, a sufficient condition for proportional acceptance is
\[
    C_{\mathrm{ub}}\eta^2\Delta_4(u^n,v^n)
    \leq
    \vartheta\eta s_n(\zeta_n-A+A\eta).
\]
Using \(\Delta_4(u^n,v^n)=\chi_ns_n\) and dividing by
\(\eta s_n>0\), this condition is equivalent to
\begin{equation}
\label{eq:eta-proof-linear-condition}
    \bigl[C_{\mathrm{ub}}\chi_n-\vartheta A\bigr]\eta
    \leq
    \vartheta(\zeta_n-A),
    \qquad 0<\eta\leq1.
\end{equation}
We determine its solution set according to the sign of
\(C_{\mathrm{ub}}\chi_n-\vartheta A\).

\medskip
\noindent
\textup{(i) \(C_{\mathrm{ub}}\chi_n>\vartheta A\).}
In this case, \eqref{eq:eta-proof-linear-condition} becomes
\[
    0<\eta\leq
    \frac{\vartheta(\zeta_n-A)}
         {C_{\mathrm{ub}}\chi_n-\vartheta A},
    \qquad \eta\leq1.
\]
Thus a positive solution exists if and only if \(\zeta_n>A\).
When this condition holds, the solution set is
\[
    \left(
        0,\,
        \min\!\left\{
            1,\,
            \frac{\vartheta(\zeta_n-A)}
                 {C_{\mathrm{ub}}\chi_n-\vartheta A}
        \right\}
    \right],
\]
whose largest element is the first branch of
\eqref{eq:explicit-eta-rule}. If \(\zeta_n\leq A\), the
solution set is empty.

\medskip
\noindent
\textup{(ii) \(C_{\mathrm{ub}}\chi_n=\vartheta A\).}
Condition \eqref{eq:eta-proof-linear-condition} reduces to
\[
    0\leq\vartheta(\zeta_n-A).
\]
Consequently, the solution set is \((0,1]\) if
\[
    \zeta_n\geq A
    =\frac{C_{\mathrm{ub}}\chi_n}{\vartheta},
\]
and is empty otherwise. In the former case, its largest
element is \(1\).

\medskip
\noindent
\textup{(iii) \(C_{\mathrm{ub}}\chi_n<\vartheta A\).}
If \(\zeta_n\geq A\), then
\[
    \bigl[C_{\mathrm{ub}}\chi_n-\vartheta A\bigr]\eta
    <0
    \leq\vartheta(\zeta_n-A)
\]
for every \(\eta\in(0,1]\). Hence the solution set is
\((0,1]\), and its largest element is \(1\).

Suppose now that \(\zeta_n<A\). Dividing
\eqref{eq:eta-proof-linear-condition} by its negative
coefficient gives
\[
    \eta\geq
    \frac{\vartheta(A-\zeta_n)}
         {\vartheta A-C_{\mathrm{ub}}\chi_n}>0.
\]
A solution in \((0,1]\) exists if and only if
\[
\begin{aligned}
    \frac{\vartheta(A-\zeta_n)}
         {\vartheta A-C_{\mathrm{ub}}\chi_n}
    \leq1
    &\quad\Longleftrightarrow\quad
    \vartheta\zeta_n\geq C_{\mathrm{ub}}\chi_n\\
    &\quad\Longleftrightarrow\quad
    \zeta_n\geq\frac{C_{\mathrm{ub}}\chi_n}{\vartheta}.
\end{aligned}
\]
More precisely, if
\[
    \frac{C_{\mathrm{ub}}\chi_n}{\vartheta}
    <\zeta_n<A,
\]
the solution set is the interval
\[
    \left[
        \frac{\vartheta(A-\zeta_n)}
             {\vartheta A-C_{\mathrm{ub}}\chi_n},
        \,1
    \right],
\]
with a lower endpoint strictly between \(0\) and \(1\).
At the boundary
\[
    \zeta_n=\frac{C_{\mathrm{ub}}\chi_n}{\vartheta}<A,
\]
the interval reduces to \(\{1\}\). If
\(\zeta_n<C_{\mathrm{ub}}\chi_n/\vartheta\), the
solution set is empty.

Combining \textup{(ii)} and \textup{(iii)}, whenever
\(C_{\mathrm{ub}}\chi_n\leq\vartheta A\), a positive
solution exists if and only if
\[
    \zeta_n\geq\frac{C_{\mathrm{ub}}\chi_n}{\vartheta},
\]
and its largest element is \(1\). This gives the second
branch of \eqref{eq:explicit-eta-rule}. The zero branch
corresponds precisely to an empty solution set.

Finally, every selected positive step satisfies
\[
    C_{\mathrm{ub}}\eta_n^2\Delta_4(u^n,v^n)
    \leq\vartheta D_n(\eta_n),
\]
which proves \eqref{eq:conv-proportional-acceptance}.
Moreover,
\[
    D_n(\eta_n)
    \geq
    \frac{C_{\mathrm{ub}}}{\vartheta}
    \eta_n^2\Delta_4(u^n,v^n)>0.
\]
Combining this inequality with
\eqref{eq:conv-cost-upper-bound} yields
\[
\begin{aligned}
    J(u^n)-J(u^{n,\eta_n})
    &\geq
    D_n(\eta_n)
    -C_{\mathrm{ub}}\eta_n^2\Delta_4(u^n,v^n)\\
    &\geq(1-\vartheta)D_n(\eta_n)>0,
\end{aligned}
\]
proving \eqref{eq:explicit-proportional-descent}.
\end{proof}

\begin{remark}[Full-step criterion]
\label{rem:explicit-full-step}
For \(r_n>0\), the rule selects \(\eta_n=1\) if and only if
\[
    \vartheta r_n
    \geq C_{\mathrm{ub}}\Delta_4(u^n,v^n).
\]
Indeed, this is \eqref{eq:eta-linear-certificate} evaluated at
\(\eta=1\). It applies in both positive branches and includes equality.
\end{remark}

\begin{remark}[A certified interval with a positive lower endpoint]
\label{rem:stepsize-acceptance-regimes}
When
\[
    \frac{C_{\mathrm{ub}}\chi_n}{\vartheta}
    \leq\zeta_n<A,
\]
one necessarily has \(C_{\mathrm{ub}}\chi_n<\vartheta A\), and the
steps certified by \eqref{eq:eta-linear-certificate} form the interval
\[
    \frac{\vartheta(A-\zeta_n)}
         {\vartheta A-C_{\mathrm{ub}}\chi_n}
    \leq\eta\leq1.
\]
Its lower endpoint is positive. The interval reduces to \(\{1\}\)
when \(\zeta_n=C_{\mathrm{ub}}\chi_n/\vartheta\).
Thus the global estimates can certify a finite update even when
they do not certify arbitrarily small updates. The rule chooses the
full step throughout this regime. Failure of the sufficient condition
does not imply the absence of an actual cost-decreasing step.
\end{remark}

\subsection{Convergence analysis}
\label{subsec:su-convergence}

Throughout this subsection, Assumptions~\ref{ass:basic-regularity} and
\ref{ass:corrected-semiconvexity} hold, \(U\) is compact and convex,
and Algorithm~\ref{alg:su-msa} uses \eqref{eq:explicit-eta-rule}.
By Corollary~\ref{cor:cost-well-posed},
\(J_{\inf}:=\inf_{u\in\mathcal U}J(u)>-\infty\).

\begin{theorem}[Cost convergence and asymptotic regularity]
\label{thm:accepted-decrease-summability}
Under the standing assumptions of this subsection, suppose that
\(\varepsilon_J=0\) and the algorithm generates an infinite
sequence of successful updates. Recall that
\[
\begin{aligned}
    D_n(\eta_n)
    &=
    \mathbb{E}\int_0^T
    \left[
        \mathcal H_t^{u^n}(u_t^n)
        -\mathcal H_t^{u^n}(u_t^{n+1})
    \right]\,dt.
\end{aligned}
\]
Then there exists \(J_\infty\geq J_{\inf}\) such that
\[
    J(u^n)\downarrow J_\infty,
\]
and
\[
    \sum_{n=0}^{\infty}D_n(\eta_n)<\infty,
    \qquad
    \sum_{n=0}^{\infty}
    \|u^{n+1}-u^n\|_{\mathbb{H}^{4}}^2<\infty.
\]
In particular,
\begin{equation}
\label{eq:accepted-asymptotic-regularity}
    D_n(\eta_n)\longrightarrow0,
    \qquad
    \|u^{n+1}-u^n\|_{\mathbb{H}^{4}}\longrightarrow0.
\end{equation}
Moreover,
\begin{equation}
\label{eq:accepted-density-L1}
    \mathbb{E}\int_0^T
    \left|
        \mathcal H_t^{u^n}(u_t^n)
        -\mathcal H_t^{u^n}(u_t^{n+1})
    \right|\,dt
    \longrightarrow0.
\end{equation}
\end{theorem}

\begin{proof}
By Proposition~\ref{prop:explicit-eta-certificate},
\[
    D_n(\eta_n)>0,
    \qquad
    J(u^n)-J(u^{n+1})
    \geq(1-\vartheta)D_n(\eta_n).
\]
Thus \(J(u^n)\) is strictly decreasing and bounded below by
\(J_{\inf}\), so it converges to a finite limit
\(J_\infty\geq J_{\inf}\). Summing the preceding inequality
and passing to the limit gives
\begin{equation}
\label{eq:accepted-decrease-sum}
    \sum_{n=0}^{\infty}D_n(\eta_n)
    \leq
    \frac{J(u^0)-J_\infty}{1-\vartheta}<\infty.
\end{equation}

Since \(u^{n+1}-u^n=\eta_n(v^n-u^n)\), proportional
acceptance yields
\[
    C_{\mathrm{ub}}
    \|u^{n+1}-u^n\|_{\mathbb{H}^{4}}^2
    =
    C_{\mathrm{ub}}\eta_n^2\Delta_4(u^n,v^n)
    \leq\vartheta D_n(\eta_n).
\]
Consequently,
\begin{equation}
\label{eq:accepted-remainder-sum}
    \sum_{n=0}^{\infty}
    \|u^{n+1}-u^n\|_{\mathbb{H}^{4}}^2
    \leq
    \frac{\vartheta}{(1-\vartheta)C_{\mathrm{ub}}}
    [J(u^0)-J_\infty]<\infty.
\end{equation}
The summability of these nonnegative terms implies
\[
    D_n(\eta_n)\longrightarrow0,
    \qquad
    \|u^{n+1}-u^n\|_{\mathbb{H}^{4}}\longrightarrow0,
\]
proving \eqref{eq:accepted-asymptotic-regularity}.

It remains to prove \eqref{eq:accepted-density-L1}. Denote
\[
    g_n(t,\omega):=
    \mathcal H_t^{u^n}(u_t^n)
    -\mathcal H_t^{u^n}(u_t^{n+1}).
\]
The pointwise semiconvex interpolation inequality and the
minimizing property of \(v^n\) give
\[
    g_n(t,\omega)
    \geq
    -A\eta_n(1-\eta_n)|v_t^n-u_t^n|^2.
\]
Using \(|g_n|=g_n+2(-g_n)_+\) and
\(\mathbb{E}\int_0^Tg_n\,dt=D_n(\eta_n)\), we obtain
\[
    \mathbb{E}\int_0^T|g_n|\,dt
    \leq D_n(\eta_n)+2A\eta_n(1-\eta_n)s_n.
\]
Compactness of \(U\) and Jensen's inequality imply
\[
    s_n\leq T\operatorname{diam}(U)^2,
    \qquad
    s_n\leq\Delta_4(u^n,v^n).
\]
Hence
\[
\begin{aligned}
    \eta_ns_n
    &\leq
    \sqrt{T}\operatorname{diam}(U)\,
    \eta_n\sqrt{\Delta_4(u^n,v^n)}\\
    &=
    \sqrt{T}\operatorname{diam}(U)\,
    \|u^{n+1}-u^n\|_{\mathbb{H}^{4}}.
\end{aligned}
\]
It follows that
\[
\begin{aligned}
    0
    &\leq\mathbb{E}\int_0^T|g_n|\,dt\\
    &\leq
    D_n(\eta_n)+
    2A\sqrt{T}\operatorname{diam}(U)\,
    \|u^{n+1}-u^n\|_{\mathbb{H}^{4}}
    \longrightarrow0,
\end{aligned}
\]
which proves \eqref{eq:accepted-density-L1}.
\end{proof}

\begin{corollary}[$L^2$ asymptotic regularity]
\label{cor:accepted-L2-regularity}
Under the hypotheses of
Theorem~\ref{thm:accepted-decrease-summability},
\[
    \|u^{n+1}-u^n\|_{L^2(\Omega\times[0,T])}
    \longrightarrow0.
\]
\end{corollary}

\begin{proof}
By Jensen's inequality,
\[
    \|u^{n+1}-u^n\|_{L^2(\Omega\times[0,T])}
    \leq
    \|u^{n+1}-u^n\|_{\mathbb{H}^{4}}
    \longrightarrow0.
\]
\end{proof}

\begin{theorem}[Residual limits for full and soft updates]
\label{thm:semiconvex-branch-residuals}
Under the hypotheses of
Theorem~\ref{thm:accepted-decrease-summability}, denote
\[
    \mathcal F:=\{n:\eta_n=1\},
    \qquad
    \mathcal S:=\{n:0<\eta_n<1\}.
\]
Then
\[
    \sum_{n\in\mathcal F}r_n<\infty,
    \qquad
    \sum_{n\in\mathcal S}(r_n-As_n)^2<\infty.
\]
Consequently, whenever the corresponding index set is infinite,
\[
    \lim_{\substack{n\to\infty\\ n\in\mathcal F}}r_n=0,
    \qquad
    \lim_{\substack{n\to\infty\\ n\in\mathcal S}}
    (r_n-As_n)=0.
\]
Moreover,
\begin{equation}
\label{eq:semiconvex-positive-excess}
    [r_n-As_n]_+\longrightarrow0,
    \qquad [x]_+:=\max\{x,0\}.
\end{equation}
\end{theorem}

\begin{proof}
For \(n\in\mathcal F\), one has
\[
    u^{n+1}=v^n,
    \qquad
    D_n(\eta_n)=D_n(1)=r_n.
\]
Hence, by \eqref{eq:accepted-decrease-sum},
\begin{equation}
\label{eq:full-step-residual-sum}
    \sum_{n\in\mathcal F}r_n
    =
    \sum_{n\in\mathcal F}D_n(\eta_n)
    \leq
    \frac{J(u^0)-J_\infty}{1-\vartheta}.
\end{equation}
Hence, if \(\mathcal F\) is infinite,
\[
    \lim_{\substack{n\to\infty\\ n\in\mathcal F}}r_n=0.
\]

For \(n\in\mathcal S\),
\eqref{eq:explicit-eta-rule} gives
\[
    r_n-As_n>0,
    \qquad
    \eta_n
    =
    \frac{\vartheta(r_n-As_n)}
         {C_{\mathrm{ub}}\Delta_4(u^n,v^n)-\vartheta As_n},
\]
with
\[
    0<
    C_{\mathrm{ub}}\Delta_4(u^n,v^n)-\vartheta As_n
    \leq
    C_{\mathrm{ub}}\Delta_4(u^n,v^n).
\]
By \eqref{eq:semiconvex-improvement-bound},
\[
\begin{aligned}
    D_n(\eta_n)
    &\geq
    \eta_n(r_n-As_n)+A\eta_n^2s_n\\
    &\geq
    \frac{\vartheta(r_n-As_n)^2}
         {C_{\mathrm{ub}}\Delta_4(u^n,v^n)-\vartheta As_n}\\
    &\geq
    \frac{\vartheta(r_n-As_n)^2}
         {C_{\mathrm{ub}}\Delta_4(u^n,v^n)}.
\end{aligned}
\]
Since
\[
    \int_0^T|v_t^n-u_t^n|^2\,dt
    \leq T\operatorname{diam}(U)^2
    \quad\text{almost surely},
\]
it follows that
\[
    \Delta_4(u^n,v^n)
    \leq T\operatorname{diam}(U)^2,
\]
and therefore
\[
    (r_n-As_n)^2
    \leq
    \frac{C_{\mathrm{ub}}T\operatorname{diam}(U)^2}
         {\vartheta}
    D_n(\eta_n),
    \qquad n\in\mathcal S.
\]
Summing and applying \eqref{eq:accepted-decrease-sum} yields
\begin{equation}
\label{eq:soft-step-excess-sum}
\begin{aligned}
    \sum_{n\in\mathcal S}(r_n-As_n)^2
    &\leq
    \frac{C_{\mathrm{ub}}T\operatorname{diam}(U)^2}
         {\vartheta}
    \sum_{n\in\mathcal S}D_n(\eta_n)\\
    &\leq
    \frac{C_{\mathrm{ub}}T\operatorname{diam}(U)^2}
         {\vartheta(1-\vartheta)}
    [J(u^0)-J_\infty].
\end{aligned}
\end{equation}
Hence, if \(\mathcal S\) is infinite,
\[
    \lim_{\substack{n\to\infty\\ n\in\mathcal S}}
    (r_n-As_n)=0.
\]

Finally, every iteration belongs to either
\(\mathcal F\) or \(\mathcal S\), and
\[
    0\leq[r_n-As_n]_+
    \leq
    \begin{cases}
        r_n, & n\in\mathcal F,\\
        r_n-As_n, & n\in\mathcal S.
    \end{cases}
\]
The preceding limits therefore give
\[
    [r_n-As_n]_+\longrightarrow0,
\]
proving \eqref{eq:semiconvex-positive-excess}.
\end{proof}

\begin{corollary}[SMP residual convergence for a convex corrected Hamiltonian]
\label{cor:convex-corrected-SMP}
Under the standing assumptions of this subsection, suppose that
Assumption~\ref{ass:corrected-semiconvexity} holds with \(A=0\),
and take \(\varepsilon_J=0\).
At every iteration with \(r_n>0\), the explicit stepsize is
\begin{equation}
\label{eq:convex-explicit-eta}
    \eta_n
    =
    \min\!\left\{
        1,
        \frac{\vartheta r_n}
             {C_{\mathrm{ub}}\Delta_4(u^n,v^n)}
    \right\}
    >0.
\end{equation}
The algorithm either terminates at a control satisfying the
corrected-SMP minimum condition or generates an infinite sequence
such that
\begin{equation}
\label{eq:convex-SMP-L1}
    \left\|
        \mathcal H_t^{u^n}(u_t^n)
        -\min_{a\in U}\mathcal H_t^{u^n}(a)
    \right\|_{L^1(\Omega\times[0,T])}
    =
    r_n
    \longrightarrow0.
\end{equation}
For an infinite sequence,
\(\sum_{n\in\mathcal F}r_n<\infty\) and
\(\sum_{n\in\mathcal S}r_n^2<\infty\).
\end{corollary}

\begin{proof}
For \(r_n>0\), one has \(s_n>0\) and
\(\Delta_4(u^n,v^n)>0\).
Setting \(A=0\) in \eqref{eq:explicit-eta-rule} gives
\[
    \eta_n
    =
    \min\!\left\{
        1,
        \frac{\vartheta\zeta_n}
             {C_{\mathrm{ub}}\chi_n}
    \right\}
    =
    \min\!\left\{
        1,
        \frac{\vartheta r_n}
             {C_{\mathrm{ub}}\Delta_4(u^n,v^n)}
    \right\}
    >0.
\]
By Proposition~\ref{prop:explicit-eta-certificate},
\[
\begin{aligned}
    J(u^n)-J(u^{n,\eta_n})
    &\geq
    (1-\vartheta)D_n(\eta_n)\\
    &\geq
    \frac{(1-\vartheta)C_{\mathrm{ub}}}{\vartheta}
    \eta_n^2\Delta_4(u^n,v^n)
    >0.
\end{aligned}
\]
Thus, with \(\varepsilon_J=0\), every iteration with
\(r_n>0\) produces a successful update.
Finite termination therefore occurs only at \(r_n=0\).

Since
\[
    \mathcal H_t^{u^n}(u_t^n)
    -\min_{a\in U}\mathcal H_t^{u^n}(a)
    \geq0,
\]
the definition of \(r_n\) gives
\[
    r_n=0
    \quad\Longleftrightarrow\quad
    \mathcal H_t^{u^n}(u_t^n)
    =
    \min_{a\in U}\mathcal H_t^{u^n}(a)
    \quad dt\otimes d\mathbb P\text{-almost everywhere}.
\]
This proves the finite-termination assertion.

For an infinite sequence,
Theorem~\ref{thm:semiconvex-branch-residuals} with \(A=0\)
yields
\[
    r_n=[r_n-As_n]_+\longrightarrow0.
\]
Finally, nonnegativity of the residual density implies
\[
\begin{aligned}
    \left\|
        \mathcal H_t^{u^n}(u_t^n)
        -\min_{a\in U}\mathcal H_t^{u^n}(a)
    \right\|_{L^1(\Omega\times[0,T])}
    &=
    \mathbb E\int_0^T
    \left[
        \mathcal H_t^{u^n}(u_t^n)
        -\min_{a\in U}\mathcal H_t^{u^n}(a)
    \right]\,dt\\
    &=r_n\longrightarrow0,
\end{aligned}
\]
proving \eqref{eq:convex-SMP-L1}. The two summability assertions follow
from \eqref{eq:full-step-residual-sum} and
\eqref{eq:soft-step-excess-sum} with \(A=0\).
\end{proof}

\begin{remark}[SMP residual convergence when \(A>0\)]
\label{rem:semiconvex-residual-scope}
Under the hypotheses of
Theorem~\ref{thm:accepted-decrease-summability}, let \(A>0\).
If \(s_n\to0\) along \(\mathcal S\), then \(r_n\to0\)
for the entire sequence. Indeed, on \(\mathcal S\),
\[
    0\leq r_n
    \leq [r_n-As_n]_++As_n
    \longrightarrow0,
\]
while convergence on \(\mathcal F\) follows from
Theorem~\ref{thm:semiconvex-branch-residuals}.
A sufficient condition for \(s_n\to0\) along \(\mathcal S\)
is that the accepted steps are uniformly bounded away from
zero on \(\mathcal S\), since
\[
    \eta_n^2s_n
    \leq
    \eta_n^2\Delta_4(u^n,v^n)
    =
    \|u^{n+1}-u^n\|_{\mathbb{H}^{4}}^2
    \longrightarrow0.
\]
\end{remark}

\begin{remark}[Finite termination with a positive cost-decrease tolerance]
\label{rem:smp-positive-tolerance}
If \(\varepsilon_J>0\), every accepted update satisfies
\[
    J(u^n)-J(u^{n+1})>\varepsilon_J.
\]
Since \(J\geq J_{\inf}\), the number \(N\) of accepted
updates is finite and satisfies
\[
    N\varepsilon_J\leq J(u^0)-J_{\inf}.
\]
\end{remark}

The corrected-SMP residual has a direct global optimality interpretation
for the following class of control problems. This is a separate structural
specialization of Assumption~\ref{ass:basic-regularity}. The global
optimality theorem below assumes convergence of the full residual
\(r_n\); for a general semiconvex corrected Hamiltonian with \(A>0\),
that convergence requires an additional argument beyond
Theorem~\ref{thm:semiconvex-branch-residuals}.

For a sufficiently smooth deterministic function
\(V:[0,T]\times\mathbb R^n\to\mathbb R\), denote
\[
    \mathcal L^aV(t,x)
    :=
    V_t(t,x)
    +
    \langle V_x(t,x),b(t,x,a)\rangle
    +
    \frac12\sum_{i=1}^d
    \sigma^i(t,x,a)^\top
    V_{xx}(t,x)
    \sigma^i(t,x,a).
\]

\begin{assumption}[Exact It\^o reduction]
\label{ass:exact-ito-reduction}
In addition to Assumption~\ref{ass:basic-regularity}, suppose
that there exist
\[
    V\in C^{1,4}
    \bigl([0,T]\times\mathbb R^n\bigr),
    \qquad
    c_\Phi\in\mathbb R,
\]
and a continuous function
\[
    \ell:[0,T]\times U\to\mathbb R
\]
such that
\begin{equation}
\label{eq:exact-ito-terminal-cost}
    \Phi(x)=V(T,x)+c_\Phi
\end{equation}
and
\begin{equation}
\label{eq:exact-ito-running-cost}
    f(t,x,a)
    =
    \ell(t,a)-\mathcal L^aV(t,x).
\end{equation}
The function \(\ell\) is continuously differentiable in
\(a\) on a neighborhood of \(U\), and
\[
    L_\ell
    :=
    \sup_{(t,a)\in[0,T]\times U}
    |\ell_a(t,a)|
    <\infty.
\]
Here \(C^{1,4}\) includes the continuous mixed derivatives
\(V_{tx}\) and \(V_{txx}\) used when applying It\^o's formula to the spatial
derivatives. The derivatives appearing below have at most polynomial
growth, uniformly in \(t\), so that It\^o's formula is applicable to
\[
    V(t,X_t^u),\qquad
    V_x(t,X_t^u),\qquad
    V_{xx}(t,X_t^u)
\]
for every \(u\in\mathcal U\).
No convexity of \(a\mapsto\ell(t,a)\) is imposed.
\end{assumption}

\begin{lemma}[Exact corrected-Hamiltonian representation]
\label{lem:exact-corrected-representation}
Let Assumption~\ref{ass:exact-ito-reduction} hold.
For every \(u\in\mathcal U\), the associated adjoint
processes satisfy
\begin{equation}
\label{eq:exact-adjoint-representation}
\begin{aligned}
    p_t^u
    &=
    V_x(t,X_t^u),\\
    q_t^{u,i}
    &=
    V_{xx}(t,X_t^u)
    \sigma^i(t,X_t^u,u_t),\\
    P_t^u
    &=
    V_{xx}(t,X_t^u),
    \qquad i=1,\ldots,d,
\end{aligned}
\end{equation}
and
\[
    Q_t^{u,i}
    =
    D_xV_{xx}(t,X_t^u)
    \left[
        \sigma^i(t,X_t^u,u_t)
    \right].
\]

Moreover, there exists a progressively measurable process
\(C^u\), independent of the candidate action \(a\), such that
\begin{equation}
\label{eq:exact-corrected-density}
    \mathcal H_t^u(a)
    =
    \ell(t,a)+C_t^u,
    \qquad a\in U,
\end{equation}
where
\begin{equation}
\label{eq:exact-candidate-independent-term}
    C_t^u
    =
    -V_t(t,X_t^u)
    +
    \frac12\sum_{i=1}^d
    \sigma^i(t,X_t^u,u_t)^\top
    P_t^u
    \sigma^i(t,X_t^u,u_t).
\end{equation}

Consequently, for every \(u,w\in\mathcal U\),
\begin{equation}
\label{eq:exact-cost-corrected-identity}
    J(w)-J(u)
    =
    \mathbb E\int_0^T
    \left[
        \mathcal H_t^u(w_t)
        -
        \mathcal H_t^u(u_t)
    \right]dt.
\end{equation}

If
\[
    m(t):=\min_{a\in U}\ell(t,a),
\]
then
\begin{equation}
\label{eq:exact-optimal-value}
    J_{\inf}
    =
    V(0,x_0)+c_\Phi+\int_0^Tm(t)\,dt,
\end{equation}
and, for every \(u\in\mathcal U\),
\begin{equation}
\label{eq:residual-exact-gap}
    r(u)
    :=
    \mathbb E\int_0^T
    \left[
        \mathcal H_t^u(u_t)
        -
        \min_{a\in U}\mathcal H_t^u(a)
    \right]dt
    =
    J(u)-J_{\inf}.
\end{equation}
\end{lemma}

\begin{proof}
The polynomial-growth bounds, the linear growth of \(\sigma\), and
Lemma~\ref{lem:state-moment-estimate} imply square integrability of the
stochastic integrands in It\^o's formulas for \(V,V_x,V_{xx}\).
Those stochastic integrals are therefore true martingales. The same
bounds put the proposed adjoints in \(\mathcal S^4\times\mathbb H^4\).
It\^o's formula and
\eqref{eq:exact-ito-running-cost} give
\[
\begin{aligned}
    J(u)
    &=
    \mathbb E
    \left[
        V(T,X_T^u)+c_\Phi
        +
        \int_0^T
        \bigl(
            \ell(t,u_t)
            -
            \mathcal L^{u_t}V(t,X_t^u)
        \bigr)dt
    \right]\\
    &=
    V(0,x_0)+c_\Phi
    +
    \mathbb E\int_0^T\ell(t,u_t)\,dt.
\end{aligned}
\]
Therefore, for every \(u,w\in\mathcal U\),
\begin{equation}
\label{eq:exact-reduced-cost-difference}
    J(w)-J(u)
    =
    \mathbb E\int_0^T
    [\ell(t,w_t)-\ell(t,u_t)]\,dt.
\end{equation}

Applying It\^o's formula to \(V_x(t,X_t^u)\) and
differentiating
\eqref{eq:exact-ito-running-cost} with respect to \(x\)
show that
\[
\begin{aligned}
    dV_x(t,X_t^u)
    &=
    -H_x
    \left(
        t,X_t^u,
        V_x(t,X_t^u),
        q_t^u,
        u_t
    \right)dt\\
    &\quad+
    \sum_{i=1}^d
    V_{xx}(t,X_t^u)
    \sigma^i(t,X_t^u,u_t)\,dW_t^i,
\end{aligned}
\]
with terminal value
\[
    V_x(T,X_T^u)=\Phi_x(X_T^u).
\]
Uniqueness of the first-order adjoint equation gives
\[
    p_t^u=V_x(t,X_t^u),
    \qquad
    q_t^{u,i}
    =
    V_{xx}(t,X_t^u)
    \sigma^i(t,X_t^u,u_t).
\]

Similarly, applying It\^o's formula to
\(V_{xx}(t,X_t^u)\) and differentiating
\eqref{eq:exact-ito-running-cost} twice with respect to \(x\)
give the following identity. In this calculation denote
\[
\begin{aligned}
 \widehat P_t&:=V_{xx}(t,X_t^u),\\
 \widehat Q_t^i&:=D_xV_{xx}(t,X_t^u)[\sigma^i(t,X_t^u,u_t)].
\end{aligned}
\]
With all action arguments held fixed during state differentiation,
\[
\begin{aligned}
 \mathcal L^{u_t}(V_{xx})(t,X_t^u)
 =-\Bigl[&(b_x^u)^\top\widehat P+\widehat P b_x^u
        +\sum_i(\sigma_x^{u,i})^\top\widehat P\sigma_x^{u,i}\\
        &+\sum_i\{(\sigma_x^{u,i})^\top\widehat Q^i
                       +(\widehat Q^i)^\top\sigma_x^{u,i}\}
        +H_{xx}^u\Bigr]_t.
\end{aligned}
\]
Here \(\mathcal L^a\) acts componentwise on matrix-valued functions.
The terms containing \(b_{xx}\) and \(\sigma_{xx}\) are precisely those
in \(H_{xx}^u\); the cross terms with \(D_xV_{xx}[\sigma^i]\) give the
displayed \(Q\)-terms. This is the second-order adjoint equation with
\[
    P_t^u=V_{xx}(t,X_t^u)
\]
and
\[
    Q_t^{u,i}
    =
    D_xV_{xx}(t,X_t^u)
    \left[
        \sigma^i(t,X_t^u,u_t)
    \right].
\]
The terminal condition follows from
\[
    V_{xx}(T,X_T^u)=\Phi_{xx}(X_T^u).
\]
Uniqueness of the second-order adjoint equation proves
\eqref{eq:exact-adjoint-representation}.

For fixed \(u\), denote
\[
    \sigma_t^i(a):=\sigma^i(t,X_t^u,a).
\]
Substituting the preceding representations and
\eqref{eq:exact-ito-running-cost} into the ordinary
Hamiltonian yields
\[
\begin{aligned}
    H_t^u(a)
    ={}&
    \ell(t,a)-V_t(t,X_t^u)\\
    &+
    \sum_{i=1}^d
    \sigma_t^i(u_t)^\top
    P_t^u\sigma_t^i(a)
    -
    \frac12\sum_{i=1}^d
    \sigma_t^i(a)^\top
    P_t^u\sigma_t^i(a).
\end{aligned}
\]
Adding the second-order correction gives
\[
\begin{aligned}
    \mathcal H_t^u(a)
    ={}&
    H_t^u(a)
    +
    \frac12\sum_{i=1}^d
    [\sigma_t^i(a)-\sigma_t^i(u_t)]^\top
    P_t^u
    [\sigma_t^i(a)-\sigma_t^i(u_t)]\\
    ={}&
    \ell(t,a)-V_t(t,X_t^u)
    +
    \frac12\sum_{i=1}^d
    \sigma_t^i(u_t)^\top
    P_t^u\sigma_t^i(u_t).
\end{aligned}
\]
This proves
\eqref{eq:exact-corrected-density} and
\eqref{eq:exact-candidate-independent-term}.
Combining
\eqref{eq:exact-corrected-density} with
\eqref{eq:exact-reduced-cost-difference} proves
\eqref{eq:exact-cost-corrected-identity}.

Since \(U\) is compact and \(\ell\) is continuous, a
measurable selection
\[
    a^\star(t)\in\arg\min_{a\in U}\ell(t,a)
\]
exists. It defines an admissible deterministic control.
Hence
\[
    J_{\inf}
    =
    V(0,x_0)+c_\Phi+\int_0^Tm(t)\,dt.
\]
Finally,
\[
\begin{aligned}
    r(u)
    &=
    \mathbb E\int_0^T
    [\ell(t,u_t)-m(t)]\,dt\\
    &=
    J(u)-J_{\inf},
\end{aligned}
\]
which proves
\eqref{eq:residual-exact-gap}.
\end{proof}

\begin{theorem}[Global optimality of a convergent SU-MSA sequence]
\label{thm:global-optimality-exact-reduction}
Let Assumption~\ref{ass:exact-ito-reduction} hold, and let
\(\{u^n\}_{n\geq0}\subset\mathcal U\) be a sequence generated
by Algorithm~\ref{alg:su-msa}. Suppose that
\[
    r_n:=r(u^n)\longrightarrow0.
\]
Then
\begin{equation}
\label{eq:exact-optimal-value-convergence}
    J(u^n)\longrightarrow J_{\inf}.
\end{equation}

If, in addition, there exists \(u^\ast\in\mathcal U\) such that
\[
    u^n\longrightarrow u^\ast
    \quad\text{in }
    L^2(\Omega\times[0,T];\mathbb R^k),
\]
then \(u^\ast\) is globally optimal:
\begin{equation}
\label{eq:exact-limit-global-optimality}
    J(u^\ast)=J_{\inf}.
\end{equation}
Moreover,
\begin{equation}
\label{eq:exact-limit-pointwise-optimality}
    \ell(t,u_t^\ast)
    =
    \min_{a\in U}\ell(t,a),
    \qquad
    dt\otimes d\mathbb P\text{-a.e.}
\end{equation}
\end{theorem}

\begin{proof}
By Lemma~\ref{lem:exact-corrected-representation},
\[
    J(u^n)-J_{\inf}=r_n.
\]
Thus \(r_n\to0\) directly yields
\[
    J(u^n)\longrightarrow J_{\inf}.
\]

Since \(\ell_a\) is uniformly bounded on
\([0,T]\times U\),
\[
    |\ell(t,a)-\ell(t,\bar a)|
    \leq L_\ell|a-\bar a|,
    \qquad a,\bar a\in U.
\]
Using the exact cost representation, we obtain
\[
\begin{aligned}
    |J(u^n)-J(u^\ast)|
    &\leq
    L_\ell
    \mathbb E\int_0^T
    |u_t^n-u_t^\ast|\,dt\\
    &\leq
    L_\ell\sqrt T
    \|u^n-u^\ast\|_
    {L^2(\Omega\times[0,T])}
    \longrightarrow0.
\end{aligned}
\]
It follows that
\[
    J(u^\ast)=J_{\inf}.
\]

Finally,
\[
\begin{aligned}
    0
    &=
    J(u^\ast)-J_{\inf}\\
    &=
    \mathbb E\int_0^T
    [\ell(t,u_t^\ast)-m(t)]\,dt.
\end{aligned}
\]
The integrand is nonnegative and must therefore vanish
\(dt\otimes d\mathbb P\)-almost everywhere. This proves
\eqref{eq:exact-limit-pointwise-optimality}.
\end{proof}

\subsection{Local extremum reversal induced by the second-order adjoint}
\label{subsec:local-extremum-reversal}

We next examine how the second-order adjoint correction changes the
local extrema of the candidate objective. Under additional assumptions,
we show that the current action can be a strict local minimizer of the
proximal objective of Sethi and \v{S}i\v{s}ka
\cite{SethiSiska2024}, while being a strict local maximizer of the
corrected Hamiltonian. We first establish this reversal on a set with
a quantitative measure bound and then obtain an almost-everywhere
version under stronger conditions.

Throughout this subsection, we restrict attention to $n=k=d=1$
and retain Assumption~\ref{ass:basic-regularity}. The additional
assumptions below are used only in this subsection.
We retain the notation $\mathcal H_t^u$ from
\eqref{eq:Hu-density} and denote
$H_t^u(a):=H(t,X_t^u,p_t^u,q_t^u,a)$.
Unless an action argument is displayed, coefficient derivatives
are evaluated at $(t,X_t^u,u_t)$.
We denote $\mathbb E_t[\cdot]:=\mathbb E[\cdot\mid\mathcal F_t]$.

The following additional boundedness condition strengthens the
moment estimate in Lemma~\ref{lem:first-adjoint-moment}
to an essential bound for $p^u$ and a BMO estimate for
$q^u\cdot W$.

\begin{assumption}[Additional regularity]
\label{ass:cr-additional}
There exist deterministic constants $F,G\ge0$ such that
\[
    |f_x(t,x,a)|\le F,
    \qquad
    |\Phi_x(x)|\le G
\]
for all $(t,x,a)\in[0,T]\times\mathbb R\times U$.
Moreover, $b,\sigma,f$ are twice continuously differentiable in the
action variable on a neighborhood of each interior action under
consideration, with measurable action derivatives.
\end{assumption}

Denote the state-derivative bounds supplied by
Assumption~\ref{ass:basic-regularity} by
\[
    |b_x|\le L_b,
    \qquad
    |\sigma_x|\le L_\sigma,
    \qquad
    |\sigma_{xx}|\le L_{xx}.
\]
Neither $\sigma_{xx}$ nor $\sigma_{aa}$ is required to vanish.
For a continuous Brownian stochastic integral, we use
\[
    \|q\cdot W\|_{\mathrm{BMO}}^2
    :=
    \sup_{\tau\le T}
    \left\|
        \mathbb E\left[
            \int_\tau^T|q_s|^2\,ds
            \,\middle|\,\mathcal F_\tau
        \right]
    \right\|_\infty,
\]
where the supremum is over stopping times.

\begin{proposition}[Uniform adjoint and BMO estimates]
\label{prop:cr-bmo}
Let Assumptions~\ref{ass:basic-regularity} and
\ref{ass:cr-additional} hold, and define
\[
    M_p=e^{L_bT}(G+TF),
    \qquad
    M_q^2
    =
    2G^2+4T(L_b+L_\sigma^2)M_p^2+4TFM_p.
\]
Then, for every admissible control $u$,
\[
    \left\|
        \sup_{0\le t\le T}|p_t^u|
    \right\|_{L^\infty(\Omega)}
    \le M_p,
    \qquad
    \|q^u\cdot W\|_{\mathrm{BMO}}\le M_q.
\]
\end{proposition}

\begin{proof}
For fixed $t$, let $\Xi_{t,t}=1$ and
\[
    d\Xi_{t,s}
    =
    b_{x,s}\Xi_{t,s}\,ds
    +
    \sigma_{x,s}\Xi_{t,s}\,dW_s,
    \qquad s\in[t,T].
\]
This process is strictly positive and has finite maximal moments of
every finite order. Applying the product formula to
$\Xi_{t,s}p_s^u$ and using \eqref{eq:first-adjoint} yields
\[
    p_t^u
    =
    \mathbb E_t\left[
        \Xi_{t,T}\Phi_x(X_T^u)
        +
        \int_t^T\Xi_{t,s}f_{x,s}\,ds
    \right].
\]
The stochastic integral in the product formula is a true martingale:
its expected quadratic-variation square root is bounded by
\[
    \mathbb E\left[
        \sup_{t\le s\le T}\Xi_{t,s}
        \left\{
            \left(\int_t^T|q_s^u|^2\,ds\right)^{1/2}
            +
            L_\sigma\sqrt T
            \sup_{t\le s\le T}|p_s^u|
        \right\}
    \right]
    <\infty,
\]
by Cauchy--Schwarz and
Lemma~\ref{lem:first-adjoint-moment}.
Since $\sigma_x$ is bounded, the stochastic exponential in
\[
    \Xi_{t,s}
    =
    \exp\left(\int_t^s b_{x,r}\,dr\right)
    \exp\left(
        \int_t^s\sigma_{x,r}\,dW_r
        -
        \frac12\int_t^s\sigma_{x,r}^2\,dr
    \right)
\]
has conditional expectation one. Consequently,
$\mathbb E_t\Xi_{t,s}\le e^{L_b(s-t)}$, and
\[
    |p_t^u|
    \le
    Ge^{L_b(T-t)}
    +
    F\int_t^T e^{L_b(s-t)}\,ds
    \le M_p.
\]
Continuity gives the asserted uniform bound.

For any stopping time $\tau\le T$, It\^o's formula gives
\begin{align*}
    |p_\tau^u|^2
    +
    \mathbb E_\tau\int_\tau^T|q_s^u|^2\,ds
    &=
    \mathbb E_\tau|\Phi_x(X_T^u)|^2\\
    &\quad+
    2\mathbb E_\tau\int_\tau^T
    p_s^u
    \bigl(
        b_{x,s}p_s^u
        +
        \sigma_{x,s}q_s^u
        +
        f_{x,s}
    \bigr)\,ds.
\end{align*}
The stochastic integral has zero conditional expectation because
$p^u$ is bounded and $q^u$ is square integrable.
Using
\[
    2|p_s^u\sigma_{x,s}q_s^u|
    \le
    \frac12|q_s^u|^2+2L_\sigma^2M_p^2,
\]
we obtain
\[
    \frac12\mathbb E_\tau\int_\tau^T|q_s^u|^2\,ds
    \le
    G^2+2T(L_b+L_\sigma^2)M_p^2+2TFM_p.
\]
Taking the essential supremum over stopping times proves the result.
\end{proof}

Fix a reference control $u$, and let $\Xi_{t,s}$ denote the
propagation process introduced in the preceding proof.
Set $m=dt\otimes d\mathbb P$ and define
\begin{align*}
    K_t^u
    &:=
    \mathbb E_t\left[
        \Xi_{t,T}^2\bigl(-\Phi_{xx}(X_T^u)\bigr)
        +
        \int_t^T\Xi_{t,s}^2
        \bigl(-f_{xx,s}-p_s^ub_{xx,s}\bigr)\,ds
    \right],\\
    c_\Xi
    &:=
    4L_b+6L_\sigma^2,\\
    \delta_t
    &:=
    M_qL_{xx}
    \left(
        \int_t^T e^{c_\Xi(s-t)}\,ds
    \right)^{1/2},\\
    A_t^u
    &:=
    (K_t^u-\delta_t)|\sigma_{a,t}|^2.
\end{align*}
Conditional-expectation processes are taken in progressively
measurable versions. For $\rho\ge0$, denote the proximal objective by
\[
    \Psi_{\rho,t}^u(a)
    :=
    H_t^u(a)+\frac{\rho}{2}|a-u_t|^2.
\]

\begin{corollary}
[Strict local extremum reversal on a large-measure set]
\label{cor:cr-large-measure}
Let Assumptions~\ref{ass:basic-regularity} and
\ref{ass:cr-additional} hold. Assume that
\[
    f_{x,t}\ge0,
    \qquad
    \Phi_x(X_T^u)\ge0,
\]
and
\[
    f_{xx,t}\le0,
    \qquad
    b_{xx,t}\le0,
    \qquad
    \Phi_{xx}(X_T^u)\le0,
\]
with the terminal inequalities understood almost surely and the
process inequalities $m$-almost everywhere. Suppose also that
\[
    u_t\in\operatorname{int}U,
    \qquad
    H_{a,t}^u(u_t)=0
    \quad m\text{-a.e.},
\]
and that, for deterministic constants $\mu\in\mathbb R$ and
$L_a,B_2\ge0$,
\[
    \mu\le f_{aa,t}\le L_a,
    \qquad
    0\le b_{aa,t}\le B_2
    \quad m\text{-a.e.}
\]
Assume
\[
    N_{aa}
    :=
    \mathbb E\left(
        \int_0^T|\sigma_{aa,t}|^2\,dt
    \right)^{1/2}
    <\infty.
\]
Fix $\rho\ge0$. If there exists $R>0$ such that
\begin{equation}
\label{eq:cr-margins}
    \mu+\rho>R,
    \qquad
    A_t^u>L_a+B_2M_p+R
    \quad m\text{-a.e.},
\end{equation}
then, for $m$-almost every $(t,\omega)$ in
\[
    E_R^u
    :=
    \bigl\{
        (t,\omega):|q_t^u\sigma_{aa,t}|\le R
    \bigr\},
\]
the current action $u_t(\omega)$ is a strict local minimizer of the
Sethi--\v{S}i\v{s}ka proximal objective $\Psi_{\rho,t}^u$ and a
strict local maximizer of the corrected Hamiltonian
$\mathcal H_t^u$.

More precisely, at each such point there exists $r>0$ such that
\begin{equation}
\label{eq:cr-local-reversal}
\begin{aligned}
    \Psi_{\rho,t}^u(a)
    &>
    \Psi_{\rho,t}^u(u_t),\\
    \mathcal H_t^u(a)
    &<
    \mathcal H_t^u(u_t),
\end{aligned}
\qquad
\text{whenever }a\in U,\quad 0<|a-u_t|<r.
\end{equation}
Moreover,
\begin{equation}
\label{eq:cr-measure}
    m(E_R^u)
    \ge
    \max\left\{
        0,\,
        T-\frac{C_FM_qN_{aa}}{R}
    \right\},
\end{equation}
where $C_F$ is a universal Fefferman constant. In particular,
$C_FM_qN_{aa}\le\varepsilon TR$, with $\varepsilon\in(0,1)$,
implies $m(E_R^u)/T\ge1-\varepsilon$.
\end{corollary}

\begin{proof}
The representation in the proof of
Proposition~\ref{prop:cr-bmo} and the sign assumptions imply
$0\le p_t^u\le M_p$, hence $K_t^u\ge0$.
Applying the product formula to $\Xi_{t,s}^2P_s^u$ and using
\eqref{eq:second-adjoint} yields
\[
    P_t^u
    =
    -K_t^u
    +
    \mathbb E_t\int_t^T
    \Xi_{t,s}^2q_s^u\sigma_{xx,s}\,ds.
\]
Indeed, the stochastic integrand is
$\Xi_{t,s}^2(Q_s^u+2\sigma_{x,s}P_s^u)$.
The maximal fourth moment of $\Xi$ and the moment estimates in
Lemma~\ref{lem:second-adjoint-moment} make its stochastic integral
a true martingale, by the same argument as above.
Boundedness of $b_x,\sigma_x$ gives
\[
    \mathbb E_t\Xi_{t,s}^4
    \le
    e^{(4L_b+6L_\sigma^2)(s-t)}.
\]
Conditional Cauchy--Schwarz and
Proposition~\ref{prop:cr-bmo} imply
\begin{align*}
    |P_t^u+K_t^u|
    &\le
    \left(
        \mathbb E_t\int_t^T|q_s^u|^2\,ds
    \right)^{1/2}
    \left(
        \mathbb E_t\int_t^T
        \Xi_{t,s}^4|\sigma_{xx,s}|^2\,ds
    \right)^{1/2}\\
    &\le\delta_t.
\end{align*}
In particular,
\[
    P_t^u|\sigma_{a,t}|^2\le-A_t^u.
\]

At the reference action, differentiation of
\eqref{eq:corrected-H-function} gives
\begin{align*}
    \Psi_{\rho,a,t}^u(u_t)
    &=
    \partial_a\mathcal H_t^u(u_t)
    =
    H_{a,t}^u(u_t)=0,\\
    \Psi_{\rho,aa,t}^u(u_t)
    &=
    H_{aa,t}^u(u_t)+\rho,\\
    \partial_{aa}\mathcal H_t^u(u_t)
    &=
    H_{aa,t}^u(u_t)+P_t^u|\sigma_{a,t}|^2.
\end{align*}
Since
\[
    H_{aa,t}^u(u_t)
    =
    f_{aa,t}+p_t^ub_{aa,t}+q_t^u\sigma_{aa,t},
\]
on $E_R^u$ we have
\[
    \Psi_{\rho,aa,t}^u(u_t)
    \ge
    \mu+\rho-R>0,
\]
and
\[
    \partial_{aa}\mathcal H_t^u(u_t)
    \le
    L_a+B_2M_p+R-A_t^u<0.
\]
The second-order Taylor expansion at the interior stationary action
proves \eqref{eq:cr-local-reversal}, with a radius that may depend
on $(t,\omega)$.

Finally, the martingale $q^u\cdot W$ belongs to BMO, and
$\sigma_{aa}\cdot W$ belongs to the continuous-martingale Hardy
space with norm
\[
    \mathbb E\left(
        \int_0^T|\sigma_{aa,t}|^2\,dt
    \right)^{1/2}
    =
    N_{aa}.
\]
Fefferman's inequality gives
\[
    \mathbb E\int_0^T
    |q_t^u\sigma_{aa,t}|\,dt
    \le
    C_FM_qN_{aa}.
\]
Consequently,
\[
    m((E_R^u)^c)
    \le
    \frac{C_FM_qN_{aa}}{R},
\]
which proves \eqref{eq:cr-measure} and the final assertion.
\end{proof}

\begin{lemma}[Almost-everywhere strict local extremum reversal]
\label{lem:cr-ae}
Under the assumptions of Corollary~\ref{cor:cr-large-measure},
suppose further that \(f\) and \(\sigma\) are twice continuously
differentiable in the action variable on a neighborhood of \(U\).
Assume that there exist deterministic constants
\(F_a,L_{aa}\geq0\) and \(s_0>0\) such that
\[
    |f_a(t,x,a)|\leq F_a,
    \qquad
    |\sigma_a(t,x,a)|\geq s_0,
    \qquad
    |\sigma_{aa}(t,x,a)|\leq L_{aa},
\]
for all \((t,x,a)\in[0,T]\times\mathbb R\times U\), and
\[
    L_{aa}\frac{F_a+M_pL}{s_0}\leq R,
\]
where \(L\) is the control-Lipschitz constant in
Assumption~\ref{ass:basic-regularity}.
Then, \(m\)-almost everywhere, the current action \(u_t\) is a
strict local minimizer of \(\Psi_{\rho,t}^u\) and a strict local
maximizer of \(\mathcal H_t^u\). In particular,
\eqref{eq:cr-local-reversal} holds \(m\)-almost everywhere on
\([0,T]\times\Omega\).
\end{lemma}

\begin{proof}
The control-Lipschitz continuity of \(b\), its action
differentiability, and \(u_t\in\operatorname{int}U\) imply
\(|b_{a,t}|\leq L\).
The stationarity condition gives
\[
    f_{a,t}+p_t^ub_{a,t}+q_t^u\sigma_{a,t}=0.
\]
Evaluating the global bounds along \((t,X_t^u,u_t)\) and using
\(|p_t^u|\leq M_p\), we obtain
\[
    q_t^u
    =
    -\frac{f_{a,t}+p_t^ub_{a,t}}{\sigma_{a,t}},
    \qquad
    |q_t^u|
    \leq
    \frac{F_a+M_pL}{s_0}
    \quad m\text{-a.e.}
\]
Consequently,
\[
    |q_t^u\sigma_{aa,t}|
    \leq
    L_{aa}\frac{F_a+M_pL}{s_0}\leq R
    \quad m\text{-a.e.}
\]
Thus \(m((E_R^u)^c)=0\), and the conclusion follows from
Corollary~\ref{cor:cr-large-measure}.
\end{proof}

\section{DeepSMP: A Neural Realization of the Soft-Update MSA}
\label{sec:deepsmp}

In this section, we develop DeepSMP, a neural implementation of the soft-update MSA introduced in Section 3. The continuous-time scheme is generally not directly computable: the state and adjoint equations rarely admit closed-form solutions, and the pointwise corrected-Hamiltonian minimization must itself be solved numerically and may be nonconvex. Moreover, minimizing actions obtained at finitely many simulated time--state pairs do not by themselves define a feedback law that can be evaluated at unseen states.

DeepSMP addresses these issues through time discretization, Monte Carlo sampling, and neural approximation. At each outer iteration, the reference control network is frozen, the state dynamics are simulated, and the first- and second-order adjoint processes are approximated using the Deep BSDE methodology. A differentiable actor objective is then constructed from a smooth surrogate of the sampled corrected-Hamiltonian minimum. Optimizing this objective yields candidate control-network parameters, which are incorporated into the reference network through a parameter-space soft update. The resulting scheme avoids spatial state grids and produces a reusable feedback law.

The distinction from a direct method lies in the optimization objective. A direct method minimizes a Monte Carlo approximation of the original cost functional \(J(\pi_\theta)\) over the control-network parameters, whereas DeepSMP constructs its actor objective from the state and adjoint variables and the corrected Hamiltonian. Monte Carlo samples are therefore used to estimate this maximum-principle-based objective and its gradient, rather than merely to provide a finite collection of pathwise minimizing actions. Although the resulting actor problem may remain nonconvex, its updates are guided by the action-level structure encoded in the corrected Hamiltonian rather than by the end-to-end cost alone.

\subsection{Deep BSDE approximation of the state and adjoint processes}
\label{subsec:deepsmp-adjoint-approximation}

Let \(0=t_0<t_1<\cdots<t_N=T\) be a uniform partition of \([0,T]\), with
\(\Delta t=T/N\), and denote
\[
    \Delta W_k:=W_{t_{k+1}}-W_{t_k}
    \sim\mathcal N(0,\Delta t\,I_d),
    \qquad k=0,\ldots,N-1.
\]

At the \(l\)-th training iteration, let \(\bar u^l\) denote the frozen
reference control. The corresponding state process is approximated by the
Euler--Maruyama scheme
\[
    X_{k+1}^{l}
    =
    X_k^{l}
    +
    b(t_k,X_k^{l},\bar u_k^{l})\Delta t
    +
    \sigma(t_k,X_k^{l},\bar u_k^{l})\Delta W_k,
    \qquad
    X_0^l=x_0.
\]
The first-order adjoint process associated with the frozen reference control
\(\bar u^l\) satisfies
\[
    dp_t
    =
    -H_x(t,X_t,p_t,q_t,\bar u_t^l)\,dt
    +
    q_t\,dW_t,
    \qquad
    p_T=\Phi_x(X_T).
\]
Following the Deep BSDE method \cite{HanJentzenE2018}, we treat the unknown initial value \(p_0\)
as a trainable parameter and parameterize the martingale integrand by
time-indexed neural networks. Specifically, we set
\[
    p_0^l=\theta_p^l,
    \qquad
    q_k^l=q_{\theta_q,k}(X_k^l),
    \qquad
    k=0,\ldots,N-1,
\]
where \(\theta_p^l\) and the network parameters
\(\{\theta_{q,k}\}_{k=0}^{N-1}\) are trainable. The adjoint process is then
propagated forward according to
\[
    p_{k+1}^{l}
    =
    p_k^{l}
    -
    H_x(t_k,X_k^{l},p_k^{l},q_k^{l},\bar u_k^{l})\Delta t
    +
    q_k^{l}\Delta W_k,
    \qquad
    k=0,\ldots,N-1.
\]
Similarly, the second-order adjoint process satisfies
\[
    dP_t
    =
    -\mathcal F(t,X_t,p_t,q_t,P_t,Q_t,\bar u_t^l)\,dt
    +
    Q_t\,dW_t,
    \qquad
    P_T=\Phi_{xx}(X_T),
\]
where \(\mathcal F\) denotes the driver of the second-order adjoint equation.
We treat \(P_0\) as a trainable parameter and parameterize \(Q\) by
time-indexed neural networks:
\[
    P_0^l=\theta_P^l,
    \qquad
    Q_k^l=Q_{\theta_Q,k}(X_k^l),
    \qquad
    k=0,\ldots,N-1.
\]
The corresponding forward propagation is
\[
    P_{k+1}^{l}
    =
    P_k^{l}
    -
    \mathcal F(t_k,X_k^{l},p_k^{l},q_k^{l},
    P_k^{l},Q_k^{l},\bar u_k^{l})\Delta t
    +
    Q_k^{l}\Delta W_k,
    \qquad
    k=0,\ldots,N-1.
\]
The trainable initial values and network parameters are optimized by
minimizing the terminal matching loss
\[
    \mathcal{L}_{\mathrm{adj}}
    :=
    \mathbb{E}_{\mathrm{batch}}
    \left[
        \lambda_1
        \left|p_N^l-\Phi_x(X_N^l)\right|^2
        +
        \lambda_2
        \left|P_N^l-\Phi_{xx}(X_N^l)\right|^2
    \right],
\]
where \(\lambda_1,\lambda_2>0\) are weighting parameters that balance the
terminal matching errors of the first- and second-order adjoint processes,
respectively.

\subsection{Smooth corrected-Hamiltonian policy improvement}
\label{subsec:deepsmp-policy-improvement}

In the neural notation, \(d_x=n\) and \(d_u=k\) denote the state
and control dimensions, respectively. At time \(t_k\), the actor
produces a raw action \(r_{\theta_u,k}(x)\in\mathbb R^{d_u}\) and
the admissible control
\[
    u_{\theta_u,k}(x):=\Psi_U\bigl(r_{\theta_u,k}(x)\bigr).
\]
For \(U=[u_{\min},u_{\max}]^{d_u}\), we use
\[
    \Psi_U(r):=\frac{u_{\max}+u_{\min}}2
    +\frac{u_{\max}-u_{\min}}2\tanh(r),
\]
with \(\tanh\) applied componentwise.

At iteration \(l\), the reference control \(\bar u^l\) and its state
and adjoint samples are frozen during the actor update. Define
\[
\begin{aligned}
    \mathcal H_k^{l,\mathrm{corr}}(a)
    :={}&H(t_k,X_k^l,p_k^l,q_k^l,a)\\
    &+\frac12\sum_{i=1}^d
    \bigl[\sigma^i(t_k,X_k^l,a)
          -\sigma^i(t_k,X_k^l,\bar u_k^l)\bigr]^\top P_k^l\\
    &\hspace{2.5em}\times
    \bigl[\sigma^i(t_k,X_k^l,a)
          -\sigma^i(t_k,X_k^l,\bar u_k^l)\bigr],
    \qquad a\in U.
\end{aligned}
\]
With \(\zeta_k^{(0)}=0\) and independent perturbations
\(\zeta_k^{(j)}\sim\mathcal N(0,\sigma_{\mathrm{act}}^2I_{d_u})\),
\(j=1,\ldots,K-1\), sampled independently of the reference data, set
\[
\begin{aligned}
    u_k^{l,(j)}
    &:=\Psi_U\bigl(r_{\theta_u,k}(X_k^l)+\zeta_k^{(j)}\bigr),\\
    \mathcal H_k^{l,(j)}
    &:=\mathcal H_k^{l,\mathrm{corr}}\bigl(u_k^{l,(j)}\bigr),
    \qquad j=0,\ldots,K-1.
\end{aligned}
\]
The sample \(j=0\) is the current actor output. The sampled smooth
minimum and its weights are
\[
\begin{aligned}
    S_{\tau,k}^l(\theta_u)
    &:=-\tau\log\left[
    \frac1K\sum_{j=0}^{K-1}e^{-\mathcal H_k^{l,(j)}/\tau}\right],\\
    w_k^{l,(j)}
    &:=\frac{e^{-\mathcal H_k^{l,(j)}/\tau}}
    {\sum_{r=0}^{K-1}e^{-\mathcal H_k^{l,(r)}/\tau}},
    \qquad \tau>0.
\end{aligned}
\]
They satisfy
\[
    \min_j\mathcal H_k^{l,(j)}
    \leq S_{\tau,k}^l
    \leq\min_j\mathcal H_k^{l,(j)}+\tau\log K.
\]
The actor objective is
\[
    \mathcal L_{\mathrm{smin}}^l(\theta_u)
    :=\mathbb E_{\mathrm{batch}}\!\left[
    \sum_{k=0}^{N-1}\Delta t\,S_{\tau,k}^l(\theta_u)\right],
\]
where \(\mathbb E_{\mathrm{batch}}\) denotes the empirical average
over a finite batch of reference trajectories, each with its associated
perturbation set.
For sampled values that are \(C^1\) in \(\theta_u\), differentiation
with the reference samples, perturbations, \(\tau\), and
\(\sigma_{\mathrm{act}}\) held fixed gives
\begin{equation}
\label{eq:deepsmp-smin-actor-gradient}
    \nabla_{\theta_u}S_{\tau,k}^l
    =\sum_{j=0}^{K-1}w_k^{l,(j)}
      \nabla_{\theta_u}\mathcal H_k^{l,(j)}.
\end{equation}
To interpret the update direction, fix a reference sample and take
\(\sigma_{\mathrm{act}}>0\). Let \(\rho_r\) be the density of
\(\mathcal N(r,\sigma_{\mathrm{act}}^2I_{d_u})\), and define
\[
    \widetilde{\mathcal H}(z)
    :=\mathcal H_k^{l,\mathrm{corr}}\bigl(\Psi_U(z)\bigr).
\]
An idealized Gaussian population counterpart is
\[
\begin{aligned}
    S_\tau(r)
    &:=-\tau\log\int_{\mathbb R^{d_u}}
       e^{-\widetilde{\mathcal H}(z)/\tau}\rho_r(z)\,dz,\\
    \widetilde\pi_{\tau,r}(dz)
    &:=\frac{e^{-\widetilde{\mathcal H}(z)/\tau}\rho_r(z)}
    {\int_{\mathbb R^{d_u}}
       e^{-\widetilde{\mathcal H}(y)/\tau}\rho_r(y)\,dy}\,dz.
\end{aligned}
\]
Continuity of \(\mathcal H_k^{l,\mathrm{corr}}\) on compact \(U\)
makes \(\widetilde{\mathcal H}\) bounded. This boundedness and the
Gaussian tails justify differentiation under the integral sign.
Using \(\nabla_r\rho_r(z)
=\sigma_{\mathrm{act}}^{-2}(z-r)\rho_r(z)\) gives
\[
    -\nabla_r S_\tau(r)
    =\frac{\tau}{\sigma_{\mathrm{act}}^2}
      \left(\mathbb E_{\widetilde\pi_{\tau,r}}[z]-r\right).
\]
The exponential factor
\(e^{-\widetilde{\mathcal H}(z)/\tau}\) gives greater relative
weight to raw actions with lower corrected-Hamiltonian values.
Thus, the negative gradient pulls the current raw-action center
\(r\) toward the Gibbs-weighted mean, favoring raw actions whose
admissible images \(\Psi_U(z)\) have lower
corrected-Hamiltonian values.

In implementation, we use
\[
    \widetilde{\mathcal L}_H^l(\theta_u)
    :=\mathbb E_{\mathrm{batch}}\!\left[
    \sum_{k=0}^{N-1}\Delta t\sum_{j=0}^{K-1}
    \operatorname{sg}(w_k^{l,(j)})\mathcal H_k^{l,(j)}\right],
\]
where \(\operatorname{sg}\) preserves its argument in the forward
evaluation and blocks its derivative during backpropagation. By
\eqref{eq:deepsmp-smin-actor-gradient}, the resulting backpropagated
gradient equals \(\nabla_{\theta_u}\mathcal L_{\mathrm{smin}}^l\).
Starting from \(\bar\theta_u^l\), actor-gradient steps
\[
    \theta_u\leftarrow\theta_u
    -\alpha_u\nabla_{\theta_u}\mathcal L_{\mathrm{smin}}^l(\theta_u)
\]
produce candidate parameters \(\widehat\theta_u^{\,l+1}\). The
reference parameters and control are then updated by
\[
\begin{aligned}
    \bar\theta_u^{\,l+1}
    &:=(1-\eta_l)\bar\theta_u^l
       +\eta_l\widehat\theta_u^{\,l+1},\qquad\eta_l\in(0,1],\\
    \bar u_k^{l+1}(x)&:=u_{\bar\theta_u^{\,l+1},k}(x).
\end{aligned}
\]

The cost analysis of this parameter-space update is given in
Subsection~\ref{subsec:deepsmp-parameter-soft-update}.

We next establish consistency of the sampled objective and its
parameter gradient. Fix \(l,k\) and the values of the reference samples
\((X_k^l,p_k^l,q_k^l,P_k^l,\bar u_k^l)\), and use the same
perturbations for all actor parameters under comparison. Probabilities
are conditional on these reference data. Limits are taken along
deterministic sequences of \((K,\tau,\sigma_{\mathrm{act}})\), with
the reference data, network architecture, and \(d_u\) fixed.

\begin{theorem}[Consistency of the sampled smooth-minimum values]
\label{thm:deepsmp-smin-value-consistency}
Suppose that \(U\) is compact and Assumption~\ref{ass:basic-regularity}
holds. Assume also that \(\Psi_U:\mathbb R^{d_u}\to U\) is globally Lipschitz, and that the
raw actor output is continuous in \(\theta_u\). Let \(K\geq2\) and set
\[
    R_K:=\max_{0\leq j\leq K-1}|\zeta_k^{(j)}|.
\]
There exists a finite constant \(C\), depending on the frozen reference
data and \(\Psi_U\), but independent of
\(\theta_u,K,\tau,\sigma_{\mathrm{act}}\), such that, simultaneously
for every \(\theta_u\) and \(\tau>0\),
\begin{equation}
\label{eq:deepsmp-smin-value-pathwise}
\begin{aligned}
    -CR_K
    &\leq S_{\tau,k}^l(\theta_u)
       -\mathcal H_k^{l,\mathrm{corr}}
        \bigl(u_{\theta_u,k}(X_k^l)\bigr)\\
    &\leq\tau\log K.
\end{aligned}
\end{equation}
Moreover, for every \(\delta\in(0,1)\), with probability at least
\(1-\delta\),
\begin{equation}
\label{eq:deepsmp-smin-radius-probability}
    R_K\leq
    \sigma_{\mathrm{act}}
    \left(\sqrt{d_u}+\sqrt{2\log\frac{K-1}{\delta}}\right).
\end{equation}
Consequently, along any such sequence satisfying
\begin{equation}
\label{eq:deepsmp-smin-localization-schedule}
    \tau\log K\longrightarrow0,
    \qquad
    \sigma_{\mathrm{act}}\sqrt{\log K}\longrightarrow0,
\end{equation}
one has
\begin{equation}
\label{eq:deepsmp-smin-value-probability}
    \sup_{\theta_u}
    \left|S_{\tau,k}^l(\theta_u)
    -\mathcal H_k^{l,\mathrm{corr}}
      \bigl(u_{\theta_u,k}(X_k^l)\bigr)\right|
    \xrightarrow{\mathbb P}0.
\end{equation}
\end{theorem}

\begin{proof}
For the fixed reference data, the control Lipschitz bounds in
Assumption~\ref{ass:basic-regularity} imply that
\(H(t_k,X_k^l,p_k^l,q_k^l,\cdot)\) is Lipschitz on \(U\).
The diffusion increments in the second-order correction are bounded
and Lipschitz on the compact set \(U\). The inequality
\[
    |v^\top P_k^l v-w^\top P_k^l w|
    \leq\|P_k^l\|_{\mathrm{op}}(|v|+|w|)|v-w|
\]
therefore shows that \(\mathcal H_k^{l,\mathrm{corr}}\) is Lipschitz
on \(U\). Composing with the globally Lipschitz map \(\Psi_U\) gives
a constant \(C<\infty\) such that
\[
\begin{aligned}
    &\left|\mathcal H_k^{l,\mathrm{corr}}(\Psi_U(z))
      -\mathcal H_k^{l,\mathrm{corr}}(\Psi_U(y))\right|
      \leq C|z-y|,\qquad z,y\in\mathbb R^{d_u}.
\end{aligned}
\]
In particular, every sampled value differs from the unperturbed value
\(\mathcal H_k^{l,(0)}(\theta_u)\) by at most \(CR_K\), uniformly
in \(\theta_u\).

Since the unperturbed sample is included, the sampled hard-minimum
bound yields
\[
\begin{aligned}
    \mathcal H_k^{l,(0)}-CR_K
    &\leq\min_j\mathcal H_k^{l,(j)}
    \leq S_{\tau,k}^l\\
    &\leq\min_j\mathcal H_k^{l,(j)}+\tau\log K
    \leq\mathcal H_k^{l,(0)}+\tau\log K.
\end{aligned}
\]
This proves \eqref{eq:deepsmp-smin-value-pathwise}.
For a standard Gaussian vector \(\xi\in\mathbb R^{d_u}\), Gaussian
concentration gives
\[
    \mathbb P\bigl(|\xi|>\sqrt{d_u}+t\bigr)
    \leq e^{-t^2/2},\qquad t\geq0.
\]
Applying this inequality to each
\(\zeta_k^{(j)}/\sigma_{\mathrm{act}}\) and taking a union bound gives,
for \(\sigma_{\mathrm{act}}>0\),
\[
    \mathbb P\!\left(
    R_K>\sigma_{\mathrm{act}}(\sqrt{d_u}+t)\right)
    \leq(K-1)e^{-t^2/2}.
\]
Setting \(t=\sqrt{2\log((K-1)/\delta)}\) proves
\eqref{eq:deepsmp-smin-radius-probability}. The case
\(\sigma_{\mathrm{act}}=0\) is immediate.
Since \(d_u\) is fixed and \(K\geq2\), the perturbation limit in
\eqref{eq:deepsmp-smin-localization-schedule} makes this radius bound
tend to zero for every fixed \(\delta\in(0,1)\). The arbitrariness
of \(\delta\) gives \(R_K\to0\) in probability. Finally,
\eqref{eq:deepsmp-smin-value-pathwise} implies
\[
    \sup_{\theta_u}
    \left|S_{\tau,k}^l(\theta_u)
    -\mathcal H_k^{l,\mathrm{corr}}
      \bigl(u_{\theta_u,k}(X_k^l)\bigr)\right|
    \leq CR_K+\tau\log K.
\]
Both terms vanish under
\eqref{eq:deepsmp-smin-localization-schedule}, proving
\eqref{eq:deepsmp-smin-value-probability}.
\end{proof}

\begin{theorem}[Consistency of the actor-parameter gradients]
\label{thm:deepsmp-smin-gradient-consistency}
Retain the regularity hypotheses of
Theorem~\ref{thm:deepsmp-smin-value-consistency} and let \(K\geq2\).
Suppose that
\[
    z\longmapsto
    \mathcal H_k^{l,\mathrm{corr}}\bigl(\Psi_U(z)\bigr)
\]
is \(C^1\) on \(\mathbb R^{d_u}\) with globally Lipschitz gradient.
Fix \(R>0\), and assume that
\(\theta_u\mapsto r_{\theta_u,k}(X_k^l)\) belongs to \(C^1\) on
an open neighborhood of the closed parameter ball
\(\{\theta_u:|\theta_u|\leq R\}\).
There exists a finite constant \(C\), which may additionally depend
on \(R\) and the network, but is independent of
\(K,\tau,\sigma_{\mathrm{act}}\), such that
\begin{equation}
\label{eq:deepsmp-smin-gradient-pathwise}
\begin{aligned}
    \sup_{|\theta_u|\leq R}
    \Bigl|\nabla_{\theta_u}S_{\tau,k}^l(\theta_u)
    -\nabla_{\theta_u}\!\left[
       \mathcal H_k^{l,\mathrm{corr}}
       \bigl(u_{\theta_u,k}(X_k^l)\bigr)\right]\Bigr|
    \leq CR_K .
\end{aligned}
\end{equation}
If \(\sigma_{\mathrm{act}}\sqrt{\log K}\to0\), then,
for arbitrary positive temperatures along the sequence,
\begin{equation}
\label{eq:deepsmp-smin-gradient-probability}
\begin{aligned}
    \sup_{|\theta_u|\leq R}
    \Bigl|\nabla_{\theta_u}S_{\tau,k}^l(\theta_u)
    -\nabla_{\theta_u}\!\left[
       \mathcal H_k^{l,\mathrm{corr}}
       \bigl(u_{\theta_u,k}(X_k^l)\bigr)\right]\Bigr|
    \xrightarrow{\mathbb P}0.
\end{aligned}
\end{equation}
\end{theorem}

\begin{proof}
The \(C^1\) regularity of the actor on a neighborhood of the compact
parameter ball implies
\[
    \sup_{|\theta_u|\leq R}
    \left\|D_{\theta_u}r_{\theta_u,k}(X_k^l)\right\|_{\mathrm{op}}
    <\infty.
\]
Since the perturbations are additive and fixed during differentiation,
the chain rule gives
\[
\begin{aligned}
    \nabla_{\theta_u}\mathcal H_k^{l,(j)}(\theta_u)
    ={}&\bigl(D_{\theta_u}r_{\theta_u,k}(X_k^l)\bigr)^\top\\
    &\times\left.
    \nabla_z\left[\mathcal H_k^{l,\mathrm{corr}}(\Psi_U(z))\right]
    \right|_{z=r_{\theta_u,k}(X_k^l)+\zeta_k^{(j)}}.
\end{aligned}
\]
Using the Jacobian bound and the assumed Lipschitz continuity of the
lifted gradient, and subtracting the \(j=0\) formula, gives a finite
constant \(C\) such that
\[
    \sup_{|\theta_u|\leq R}
    \left|\nabla_{\theta_u}\mathcal H_k^{l,(j)}(\theta_u)
    -\nabla_{\theta_u}\mathcal H_k^{l,(0)}(\theta_u)\right|
    \leq C|\zeta_k^{(j)}|.
\]
The weights in \eqref{eq:deepsmp-smin-actor-gradient} are
nonnegative and sum to one. It follows that
\[
\begin{aligned}
    &\left|\nabla_{\theta_u}S_{\tau,k}^l
      -\nabla_{\theta_u}\mathcal H_k^{l,(0)}\right|\\
    &\qquad\leq\sum_{j=0}^{K-1}w_k^{l,(j)}
      \left|\nabla_{\theta_u}\mathcal H_k^{l,(j)}
            -\nabla_{\theta_u}\mathcal H_k^{l,(0)}\right|
      \leq CR_K,
\end{aligned}
\]
uniformly on the parameter ball and for all \(\tau>0\).
Since \(\mathcal H_k^{l,(0)}(\theta_u)
=\mathcal H_k^{l,\mathrm{corr}}
(u_{\theta_u,k}(X_k^l))\), this proves
\eqref{eq:deepsmp-smin-gradient-pathwise}.
Convergence follows from
\eqref{eq:deepsmp-smin-radius-probability} and
\(\sigma_{\mathrm{act}}\sqrt{\log K}\to0\).
\end{proof}

For a fixed finite batch and time grid, the corresponding value and
gradient convergence statements hold uniformly on \(|\theta_u|\leq R\), under their
respective sampling limits, for \(\mathcal L_{\mathrm{smin}}^l\) and
its unperturbed counterpart
\[
    \mathbb E_{\mathrm{batch}}\!\left[
    \sum_{k=0}^{N-1}\Delta t\,
    \mathcal H_k^{l,(0)}(\theta_u)\right].
\]
Indeed, a finite union bound controls all perturbation radii; averaging
the pathwise bounds with the largest of their finitely many constants
and using \(\sum_k\Delta t=T\) proves the assertions. Differentiation
commutes with the finite empirical average. For random reference data held fixed along the sequence and satisfying
these hypotheses almost surely, integrating the conditional error
probabilities and applying bounded convergence gives the same convergence
under the joint law of reference data and perturbations.

\subsection{DeepSMP and parameter-space soft update analysis}
\label{subsec:deepsmp-parameter-soft-update}

DeepSMP jointly trains the first- and second-order adjoint networks and
the control networks within a single iterative procedure. This differs
from classical MSA and many of its modified numerical implementations,
which typically employ a nested inner--outer structure: for each outer
control update, the state and adjoint systems are first solved, or the
corresponding inner problem is optimized, to a prescribed accuracy.
Instead, DeepSMP updates the adjoint-related and control parameters
simultaneously, while separating their coupled evolution through a
two-time-scale learning rule, with faster updates for the adjoint-related
parameters and slower updates for the control parameters. The reference
control parameters are further stabilized through parameter-space soft
interpolation. This simultaneous scheme avoids repeatedly solving the
inner problem to high accuracy before each control update, thereby
reducing the computational burden of nested iterations, particularly in
high-dimensional settings. The complete procedure is summarized in
Algorithm~\ref{alg:deepsmp}.

\begin{algorithm}[tp]
\caption{Deep SMP Policy Iteration with Joint Two-Time-Scale Training}
\label{alg:deepsmp}
\small
\begin{algorithmic}[1]

\Require Initial state $X_0 = x_0$, terminal time $T$, time intervals $N$, step size $\Delta t = T/N$; batch size $M$, total iterations $K$, candidate size $J$, temperature $\tau$, learning rates $\alpha_{\text{adj}}$ and $\alpha_u$ ($\alpha_{\text{adj}} > \alpha_u$), loss weights $\lambda_1, \lambda_2, \lambda_H$, soft-update rate $\eta\in(0,1]$.

\State \textbf{Initialize Adjoint System:} Initial adjoints $p_0, P_0$ and point-wise networks $\theta_q = \{\theta_q^i\}_{i=0}^{N-1}$, $\theta_Q = \{\theta_Q^i\}_{i=0}^{N-1}$.
\State \textbf{Initialize Online Policy:} Point-wise online policy networks $\theta_u = \{\theta_u^i\}_{i=0}^{N-1}$.
\State \textbf{Initialize Anchor Policy:} Point-wise anchor policy networks $\bar{\theta}_u \leftarrow \theta_u$.

\For{$k = 1$ \textbf{to} $K$}
    \State Sample a mini-batch of $M$ independent Brownian increment paths $\{\Delta W_i^{(m)}\}_{i=0}^{N-1}$.
    \State Set $X_0^{(m)} = x_0$, $p_0^{(m)} = p_0$, $P_0^{(m)} = P_0$ for all $m = 1, \dots, M$.

    \For{$i = 0$ \textbf{to} $N - 1$}
        \Statex \hfill $\triangleright$ \textit{1. Network Forward Passes}
        \State $r_i = \text{Net}_u^i(X_i \mid \theta_u^i), \quad u_i=\Psi_U(r_i), \quad \bar r_i = \text{Net}_u^i(X_i \mid \bar{\theta}_u^i), \quad \bar u_i=\Psi_U(\bar r_i), \quad q_i = \text{Net}_q^i(X_i \mid \theta_q^i), \quad Q_i = \text{Net}_Q^i(X_i \mid \theta_Q^i)$.

        \Statex \hfill $\triangleright$ \textit{2. Gradient Detachment for Physical Isolation}
        \State $\bar{u}_i \leftarrow \text{stop\_gradient}(\bar{u}_i), \quad \bar{p}_i \leftarrow \text{stop\_gradient}(p_i), \quad \bar{q}_i \leftarrow \text{stop\_gradient}(q_i), \quad \bar{P}_i \leftarrow \text{stop\_gradient}(P_i)$.

        \Statex \hfill $\triangleright$ \textit{3. Randomized Action Candidates \& Corrected Hamiltonian}
        \State Generate action candidates $a_i^{(j)} = \Psi_U(r_i + \sigma_{\text{act}}\zeta_j)$ with $\zeta_0=0$ and $\zeta_j \sim \mathcal{N}(0, I)$ for $j=1,\dots,J-1$ ($a_i^{(0)} = u_i$).
        \State $\mathcal{H}_i(a_i^{(j)}) = \bar{p}_i^\top b(t_i, X_i, a_i^{(j)}) + \operatorname{Tr}\!\left[\bar{q}_i^\top \sigma(t_i, X_i, a_i^{(j)})\right] + f(t_i, X_i, a_i^{(j)})$
        \Statex \hfill $+ \frac{1}{2}\operatorname{Tr}\!\left(\left[\sigma(t_i, X_i, a_i^{(j)}) - \sigma(t_i, X_i, \bar{u}_i)\right]^\top \bar{P}_i \left[\sigma(t_i, X_i, a_i^{(j)}) - \sigma(t_i, X_i, \bar{u}_i)\right]\right)$.
        \State $w_i^{(j)} = \frac{\exp(-\mathcal{H}_i(a_i^{(j)}) / \tau)}{\sum_{r=0}^{J-1} \exp(-\mathcal{H}_i(a_i^{(r)}) / \tau)}, \qquad \bar{w}_i^{(j)} \leftarrow \text{stop\_gradient}(w_i^{(j)})$.

        \Statex \hfill $\triangleright$ \textit{4. Forward and Adjoint Advancements Driven by Anchor Policy $\bar{u}_i$}
        \State $X_{i+1} = X_i + b(t_i, X_i, \bar{u}_i)\Delta t + \sigma(t_i, X_i, \bar{u}_i)\Delta W_i$.
        \State $p_{i+1} = p_i - \nabla_x H(t_i, X_i, p_i, q_i, \bar{u}_i)\Delta t + q_i \Delta W_i$.
        \State $P_{i+1} = P_i - \mathcal F(t_i, X_i, p_i, q_i, P_i, Q_i, \bar{u}_i)\Delta t + Q_i \Delta W_i$.
    \EndFor

    \Statex \hfill $\triangleright$ \textit{5. Loss Aggregation}
    \State $\mathcal{L}_{\text{BSDE}}
    = \frac{1}{M}\sum_{m=1}^M
    \left[
        \lambda_1
        |p_N^{(m)}-\Phi_x(X_N^{(m)})|^2
        +
        \lambda_2
        |P_N^{(m)}-\Phi_{xx}(X_N^{(m)})|^2
    \right]$.
    \State $\mathcal{L}_H = \frac{\Delta t}{M}\sum_{m=1}^M \sum_{i=0}^{N-1} \sum_{j=0}^{J-1} \bar{w}_{i,m}^{(j)} \mathcal{H}_{i,m}(a_{i,m}^{(j)})$.
    \State $\mathcal{L}_{\text{Total}} = \mathcal{L}_{\text{BSDE}} + \lambda_H \mathcal{L}_H$.

    \Statex \hfill $\triangleright$ \textit{6. Two-Time-Scale Update Rule \& Anchor Policy Soft Update}
    \State $\theta_{\text{adj}} \leftarrow \theta_{\text{adj}} - \alpha_{\text{adj}} \nabla_{\theta_{\text{adj}}} \mathcal{L}_{\text{Total}}, \qquad \text{where } \theta_{\text{adj}} = (p_0, P_0, \theta_q, \theta_Q)$.
    \State $\theta_u \leftarrow \theta_u - \alpha_u \nabla_{\theta_u} \mathcal{L}_{\text{Total}}$.
    \State $\bar{\theta}_u \leftarrow (1-\eta)\bar{\theta}_u + \eta\theta_u$. \hfill $\triangleright$ Parameter-space soft update
\EndFor

\Ensure Optimized anchor policy networks $\{\text{Net}_u^i(\cdot \mid \bar{\theta}_u^i)\}_{i=0}^{N-1}$ and initial adjoints $p_0, P_0$.

\end{algorithmic}
\end{algorithm}

We now turn to the theoretical role of the parameter-space soft update in
DeepSMP. In the process-level analysis of
Section~\ref{sec:su-msa}, the soft update is performed directly
in the admissible control space. In the neural-network implementation,
however, a literal process-level update is not computationally attractive.
Indeed, if one updates the feedback control itself by
\[
    u_{k}^{n+1}(\cdot)
    =
    (1-\eta_n)u_k^n(\cdot)
    +
    \eta_n v_k^n(\cdot),
    \qquad
    k=0,\ldots,N-1,
\]
with \(u_k^n=\pi_{\theta_{u,k}^n}\) and
\(v_k^n=\pi_{\theta_{u,k}^{\mathrm c,n}}\), then repeated substitution
leads to an ensemble-type representation
\[
    u_k^{n+1}(\cdot)
    =
    \omega_{0,n}\pi_{\theta_{u,k}^{0}}(\cdot)
    +
    \sum_{r=0}^{n}
    \omega_{r,n}\pi_{\theta_{u,k}^{\mathrm c,r}}(\cdot),
    \qquad
    k=0,\ldots,N-1,
\]
for suitable weights \(\omega_{r,n}\) determined by the previous
soft-update stepsizes. Hence, an exact evaluation of \(u_k^{n+1}\) would
require storing and evaluating all historical control networks at each
time level \(t_k\). This cost increases with the number of iterations, the
number of time intervals, and the dimension of the underlying state
space.

Moreover, applying the soft update only to the realized control outputs
along simulated trajectories does not define a feedback map on the whole
state space. If a new Monte Carlo batch visits states that did not appear
in the previous trajectories, the updated control value at those states is
not specified. By contrast, a parameterized control network represents a
feedback map
\[
    x\longmapsto \pi_{\theta_{u,k}}(t_k,x)
\]
on the state space, and this parametric representation is precisely what
allows the learned control to generalize across newly simulated
trajectories.

For these reasons, DeepSMP performs the soft update in parameter space:
\[
    \bar\theta_{u,k}^{\,n+1}
    =
    (1-\eta_n)\bar\theta_{u,k}^{\,n}
    +
    \eta_n\theta_{u,k}^{n+1},
    \qquad
    k=0,\ldots,N-1.
\]
Since a convex interpolation of neural-network parameters does not in
general coincide with a pointwise convex interpolation of the corresponding
feedback controls, the process-level descent estimate cannot be applied
verbatim. We therefore establish a parameter-space counterpart of the
corrected-Hamiltonian upper bound. The analysis is carried out in an
exact-adjoint oracle setting: for a given reference control, the first- and
second-order adjoint processes are assumed to be the exact solutions of
their corresponding adjoint equations. Under this idealized setting, the
result below shows that the cost variation along a parameter interpolation
path is controlled by the variation of the corrected Hamiltonian, up to a
second-order remainder measured in a weighted discrete \(L^2\)-norm of the
parameter displacement. In the practical DeepSMP algorithm, these adjoints
are replaced by learned neural-network approximations, whose effect can be
viewed as an additional inexactness term.

Let \(0=t_0<t_1<\cdots<t_N=T\) be the time grid, with
\(\Delta t_k=t_{k+1}-t_k\). Since DeepSMP employs one control network at
each time level, denote the reference and candidate parameter collections
by
\[
\begin{aligned}
    \bar\Theta_u
    &:=
    (\bar\theta_{u,0},\ldots,\bar\theta_{u,N-1}),&
    \Theta_u^{\mathrm c}
    &:=
    (\theta_{u,0}^{\mathrm c},\ldots,\theta_{u,N-1}^{\mathrm c}),\\
    \delta\Theta
    &:=(\delta\Theta_0,\ldots,\delta\Theta_{N-1}),\\
    \delta\Theta_k
    &:=
    \theta_{u,k}^{\mathrm c}-\bar\theta_{u,k},
    &&k=0,\ldots,N-1.
\end{aligned}
\]
Let \(\bar\theta_u(\cdot)\) and \(\delta\Theta(\cdot)\) be their piecewise constant
interpolations, namely,
\[
\begin{aligned}
    \bar\theta_u(t)&:=\bar\theta_{u,k},&
    \delta\Theta(t)&:=\delta\Theta_k,
    &&t\in[t_k,t_{k+1}),\\
    |\delta\Theta|_\Delta^2
    &:=\sum_{k=0}^{N-1}\Delta t_k|\delta\Theta_k|^2
    =\int_0^T|\delta\Theta(t)|^2\,dt,&
    \theta_u^\eta(t)
    &:=\bar\theta_u(t)+\eta \delta\Theta(t),
    &&0\le\eta\le1.
\end{aligned}
\]
For the population analysis below, the reference and candidate
parameter collections are fixed, and expectations are taken
with respect to the Brownian motion driving the state equation.
The family \(\theta_u^\eta\) defines the parameter-space soft-update path
joining the reference and candidate control-network parameters.
For any piecewise constant parameter path
\(\vartheta(t)=\vartheta_k\) on \([t_k,t_{k+1})\), we define the
corresponding continuous-time feedback law by
\[
    \pi_{\vartheta(t)}(t,x):=\Pi_{\vartheta_k}(x),
    \qquad
    t\in[t_k,t_{k+1}),\quad k=0,\ldots,N-1.
\]

\begin{assumption}[Neural-network architecture and bounded parameter paths]
\label{ass:nn-architecture}
For each time level \(t_k\), \(k=0,\ldots,N-1\), the feedback control is
parameterized by a feedforward network
\(\pi_{\vartheta}(t_k,x)=\Pi_{\vartheta}(x)\), where
\(\Pi_{\vartheta}:\mathbb R^{d_x}\to\mathbb R^{d_u}\). The following
conditions hold.

\textup{(NN.1)} The network architecture is fixed. In particular, the
number of hidden layers \(L\), the hidden-layer widths
\(n_1,\ldots,n_L\), the raw output dimension \(d_z\), and the control
dimension \(d_u\) are independent of the iteration index \(n\) and the
time index \(k\). Writing \(n_0=d_x\), the network takes the form
\[
\begin{aligned}
    h_0&=x,\\
    z_\ell&=W_\ell h_{\ell-1}+b_\ell,
    & h_\ell&=\varphi_\ell(z_\ell),
    &&\ell=1,\ldots,L,\\
    z_{L+1}&=W_{L+1}h_L+b_{L+1},
    &\Pi_{\vartheta}(x)&=\psi(z_{L+1}),
\end{aligned}
\]
where
\(\vartheta=(W_1,b_1,\ldots,W_{L+1},b_{L+1})\),
\(W_\ell\in\mathbb R^{n_\ell\times n_{\ell-1}}\),
\(b_\ell\in\mathbb R^{n_\ell}\) for \(\ell=1,\ldots,L\), and
\(W_{L+1}\in\mathbb R^{d_z\times n_L}\),
\(b_{L+1}\in\mathbb R^{d_z}\).

\textup{(NN.2)} For each \(\ell=1,\ldots,L\), the hidden-layer
activation map
\(\varphi_\ell:\mathbb R^{n_\ell}\to\mathbb R^{n_\ell}\) belongs to
\(C^1\), and there exists \(M_\ell<\infty\) such that
\[
    \sup_{z\in\mathbb R^{n_\ell}}
    \|D\varphi_\ell(z)\|_{\mathrm{op}}
    \le M_\ell,
    \qquad \ell=1,\ldots,L.
\]

\textup{(NN.3)} The output map
\(\psi:\mathbb R^{d_z}\to\mathbb R^{d_u}\) belongs to \(C^1\), has a
globally bounded Jacobian, and enforces the control constraint; namely,
for some \(M_{\mathrm{out}}<\infty\),
\[
    \sup_{z\in\mathbb R^{d_z}}
    \|D\psi(z)\|_{\mathrm{op}}
    \le M_{\mathrm{out}},
    \qquad
    \psi(\mathbb R^{d_z})\subseteq U.
\]

\textup{(NN.4)} The reference and candidate control-network parameters
generated by the algorithm are uniformly bounded. More precisely, there
exists a constant \(R>0\), independent of the iteration index \(n\) and
the time index \(k\), such that
\[
    \sup_{n\geq0}\max_{0\leq k\leq N-1}
    \left(
        |\bar\vartheta_{n,k}|
        +
        |\vartheta_{n,k}^{\mathrm c}|
    \right)
    \leq R,
\]
where the parameter norm is defined by
\[
    |\vartheta|^2
    :=
    \sum_{\ell=1}^{L+1}
    \left(
        \|W_\ell\|_{\mathrm F}^2+|b_\ell|^2
    \right).
\]
\end{assumption}

\begin{remark}[Scope of Assumption~\ref{ass:nn-architecture}]
\label{rem:scope-nn-assumption}
Assumption~\ref{ass:nn-architecture} covers a broad class of standard
feedforward networks. In particular, condition \textup{(NN.2)} is
satisfied by commonly used smooth activations with globally bounded
derivatives, including Swish/SiLU, \(\tanh\), sigmoid, softplus, and ELU
with its standard shape parameter equal to one,
while condition \textup{(NN.3)} is satisfied by standard smooth
squashing maps, such as affine transformations of \(\tanh\), that enforce
the control constraint.

The uniform boundedness condition in \textup{(NN.4)} is a standard
stability requirement in neural-network and stochastic-approximation
analysis; related bounded-iterate, compact-parameter, or projected-update
conditions are used, for example, in
\cite{heusel2017gans,senencerda2025dropout,
oymak2018compact,leimkuhler2021constrained}.
It can be enforced through compact parameter constraints, projected
updates, or explicit weight and bias clipping. Moreover, \textup{(NN.4)}
automatically controls every parameter-space soft-update path by convexity
of the parameter norm. Hence, it is a conventional and algorithmically
enforceable stability condition rather than a restrictive assumption on
the network architecture.
\end{remark}

For \(n\geq0\), \(k=0,\ldots,N-1\), and \(\eta\in[0,1]\), denote
\[
    \vartheta_{n,k}^{\eta}
    :=
    (1-\eta)\bar\vartheta_{n,k}
    +\eta\vartheta_{n,k}^{\mathrm c}.
\]
By Assumption~\ref{ass:nn-architecture}\textup{(NN.4)},
\[
    |\vartheta_{n,k}^{\eta}|
    \leq
    (1-\eta)|\bar\vartheta_{n,k}|
    +\eta|\vartheta_{n,k}^{\mathrm c}|
    \leq R.
\]

\begin{proposition}[Derivative bounds for fixed feedforward control networks]
\label{prop:network-derivative-bounds}
Suppose that Assumption~\ref{ass:nn-architecture} holds. Then there exist
constants \(K_x,K_\vartheta>0\), depending only on the fixed network
architecture, the activation-derivative bounds
\(\{M_\ell\}_{\ell=1}^{L}\), the output-map bound
\(M_{\mathrm{out}}\), the uniform parameter bound \(R\), and the values
\(\{|\varphi_\ell(0)|\}_{\ell=1}^{L}\), but independent of the iteration
index \(n\), the time index \(k\), the interpolation parameter
\(\eta\in[0,1]\), and the state \(x\), such that, for every
\(n\geq0\), \(k=0,\ldots,N-1\), \(\eta\in[0,1]\), and
\(x\in\mathbb R^{d_x}\),
\[
\begin{aligned}
    \bigl\|
        D_x\Pi_{\vartheta_{n,k}^{\eta}}(x)
    \bigr\|_{\mathrm{op}}
    &\leq K_x,
    &
    \bigl\|
        D_\vartheta\Pi_{\vartheta_{n,k}^{\eta}}(x)
    \bigr\|_{\mathrm{op}}
    &\leq K_\vartheta(1+|x|).
\end{aligned}
\]

Here \(D_\vartheta\Pi_\vartheta(x)\) denotes the Fr\'echet derivative of
the network output with respect to the full parameter vector
\(\vartheta=(W_1,b_1,\ldots,W_{L+1},b_{L+1})\), regarded as a linear
operator from the product parameter space into \(\mathbb R^{d_u}\).
For a parameter perturbation
\(\delta\vartheta
=(\delta W_1,\delta b_1,\ldots,
\delta W_{L+1},\delta b_{L+1})\), the corresponding operator and parameter
norms are defined by
\[
\begin{aligned}
    \bigl\|
        D_\vartheta\Pi_\vartheta(x)
    \bigr\|_{\mathrm{op}}
    &:=
    \sup_{\delta\vartheta\neq0}
    \frac{
        \bigl|
            D_\vartheta\Pi_\vartheta(x)
            [\delta\vartheta]
        \bigr|
    }{
        |\delta\vartheta|
    },
    &
    |\delta\vartheta|^2
    &:=
    \sum_{\ell=1}^{L+1}
    \left(
        \|\delta W_\ell\|_{\mathrm F}^{2}
        +
        |\delta b_\ell|^{2}
    \right).
\end{aligned}
\]
\end{proposition}
\begin{proof}
Assumption~\ref{ass:nn-architecture}\textup{(NN.4)} implies, uniformly
in \(n\), \(k\), and \(\eta\),
\[
    \|W_\ell\|_{\mathrm{op}}
    \leq\|W_\ell\|_{\mathrm F}\leq R,
    \qquad
    |b_\ell|\leq R,
    \qquad \ell=1,\ldots,L+1.
\]
Moreover, Assumption~\ref{ass:nn-architecture}\textup{(NN.2)} and the
mean value theorem give
\[
    |\varphi_\ell(z)|
    \leq
    |\varphi_\ell(0)|+M_\ell|z|,
    \qquad \ell=1,\ldots,L.
\]
It follows inductively from the forward propagation equations that there
exist constants \(A_\ell>0\), depending only on the quantities specified
in the statement, such that
\[
    |h_\ell|\leq A_\ell(1+|x|),
    \qquad \ell=1,\ldots,L.
\]

The chain rule yields
\[
    D_x\Pi_\vartheta(x)
    =
    D\psi(z_{L+1})W_{L+1}
    D\varphi_L(z_L)W_L\cdots
    D\varphi_1(z_1)W_1.
\]
Consequently,
\[
    \|D_x\Pi_\vartheta(x)\|_{\mathrm{op}}
    \leq
    M_{\mathrm{out}}R^{L+1}
    \prod_{\ell=1}^L M_\ell
    =:K_x.
\]

It remains to estimate the derivative with respect to the parameters.
For a perturbation
\(\delta\vartheta
=(\delta W_1,\delta b_1,\ldots,\delta W_{L+1},\delta b_{L+1})\),
denote the corresponding first-order variations of the hidden variables
by \(\delta h_\ell\) and \(\delta z_\ell\). They satisfy
\[
\begin{aligned}
    \delta z_\ell
    &=
    \delta W_\ell h_{\ell-1}
    +W_\ell\delta h_{\ell-1}
    +\delta b_\ell,
    \qquad
    \delta h_\ell
    =
    D\varphi_\ell(z_\ell)\delta z_\ell,
    \quad \ell=1,\ldots,L,
\end{aligned}
\]
with \(\delta h_0=0\), and
\[
    D_\vartheta\Pi_\vartheta(x)[\delta\vartheta]
    =
    D\psi(z_{L+1})
    \bigl(
        \delta W_{L+1}h_L
        +W_{L+1}\delta h_L
        +\delta b_{L+1}
    \bigr).
\]
Since
\[
    \|\delta W_\ell\|_{\mathrm{op}}
    \leq\|\delta W_\ell\|_{\mathrm F}
    \leq|\delta\vartheta|,
    \qquad
    |\delta b_\ell|\leq|\delta\vartheta|,
\]
the hidden-state bounds above and induction over the layers give
\[
    |\delta h_\ell|
    \leq B_\ell(1+|x|)|\delta\vartheta|,
    \qquad \ell=1,\ldots,L,
\]
for constants \(B_\ell\) depending only on the fixed architecture and
the bounds in the statement. Hence
\[
    \bigl|
        D_\vartheta\Pi_\vartheta(x)[\delta\vartheta]
    \bigr|
    \leq
    K_\vartheta(1+|x|)|\delta\vartheta|.
\]
Taking the supremum over all nonzero \(\delta\vartheta\) proves
\[
    \|D_\vartheta\Pi_\vartheta(x)\|_{\mathrm{op}}
    \leq K_\vartheta(1+|x|).
\]
Since the parameter bound is uniform along all interpolation paths, the
constants \(K_x\) and \(K_\vartheta\) are independent of \(n\), \(k\),
\(\eta\), and \(x\).
\end{proof}

By Proposition~\ref{prop:network-derivative-bounds}, the uniform state-
and parameter-derivative estimates required below follow from
Assumption~\ref{ass:nn-architecture}.

\begin{theorem}[Cost upper bound under parameter-space soft update]
\label{thm:parameter-space-soft-update-upper-bound}
Suppose that Assumptions~\ref{ass:basic-regularity} and
\ref{ass:nn-architecture} hold. Fix an iteration index and suppress it
from the notation. Let
\[
    u_t:=\pi_{\bar\theta_u(t)}(t,X_t),
    \qquad
    u_t^\eta:=\pi_{\theta_u^\eta(t)}(t,X_t^\eta),
    \qquad 0\leq\eta\leq1,
\]
where \(X\) and \(X^\eta\) are the corresponding closed-loop state
processes. Let \((p,q)\) and \((P,Q)\) be the exact first- and
second-order adjoint processes associated with the reference pair
\((X,u)\). Along the reference trajectory, denote the corrected Hamiltonian by
\(\mathcal H_t^u\), as in \eqref{eq:Hu-density}.
Then there exists a constant \(C>0\), independent of the interpolation
parameter \(\eta\), the parameter displacement \(\delta\Theta\), and the iteration
index, such that, for every \(\eta\in[0,1]\),
\[
\begin{aligned}
    J\!\left(\pi_{\bar\Theta_u+\eta\delta\Theta}\right)
    -
    J\!\left(\pi_{\bar\Theta_u}\right)
    \leq
    \mathbb E\int_0^T
    \left[
        \mathcal H_t^u(u_t^\eta)
        -
        \mathcal H_t^u(u_t)
    \right]dt
    +
    C\eta^2|\delta\Theta|_\Delta^2 .
\end{aligned}
\]
\end{theorem}
\begin{proof}
Set
\[
    \delta X_t:=X_t^\eta-X_t,
    \qquad
    \delta u_t:=u_t^\eta-u_t,
    \qquad
    s_{\eta,\delta\Theta}:=\eta|\delta\Theta|_\Delta.
\]
By Proposition~\ref{prop:network-derivative-bounds} and the mean value
theorem,
\[
    |\delta u_t|
    \leq
    K_x|\delta X_t|
    +K_\vartheta\eta(1+|X_t|)|\delta\Theta(t)|.
\]
Combining this with the model Lipschitz bounds, BDG and H\"older give,
for \(p=2,4\),
\[
\begin{aligned}
 \mathbb E\sup_{r\leq t}|\delta X_r|^p
 &\leq C_p\int_0^t\mathbb E\sup_{r\leq s}|\delta X_r|^p\,ds\\
 &\quad+C_p\eta^p\mathbb E\left[
    (1+\sup_{s\leq T}|X_s|)^p
    \left(\int_0^T|\delta\Theta(s)|^2\,ds\right)^{p/2}\right].
\end{aligned}
\]
The parameter direction is deterministic in this population estimate;
Lemma~\ref{lem:state-moment-estimate} therefore bounds the second line by
\(C_p\eta^p|\delta\Theta|_\Delta^p\). Gronwall's inequality and the
preceding bound on \(\delta u\) yield
\[
    \mathbb E\left[\sup_{0\leq t\leq T}|\delta X_t|^p\right]
    +
    \mathbb E\left[\left(\int_0^T|\delta u_t|^2\,dt\right)^{p/2}\right]
    \leq C_p s_{\eta,\delta\Theta}^p.
\]
Moreover, Assumption~\ref{ass:nn-architecture}\textup{(NN.4)} implies
that \(|\delta\Theta|_\Delta\) is uniformly bounded. Hence \(s_{\eta,\delta\Theta}\leq C\)
uniformly in the iteration index, and consequently
\(s_{\eta,\delta\Theta}^3+s_{\eta,\delta\Theta}^4\leq Cs_{\eta,\delta\Theta}^2\).

Define the coefficient increments along the reference state by
\[
\begin{aligned}
    \Delta b_t&:=b(t,X_t,u_t^\eta)-b(t,X_t,u_t),\\
    \Delta\sigma_t^i&:=\sigma^i(t,X_t,u_t^\eta)
    -\sigma^i(t,X_t,u_t),\qquad i=1,\ldots,d,\\
    \Delta f_t&:=f(t,X_t,u_t^\eta)-f(t,X_t,u_t).
\end{aligned}
\]
Then, for \(p=2,4\),
\[
    \mathbb E\left[
        \left(
            \int_0^T
            \left(
                |\Delta b_t|^2
                +\sum_{i=1}^d|\Delta\sigma_t^i|^2
            \right)dt
        \right)^{p/2}
    \right]
    \leq C_p s_{\eta,\delta\Theta}^p.
\]

Let
\[
    A_t:=b_x(t,X_t,u_t),
    \qquad
    C_t^i:=\sigma_x^i(t,X_t,u_t),
\]
and let \(\xi\) solve
\[
\begin{cases}
    d\xi_t
    =
    (A_t\xi_t+\Delta b_t)\,dt
    +
    \displaystyle\sum_{i=1}^d
    (C_t^i\xi_t+\Delta\sigma_t^i)\,dW_t^i,\\
    \xi_0=0.
\end{cases}
\]
The linear SDE estimate gives
\[
    \mathbb E\left[\sup_{0\leq t\leq T}|\xi_t|^p\right]
    \leq C_p s_{\eta,\delta\Theta}^p,
    \qquad p=2,4.
\]
Furthermore, \(r:=\delta X-\xi\) satisfies a linear SDE whose forcing
terms are bounded by
\[
    C\bigl(|\delta X_t|^2+|\delta u_t|\,|\delta X_t|\bigr).
\]
For the mixed forcing term, Cauchy--Schwarz gives
\[
 \mathbb E\int_0^T|\delta u_t|^2|\delta X_t|^2\,dt
 \leq \|\delta u\|_{\mathbb H^4}^2\|\delta X\|_{\mathcal S^4}^2
 \leq Cs_{\eta,\delta\Theta}^4.
\]
The pure state term is bounded by
\(T\|\delta X\|_{\mathcal S^4}^4\).
The \(L^2\) estimate for the linear remainder equation consequently gives
\[
    \mathbb E\left[\sup_{0\leq t\leq T}|r_t|^2\right]
    \leq C s_{\eta,\delta\Theta}^4.
\]

Taylor expansion with respect to the state variable, together with
\(\delta X=\xi+r\), now gives
\[
\begin{aligned}
    J\!\left(\pi_{\bar\Theta_u+\eta\delta\Theta}\right)
    -J\!\left(\pi_{\bar\Theta_u}\right)
    &=
    \mathbb E\!\left[\Phi_x(X_T)^\top\xi_T\right]
    +\mathbb E\int_0^T
    \left[f_x(t,X_t,u_t)^\top\xi_t+\Delta f_t\right]dt\\
    &\quad
    +\frac12\mathbb E\!\left[
        \xi_T^\top\Phi_{xx}(X_T)\xi_T
    \right]
    +\frac12\mathbb E\int_0^T
        \xi_t^\top f_{xx}(t,X_t,u_t)\xi_t\,dt
    +R_0^{\eta,\delta\Theta},
\end{aligned}
\]
where
\[
\begin{aligned}
    R_0^{\eta,\delta\Theta}
    ={}&
    \mathbb E\Big[
        \Phi(X_T^\eta)-\Phi(X_T)
        -\Phi_x(X_T)^\top\xi_T
        -\frac12\xi_T^\top\Phi_{xx}(X_T)\xi_T
    \Big]\\
    &+\mathbb E\int_0^T\Big[
        f(t,X_t^\eta,u_t^\eta)-f(t,X_t,u_t^\eta)
        -f_x(t,X_t,u_t)^\top\xi_t\\
    &\hspace{38mm}
        -\frac12\xi_t^\top f_{xx}(t,X_t,u_t)\xi_t
    \Big]dt
\end{aligned}
\]
satisfies
\[
    |R_0^{\eta,\delta\Theta}|
    \leq Cs_{\eta,\delta\Theta}^2.
\]
The terms linear in \(r\), the replacement of \(\delta X\) by \(\xi\)
in the quadratic forms, and the state Taylor remainders are bounded by
\[
\begin{aligned}
    C\Bigl(
        \|r\|_{\mathcal S^2}
        +\|\xi\|_{\mathcal S^2}\|r\|_{\mathcal S^2}
        +\|r\|_{\mathcal S^2}^2
        +\|\delta X\|_{\mathcal S^4}^3
    \Bigr)
    &\leq
    C\bigl(
        s_{\eta,\delta\Theta}^2
        +s_{\eta,\delta\Theta}^3
        +s_{\eta,\delta\Theta}^4
    \bigr)\\
    &\leq Cs_{\eta,\delta\Theta}^2.
\end{aligned}
\]
Here the linear terms use the linear-growth bounds on \(\Phi_x\)
and \(f_x\), together with the uniform state moment estimate.

For the change of \(f_x\) in the control variable, timewise
Cauchy--Schwarz and H\"older's inequality with exponents \(4,4,2\)
give
\[
\begin{aligned}
    &\mathbb E\int_0^T
    \bigl|
        f_x(t,X_t,u_t^\eta)-f_x(t,X_t,u_t)
    \bigr|\,|\delta X_t|\,dt\\
    &\qquad\leq
    C\left\|
        1+\sup_{0\leq t\leq T}|X_t|
    \right\|_{L^4}
    \|\delta u\|_{\mathbb H^4}
    \|\delta X\|_{\mathcal S^2}
    \leq Cs_{\eta,\delta\Theta}^2.
\end{aligned}
\]
The corresponding change of \(f_{xx}\) contributes at most
\[
    2L\mathbb E\int_0^T|\delta X_t|^2\,dt
    \leq Cs_{\eta,\delta\Theta}^2.
\]
These estimates establish the stated bound on
\(R_0^{\eta,\delta\Theta}\).

Applying It\^o's formula to \(p_t^\top\xi_t\) and using the
first-order adjoint equation gives
\[
\begin{aligned}
    d(p_t^\top\xi_t)
    &=
    \left[
        -f_x(t,X_t,u_t)^\top\xi_t
        +p_t^\top\Delta b_t
        +\sum_{i=1}^d(q_t^i)^\top\Delta\sigma_t^i
    \right]dt\\
    &\quad+
    \sum_{i=1}^d
    \left[
        (q_t^i)^\top\xi_t
        +p_t^\top(C_t^i\xi_t+\Delta\sigma_t^i)
    \right]dW_t^i.
\end{aligned}
\]
The stochastic integrand satisfies
\[
\begin{aligned}
    &\mathbb E\left[
        \left(
            \int_0^T\sum_{i=1}^d
            \left|
                (q_t^i)^\top\xi_t
                +p_t^\top(C_t^i\xi_t+\Delta\sigma_t^i)
            \right|^2dt
        \right)^{1/2}
    \right]\\
    &\qquad\leq
    C\bigl(
        \|q\|_{\mathbb H^2}+\|p\|_{\mathcal S^2}
    \bigr)\|\xi\|_{\mathcal S^2}\\
    &\qquad\quad+
    C\|p\|_{\mathcal S^2}
    \left(
        \mathbb E\int_0^T
        \sum_{i=1}^d|\Delta\sigma_t^i|^2\,dt
    \right)^{1/2}
    \leq Cs_{\eta,\delta\Theta}<\infty.
\end{aligned}
\]
By the BDG inequality, the stochastic integral has an integrable
supremum and is therefore a true martingale with zero expectation.
Integrating from \(0\) to \(T\), using
\(p_T=\Phi_x(X_T)\) and \(\xi_0=0\), and adding the term involving
\(\Delta f\), we obtain
\[
\begin{aligned}
    &\mathbb E\!\left[\Phi_x(X_T)^\top\xi_T\right]
    +\mathbb E\int_0^T
        \left[
            f_x(t,X_t,u_t)^\top\xi_t+\Delta f_t
        \right]dt\\
    &\qquad=
    \mathbb E\int_0^T
    \left[
        H(t,X_t,p_t,q_t,u_t^\eta)
        -H(t,X_t,p_t,q_t,u_t)
    \right]dt.
\end{aligned}
\]

Next, applying It\^o's formula to
\(\frac12\xi_t^\top P_t\xi_t\), using the second-order adjoint
equation and the symmetry of \(P_t\) and \(Q_t^i\), gives
\[
\begin{aligned}
    d\left(\frac12\xi_t^\top P_t\xi_t\right)
    &=
    \Bigg[
        -\frac12\xi_t^\top
        H_{xx}(t,X_t,p_t,q_t,u_t)\xi_t
        +\xi_t^\top P_t\Delta b_t\\
    &\qquad+
        \sum_{i=1}^d
        \xi_t^\top\bigl[(C_t^i)^\top P_t+Q_t^i\bigr]
        \Delta\sigma_t^i\\
    &\qquad+
        \frac12\sum_{i=1}^d
        (\Delta\sigma_t^i)^\top P_t\Delta\sigma_t^i
    \Bigg]dt\\
    &\quad+
    \sum_{i=1}^d
    \left[
        \frac12\xi_t^\top Q_t^i\xi_t
        +\xi_t^\top P_t(C_t^i\xi_t+\Delta\sigma_t^i)
    \right]dW_t^i.
\end{aligned}
\]
Timewise Cauchy--Schwarz and H\"older's inequality give
\[
\begin{aligned}
    &\mathbb E\left[
        \left(
            \int_0^T\sum_{i=1}^d
            \left|
                \frac12\xi_t^\top Q_t^i\xi_t
                +\xi_t^\top P_t(C_t^i\xi_t+\Delta\sigma_t^i)
            \right|^2dt
        \right)^{1/2}
    \right]\\
    &\qquad\leq
    C\bigl(
        \|P\|_{\mathcal S^2}+\|Q\|_{\mathbb H^2}
    \bigr)\|\xi\|_{\mathcal S^4}^2\\
    &\qquad\quad+
    C\|P\|_{\mathcal S^2}\|\xi\|_{\mathcal S^4}
    \left[
        \mathbb E\left(
            \int_0^T
            \sum_{i=1}^d|\Delta\sigma_t^i|^2\,dt
        \right)^2
    \right]^{1/4}
    \leq Cs_{\eta,\delta\Theta}^2<\infty.
\end{aligned}
\]
The BDG inequality therefore also makes this stochastic integral a
true martingale with zero expectation.
Integrating and taking expectations, using
\(P_T=\Phi_{xx}(X_T)\) and \(\xi_0=0\), yields
\[
\begin{aligned}
    &\frac12\mathbb E\!\left[
        \xi_T^\top\Phi_{xx}(X_T)\xi_T
    \right]
    +\frac12\mathbb E\int_0^T
        \xi_t^\top f_{xx}(t,X_t,u_t)\xi_t\,dt\\
    &\qquad=
    \frac12\mathbb E\int_0^T
        \sum_{i=1}^d
        (\Delta\sigma_t^i)^\top P_t\Delta\sigma_t^i\,dt
    +R_P^{\eta,\delta\Theta},
\end{aligned}
\]
where
\[
\begin{aligned}
    R_P^{\eta,\delta\Theta}
    =\mathbb E\int_0^T\Bigg[
        &\xi_t^\top P_t\Delta b_t
        +\sum_{i=1}^d
        \xi_t^\top\bigl[(C_t^i)^\top P_t+Q_t^i\bigr]
        \Delta\sigma_t^i\\
        &-\frac12\xi_t^\top
        \bigl[
            H_{xx}(t,X_t,p_t,q_t,u_t)
            -f_{xx}(t,X_t,u_t)
        \bigr]\xi_t
    \Bigg]dt.
\end{aligned}
\]

The uniform boundedness of \(C_t^i\), timewise
Cauchy--Schwarz, and H\"older's inequality with exponents \(2,4,4\)
give
\[
\begin{aligned}
    &\mathbb E\int_0^T
    \Bigg[
        |\xi_t||P_t|
        \left(
            |\Delta b_t|
            +\sum_{i=1}^d|C_t^i||\Delta\sigma_t^i|
        \right)
        +\sum_{i=1}^d
        |\xi_t||Q_t^i||\Delta\sigma_t^i|
    \Bigg]dt\\
    &\qquad\leq
    C\bigl(
        \|P\|_{\mathcal S^2}+\|Q\|_{\mathbb H^2}
    \bigr)\|\xi\|_{\mathcal S^4}\\
    &\qquad\quad\times
    \left[
        \mathbb E\left(
            \int_0^T
            \left[
                |\Delta b_t|^2
                +\sum_{i=1}^d|\Delta\sigma_t^i|^2
            \right]dt
        \right)^2
    \right]^{1/4}
    \leq Cs_{\eta,\delta\Theta}^2.
\end{aligned}
\]
Moreover, the uniform boundedness of \(b_{xx}\) and
\(\sigma_{xx}\) implies
\[
    \left|
        H_{xx}(t,X_t,p_t,q_t,u_t)
        -f_{xx}(t,X_t,u_t)
    \right|
    \leq C\bigl(|p_t|+|q_t|\bigr).
\]
Consequently,
\[
\begin{aligned}
    &\mathbb E\int_0^T
    \left|
        \xi_t^\top
        \bigl[
            H_{xx}(t,X_t,p_t,q_t,u_t)
            -f_{xx}(t,X_t,u_t)
        \bigr]\xi_t
    \right|dt\\
    &\qquad\leq
    C\bigl(
        \|p\|_{\mathcal S^2}+\|q\|_{\mathbb H^2}
    \bigr)\|\xi\|_{\mathcal S^4}^2
    \leq Cs_{\eta,\delta\Theta}^2.
\end{aligned}
\]
Thus
\[
    |R_P^{\eta,\delta\Theta}|
    \leq Cs_{\eta,\delta\Theta}^2.
\]

Substituting the two adjoint identities into the cost expansion
and using the definition of \(\mathcal H_t^u\), we obtain
\[
\begin{aligned}
    J\!\left(\pi_{\bar\Theta_u+\eta\delta\Theta}\right)
    -J\!\left(\pi_{\bar\Theta_u}\right)
    &=
    \mathbb E\int_0^T
    \left[
        \mathcal H_t^u(u_t^\eta)
        -\mathcal H_t^u(u_t)
    \right]dt
    +R^{\eta,\delta\Theta},
\end{aligned}
\]
where
\[
\begin{aligned}
    R^{\eta,\delta\Theta}
    &=R_0^{\eta,\delta\Theta}+R_P^{\eta,\delta\Theta},\\
    |R^{\eta,\delta\Theta}|
    &\leq Cs_{\eta,\delta\Theta}^2
    =C\eta^2|\delta\Theta|_\Delta^2.
\end{aligned}
\]
The constant \(C\) depends only on the model constants, \(T\), the
fixed network architecture, and the uniform parameter bound in
Assumption~\ref{ass:nn-architecture}, and is therefore independent of
\(\eta\), \(\delta\Theta\), and the iteration index.
Taking the upper bound for \(R^{\eta,\delta\Theta}\) proves the result.
\end{proof}

Theorem~\ref{thm:parameter-space-soft-update-upper-bound} shows that the
cost variation along a parameter-space interpolation is controlled by a
corrected-Hamiltonian variation plus a second-order remainder. We next
make the order separation in this estimate explicit. This is important
because the soft-update parameter \(\eta\) enters the estimate through
different mechanisms: the Hamiltonian variation contains the leading
first-order response along the parameter direction, whereas both the
second-order correction term and the residual term are of quadratic order
in the parameter displacement. 
The result below makes this order separation explicit.

\begin{corollary}[Order separation under parameter-space soft update]
\label{cor:parameter-space-order-separation}
Suppose that Assumptions~\ref{ass:basic-regularity} and
\ref{ass:nn-architecture} hold, and retain the notation of
Theorem~\ref{thm:parameter-space-soft-update-upper-bound}.
Define
\[
\begin{aligned}
    \Delta_{\mathcal H}^{\eta}
    &:=
    \mathbb E\int_0^T
    \left[
        \mathcal H_t^u(u_t^\eta)
        -
        \mathcal H_t^u(u_t)
    \right]dt,\\
    \Delta_H^{\eta}
    &:=
    \mathbb E\int_0^T
    \left[
        H(t,X_t,p_t,q_t,u_t^\eta)
        -
        H(t,X_t,p_t,q_t,u_t)
    \right]dt,\\
    \Delta_P^{\eta}
    &:=
    \frac12
    \mathbb E\int_0^T
    \sum_{\ell=1}^{d}
    \bigl[
        \sigma^\ell(t,X_t,u_t^\eta)
        -
        \sigma^\ell(t,X_t,u_t)
    \bigr]^\top
    P_t
    \bigl[
        \sigma^\ell(t,X_t,u_t^\eta)
        -
        \sigma^\ell(t,X_t,u_t)
    \bigr]dt .
\end{aligned}
\]
Then
\[
    \Delta_{\mathcal H}^{\eta}
    =
    \Delta_H^{\eta}
    +
    \Delta_P^{\eta}.
\]
Moreover, there exists a constant \(C>0\), independent of
\(\eta\), \(\delta\Theta\), and the iteration index, such that
\[
    |\Delta_H^{\eta}|
    \le
    C\eta |\delta\Theta|_\Delta,
    \qquad
    |\Delta_P^{\eta}|
    \le
    C\eta^2 |\delta\Theta|_\Delta^2 .
\]
Consequently,
\[
    |\Delta_{\mathcal H}^{\eta}|
    \le
    C\eta |\delta\Theta|_\Delta
    +
    C\eta^2 |\delta\Theta|_\Delta^2 .
\]
Combining this with
Theorem~\ref{thm:parameter-space-soft-update-upper-bound}, the cost
variation admits the scale separation
\[
\begin{aligned}
    J\!\left(\pi_{\bar\Theta_u+\eta\delta\Theta}\right)
    -
    J\!\left(\pi_{\bar\Theta_u}\right)
    \le
    \Delta_{\mathcal H}^{\eta}
    +
    C\eta^2|\delta\Theta|_\Delta^2,
\end{aligned}
\]
where the leading Hamiltonian variation is of order
\(O(\eta |\delta\Theta|_\Delta)\), while the corrected-Hamiltonian quadratic term and
the upper-bound remainder are both of order
\(O(\eta^2 |\delta\Theta|_\Delta^2)\).
\end{corollary}
\begin{proof}
The identity
\[
    \Delta_{\mathcal H}^{\eta}
    =
    \Delta_H^{\eta}+\Delta_P^{\eta}
\]
follows immediately from the definition of \(\mathcal H_t^u\). Denote
\(\delta u_t:=u_t^\eta-u_t\). The state--control estimate established in
the proof of Theorem~\ref{thm:parameter-space-soft-update-upper-bound}
gives
\[
    \mathbb E\int_0^T|\delta u_t|^2\,dt
    \leq C\eta^2|\delta\Theta|_\Delta^2,
    \qquad
    \mathbb E\left[
        \left(\int_0^T|\delta u_t|^2\,dt\right)^2
    \right]
    \leq C\eta^4|\delta\Theta|_\Delta^4.
\]

By the control-Lipschitz assumptions on \(b\), \(\sigma\), and \(f\),
\[
\begin{aligned}
    &|b(t,X_t,u_t^\eta)-b(t,X_t,u_t)|
    +|\sigma(t,X_t,u_t^\eta)-\sigma(t,X_t,u_t)|
    \leq C|\delta u_t|,\\
    &|f(t,X_t,u_t^\eta)-f(t,X_t,u_t)|
    \leq C(1+|X_t|^2)|\delta u_t|.
\end{aligned}
\]
Consequently, the Cauchy--Schwarz inequality and the uniform moment
estimates for \(X\), \(p\), and \(q\) yield
\[
\begin{aligned}
    |\Delta_H^\eta|
    &\leq
    C\mathbb E\int_0^T
    \bigl(1+|X_t|^2+|p_t|+|q_t|\bigr)|\delta u_t|\,dt\\
    &\leq
    C\left(
        \mathbb E\int_0^T|\delta u_t|^2\,dt
    \right)^{1/2}
    \leq C\eta|\delta\Theta|_\Delta.
\end{aligned}
\]

For the second-order correction, the control-Lipschitz continuity of
\(\sigma\) gives
\[
\begin{aligned}
    |\Delta_P^\eta|
    &\leq
    C\mathbb E\left[
        \sup_{0\leq t\leq T}|P_t|
        \int_0^T|\delta u_t|^2\,dt
    \right]\\
    &\leq
    C\left(
        \mathbb E\sup_{0\leq t\leq T}|P_t|^2
    \right)^{1/2}
    \left\{
        \mathbb E\left[
            \left(\int_0^T|\delta u_t|^2\,dt\right)^2
        \right]
    \right\}^{1/2}\\
    &\leq C\eta^2|\delta\Theta|_\Delta^2.
\end{aligned}
\]
Thus,
\[
    |\Delta_{\mathcal H}^{\eta}|
    \leq
    C\eta|\delta\Theta|_\Delta+C\eta^2|\delta\Theta|_\Delta^2.
\]
The final cost estimate follows directly from
Theorem~\ref{thm:parameter-space-soft-update-upper-bound}.
\end{proof}

\section{Numerical experiments}
\label{sec:experiments}

\subsection{General experimental setup}
\label{subsec:numerical-general-setting}

In this section, we evaluate the proposed second-order corrected DeepSMP method on
a sequence of nonconvex stochastic control problems. The numerical experiments are
organized in two layers. The first layer compares DeepSMP with the direct method,
where the control network is trained by directly minimizing a Monte Carlo
approximation of the cost functional. The second layer compares DeepSMP with a
modified-MSA-type baseline, in order to isolate the effect of the second-order
adjoint correction.

All continuous-time problems are discretized by the Euler--Maruyama scheme
on the uniform grid \(0=t_0<t_1<\cdots<t_N=T\), with
\(\Delta t=T/N\). Unless otherwise stated, we set \(T=1\), \(N=40\), and
\(U=[-2,2]^{d_u}\), and enforce the control constraint through the
componentwise tanh-box parameterization
\(u_{\theta,k}(x)=2\tanh(r_{\theta,k}(x))\). All algorithms are implemented
in Python~3.10.20 with TensorFlow and run in double precision on an NVIDIA
GeForce RTX~4090 GPU. The default training batch size is \(256\), and the
reported costs are evaluated on a fixed validation sample of \(50{,}000\)
Brownian trajectories. Unless otherwise specified, results are reported as
the mean and standard deviation over five independent runs; within each
comparison, all methods use the same random seeds and validation paths.

In DeepSMP, the control functions and the martingale-integrand processes
\(q\) and \(Q\) are represented by separate time-indexed fully connected
networks, one for each time level, while \(p_0\) and \(P_0\) are trainable
parameters. Each network has two hidden layers with Swish activations, with
the layer widths specified for each test problem. The \(q\)- and
\(Q\)-networks have linear outputs, whereas the control-network outputs are
mapped through the tanh-box transformation above. No batch normalization is
used. 
The adjoint-related and control-network parameters are updated by separate
Adam optimizers with learning rates \(10^{-3}\) and \(2\times10^{-4}\),
respectively, thereby implementing a practical two-time-scale training
scheme in which the adjoint-related parameters evolve on the faster scale.
The control network is trained through
a candidate-based smooth approximation to the pointwise minimization of the
second-order corrected Hamiltonian. The Swish--tanh architecture satisfies
the activation and output-map conditions of
Assumption~\ref{ass:nn-architecture}; under its bounded-parameter-path
condition, Proposition~\ref{prop:network-derivative-bounds} applies.

Following Han and E~\cite{HanE2016}, the direct method represents the
feedback control at each time level by a separate neural network, propagates
the state through the Euler--Maruyama dynamics, and jointly trains all
control-network parameters by directly minimizing the Monte Carlo
approximation of the discretized cost functional. It therefore introduces
neither adjoint equations nor a Hamiltonian-based control update. The direct
baseline uses the same control-network architecture and tanh-box output
parameterization as DeepSMP. In the main comparison, batch normalization is
additionally inserted after each hidden affine transformation and before the
Swish activation.

For the comparison with modified MSA, we construct a modified-MSA-type
neural baseline in the same computational framework as DeepSMP. The
second-order adjoint system and the corresponding matrix-weighted
Hamiltonian correction are replaced by the scalar-weighted isotropic
\(\rho\)-augmentation of modified MSA, while the remaining neural
implementation and training protocol are unchanged. This comparison is
designed to isolate the effect of the second-order adjoint correction.

\begin{remark}[Two-time-scale training]
\label{rem:two-time-scale-training}
The use of different update speeds is motivated by the classical
two-time-scale stochastic-approximation principle
\cite{Borkar1997,KondaBorkar1999,KondaTsitsiklis2003}. The
adjoint-related parameters are assigned a larger learning rate than the
control-network parameters. On the faster adjoint time scale, the
reference control varies slowly and may therefore be regarded as
approximately frozen over a short training window, allowing the first-
and second-order adjoint networks to track the adjoint processes
associated with the current reference control. Conversely, on the slower
control time scale, the adjoint variables adapt more rapidly, so that the
corrected-Hamiltonian update is based on adjoint estimates that are better
adapted to the current reference control. This separation mitigates the
feedback instability caused by simultaneously changing the control and
the adjoint approximations. Since fixed learning rates are used in the
present implementation, the term ``two-time-scale'' refers to a practical
separation of update speeds rather than to the classical asymptotic
stochastic-approximation condition requiring the ratio of the slow and
fast step sizes to vanish.
\end{remark}

\subsection{Comparison with the direct method}
\label{subsec:comparison-direct-method}
We first compare DeepSMP with the direct method. The direct method optimizes
the mapping
\[
    \theta \longmapsto J(u^\theta)
\]
directly in the neural-network parameter space, without introducing adjoint
equations or a Hamiltonian-based control update. Although this approach is
straightforward and does not require a spatial grid, its end-to-end
optimization may be sensitive to nonconvex parameter landscapes containing
locally flat regions or multiple competing low-cost regions.

By contrast, DeepSMP exploits the state--adjoint--Hamiltonian structure
induced by the stochastic maximum principle and performs the control update
through a candidate-based smooth minimization of the second-order corrected
Hamiltonian in the action space. Rather than relying solely on an
infinitesimal descent direction of the Monte Carlo cost in parameter space,
the method compares finitely perturbed candidate actions according to their
corrected-Hamiltonian values. This action-space comparison may provide a
more discriminative update signal when the end-to-end parameter gradient is
weak or locally flat, and offers a mechanism for moving between separated
low-value action regions. 

\subsubsection{A one-dimensional nonconvex double-actuator problem}
We first consider a one-dimensional nonconvex stochastic control problem
with two competing actuator modes. The state process satisfies
\[
    dX_t=b(X_t,u_t)\,dt+\sigma(u_t)\,dW_t,
    \qquad X_0=0,\qquad u_t\in[-2,2],
\]
where
\[
\begin{aligned}
    b(x,u)
    &=
    -0.7x
    +1.45\exp\{-12(u-0.9)^2\}
    +1.25\exp\{-12(u+1.1)^2\}
    -0.05,\\
    \sigma(u)
    &=
    0.12
    +0.75\exp\{-12(u-0.9)^2\}
    +0.08\exp\{-12(u+1.1)^2\}.
\end{aligned}
\]
The cost functional is
\[
    J(u)
    =
    \mathbb E\left[
        4(X_T-1)^2
        +
        \int_0^1
        \left\{
            0.35(X_t-1)^2+0.015u_t^2
        \right\}dt
    \right].
\]

The two Gaussian actuator profiles generate two separated action modes.
Near \(u=0.9\), the control-induced drift increment is approximately
\(1.45\), while the diffusion coefficient is approximately \(0.87\).
Near \(u=-1.1\), the corresponding drift increment is slightly smaller,
approximately \(1.25\), whereas the diffusion coefficient is reduced to
approximately \(0.20\). The negative mode therefore provides a more
favorable drift--diffusion tradeoff for the quadratic state-tracking
objective. 
Moreover, away from the two actuator modes, the Gaussian
profiles are nearly flat in the action variable. This problem thus tests
whether the candidate-based corrected-Hamiltonian update can identify the
favorable action mode in a nonconvex landscape containing separated modes
and locally weak action sensitivity.

Since the two relevant actuator modes lie on opposite sides of the origin,
we measure their selection through the negative-region occupation rate
\[
    \rho_-
    :=
    \frac{1}{MN}
    \sum_{m=1}^{M}
    \sum_{k=0}^{N-1}
    \mathbf 1_{\{u_k^{(m)}<0\}},
\]
where \(u_k^{(m)}\) denotes the learned reference control evaluated along
the \(m\)-th validation trajectory. A value of \(\rho_-\) close to one
indicates predominant selection of the negative actuator mode. This
quantity is used as an auxiliary branch-selection diagnostic, while the
validation cost remains the primary performance criterion.

As a reference, under the constant control \(u_t\equiv-1.1\), the state
equation reduces to an Ornstein--Uhlenbeck-type process, and the
corresponding objective value can be evaluated explicitly as
\(J(u_t\equiv-1.1)\approx0.2996\). This value serves only as a
constant-control benchmark and is not asserted to be the optimal value of
the full feedback-control problem.

For this experiment, the control networks in both methods, as well as the
\(q\)- and \(Q\)-networks in DeepSMP, have hidden-layer width \(32\).
DeepSMP uses \(16\) randomized action candidates at each time level,
consisting of the current action and \(15\) candidates generated by
Gaussian perturbations of the raw network output with standard deviation
\(0.8\). The smooth candidate-aggregation temperature is \(0.15\).

To mitigate the possibility that the observed performance gap is merely
caused by insufficient optimization of the direct method, we allow the
direct method to train for $10^{5}$ iterations, whereas DeepSMP is
evaluated after $3\times10^{4}$ iterations, at which point its validation
cost has already stabilized. Thus, the direct method is given more than
three times as many optimization iterations, providing a deliberately
favorable training budget for the baseline. The optimization trajectories
of the two methods over the first $3\times10^{4}$ iterations are presented
in Figure~\ref{fig:comparison_1d_curves}, and the corresponding
final-iteration statistics across the five random seeds are summarized in
Table~\ref{tab:comparison_1d_results}.

\begin{figure}[!htbp]
    \centering

    \begin{minipage}{0.98\textwidth}
        \centering

        \captionsetup[subfigure]{
            font=footnotesize,
            skip=2pt
        }

        \begin{subfigure}[t]{0.48\linewidth}
            \centering
            \includegraphics[width=0.98\linewidth]
            {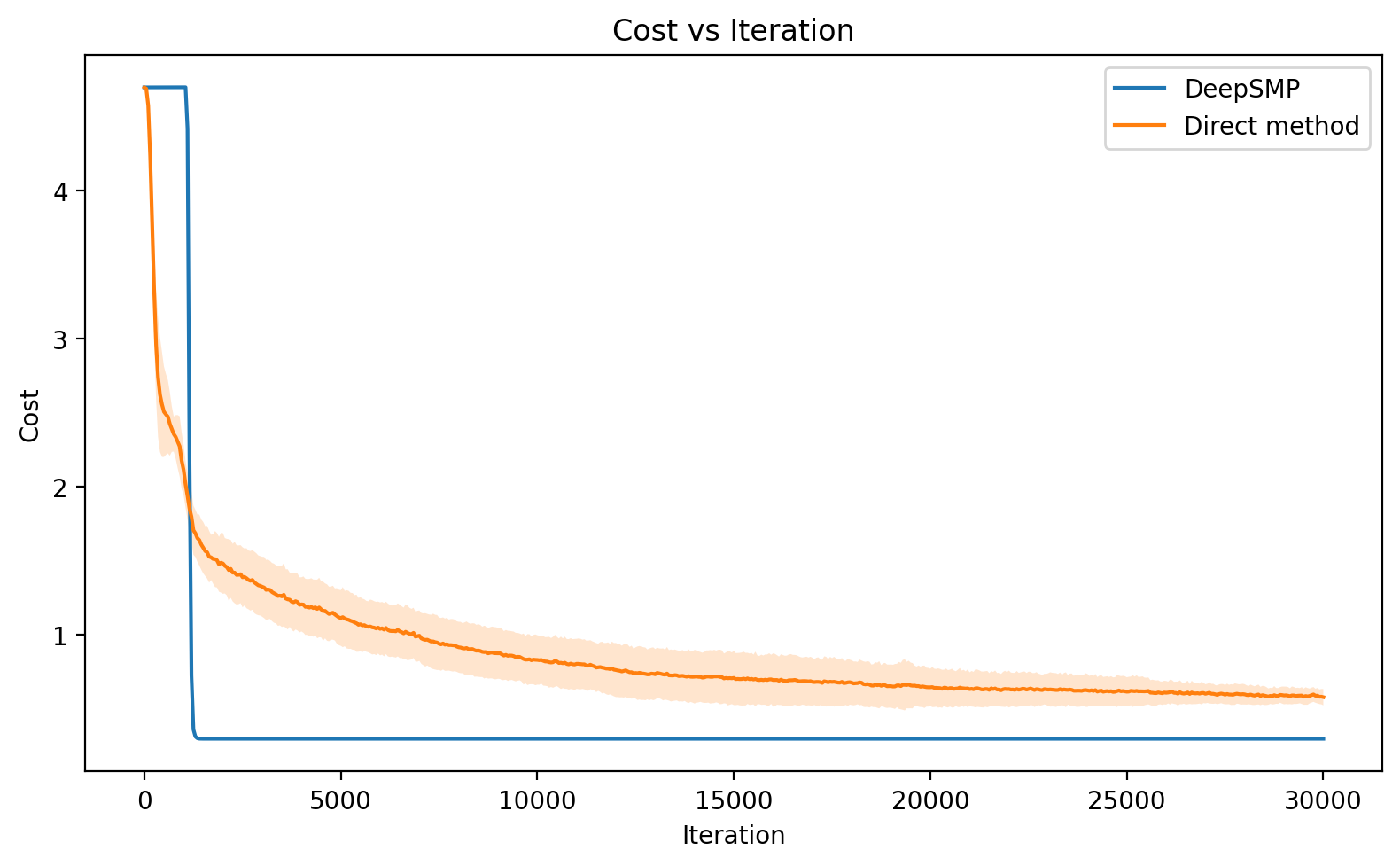}
            \caption{Cost.}
            \label{fig:comparison_1d_cost}
        \end{subfigure}
        \hfill
        \begin{subfigure}[t]{0.48\linewidth}
            \centering
            \includegraphics[width=0.98\linewidth]
            {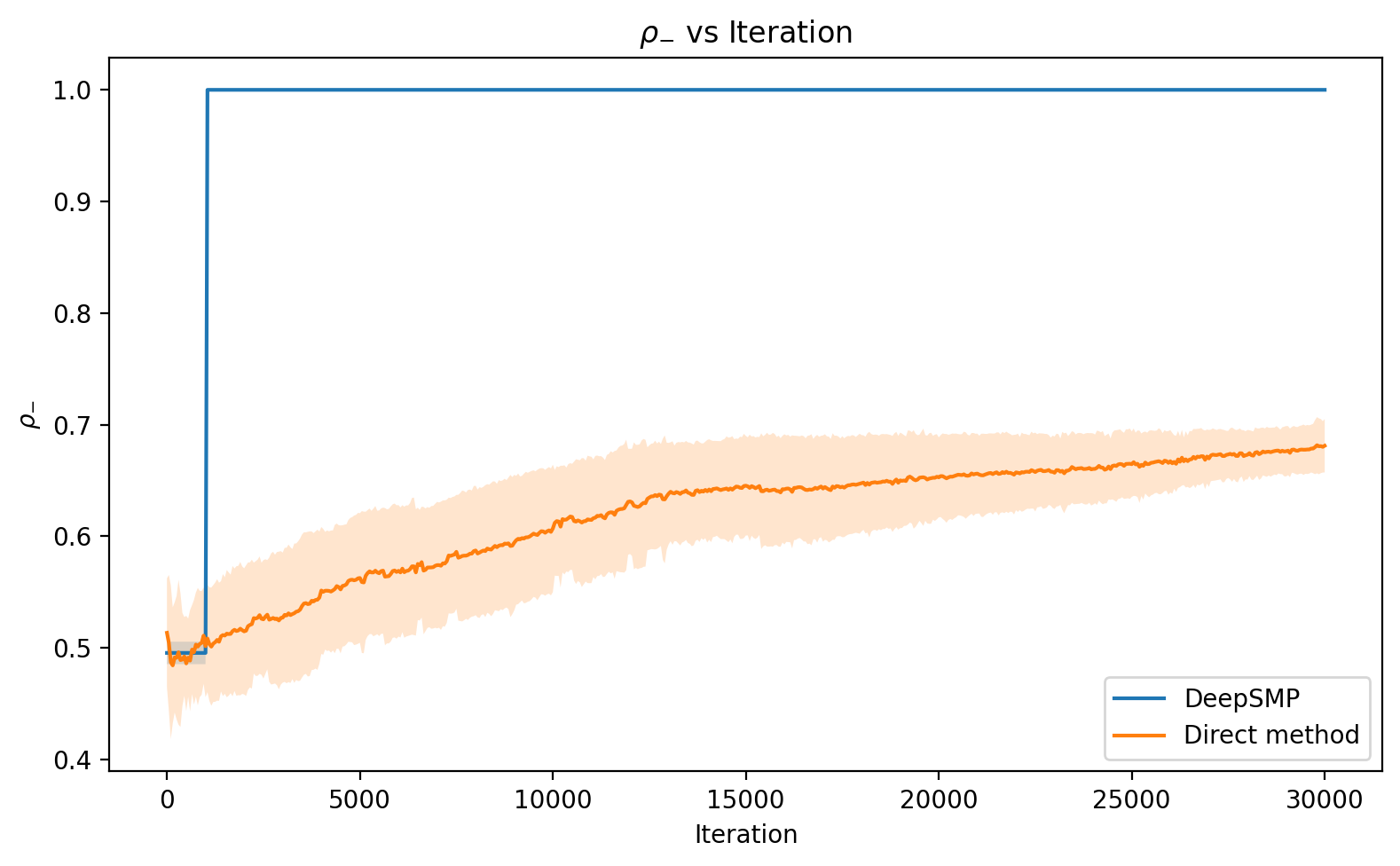}
            \caption{Negative-region proportion $\rho_{-}$.}
            \label{fig:comparison_1d_rho}
        \end{subfigure}

        \vspace{-0.05cm}

        \begin{subfigure}[t]{0.48\linewidth}
            \centering
            \includegraphics[width=0.98\linewidth]
            {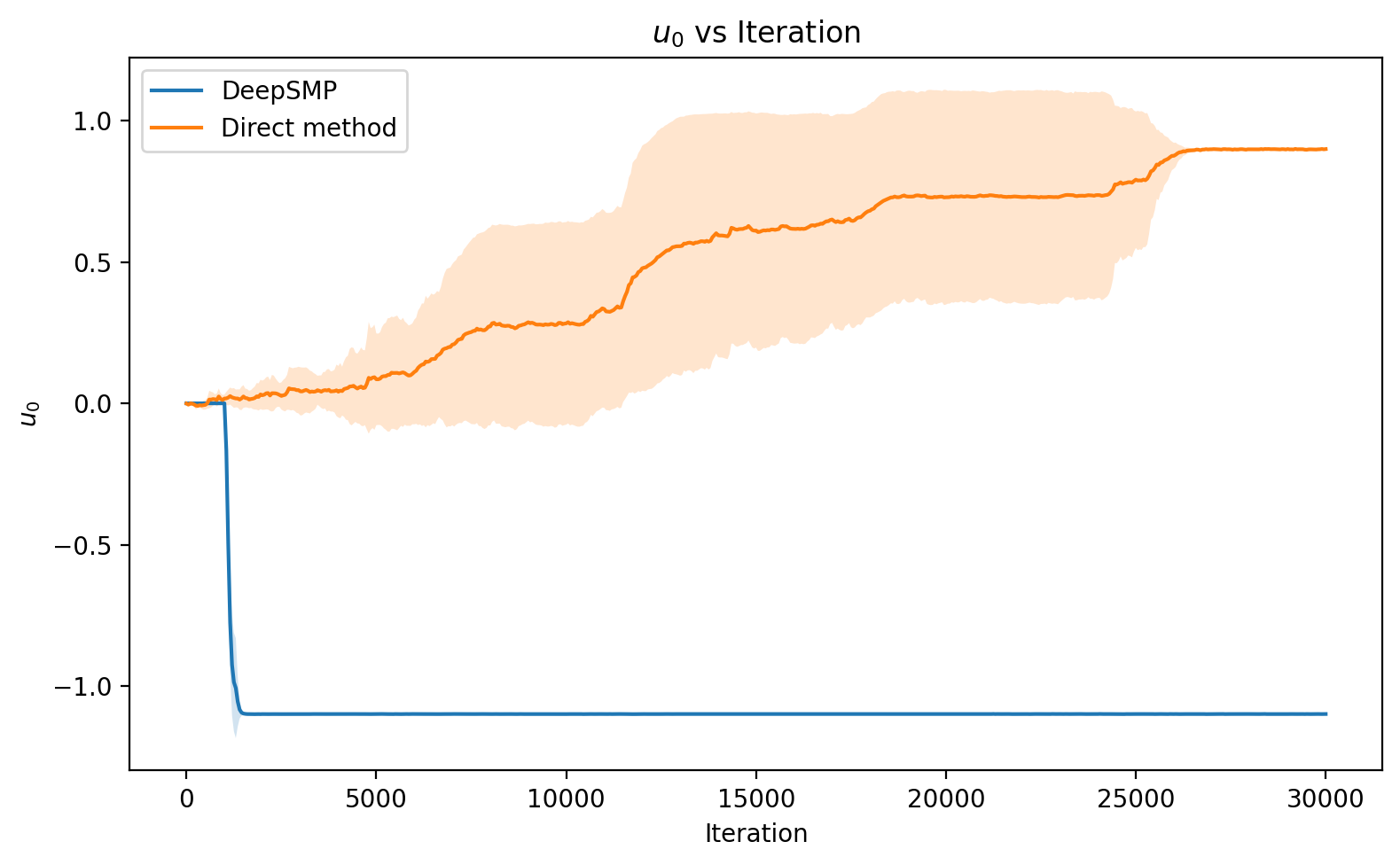}
            \caption{Initial control $u_0$.}
            \label{fig:comparison_1d_u0}
        \end{subfigure}
        \hfill
        \begin{subfigure}[t]{0.48\linewidth}
            \centering
            \includegraphics[width=0.98\linewidth]
            {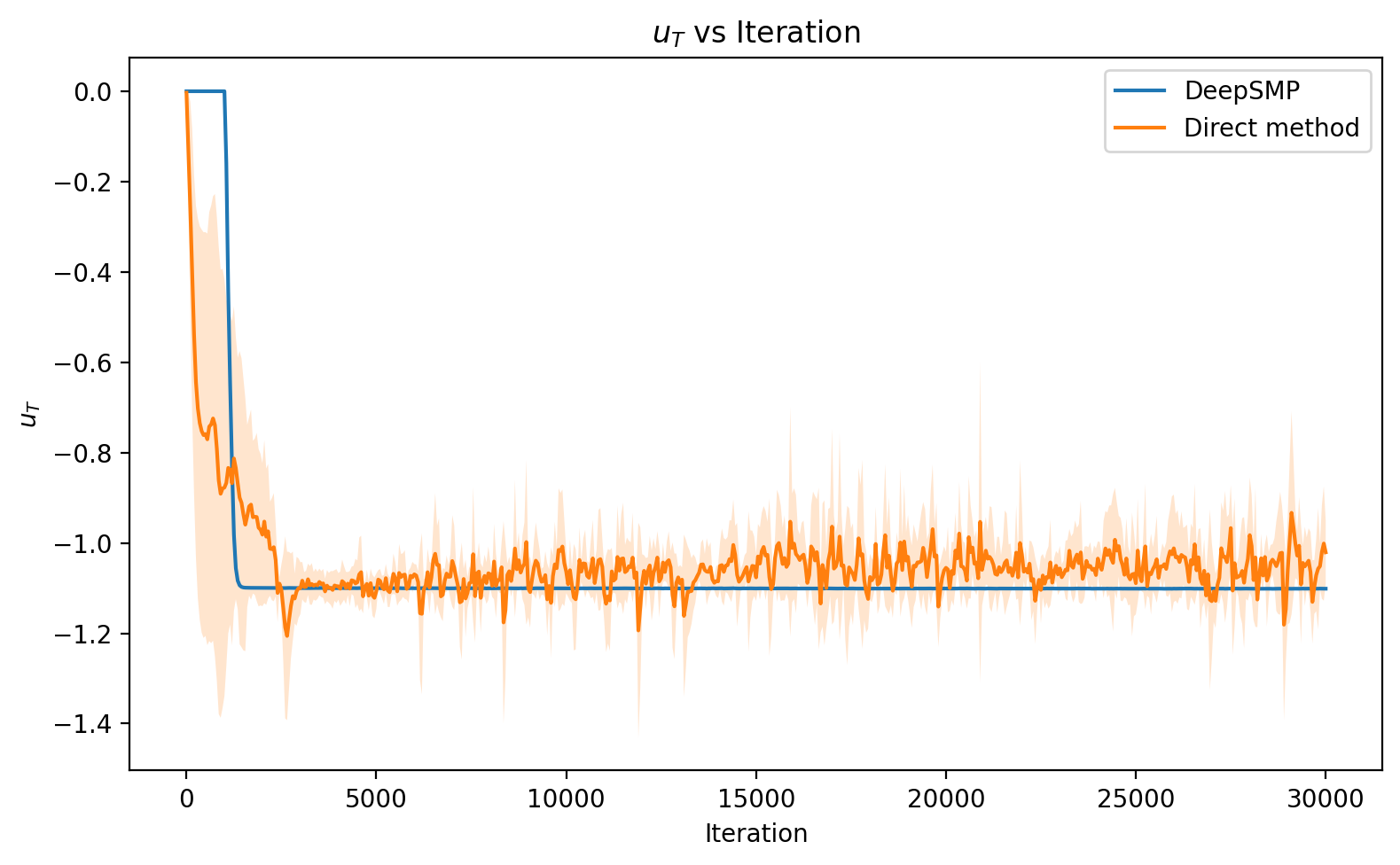}
            \caption{Terminal control $u_T$.}
            \label{fig:comparison_1d_uT}
        \end{subfigure}

    \end{minipage}

    \captionsetup{
        font=small,
        width=0.88\textwidth,
        skip=4pt
    }
    \caption{
        Optimization trajectories of DeepSMP and the direct method in the
        one-dimensional setting. The solid curves represent the means over
        five random seeds, while the shaded regions represent one standard
        deviation. For DeepSMP, all quantities are evaluated using the target
        control network. Both methods are displayed over the first
        $3\times10^{4}$ training iterations.
    }
    \label{fig:comparison_1d_curves}
\end{figure}

\begin{table}[htbp]
    \centering
    \caption{Performance comparison of DeepSMP and the direct method under
    five random seeds in the one-dimensional setting.
    They are evaluated using the control network at the
    final iteration.
    }
    \label{tab:comparison_1d_results}

    \setlength{\tabcolsep}{4pt}
    \renewcommand{\arraystretch}{1.15}

    \resizebox{\textwidth}{!}{%
    \begin{tabular}{l c cc cc cc cc}
        \toprule
        \multirow{2}{*}{Method}
        & \multirow{2}{*}{Iter.}
        & \multicolumn{2}{c}{Cost}
        & \multicolumn{2}{c}{$u_0$}
        & \multicolumn{2}{c}{$u_T$}
        & \multicolumn{2}{c}{$\rho_{-}$} \\
        \cmidrule(lr){3-4}
        \cmidrule(lr){5-6}
        \cmidrule(lr){7-8}
        \cmidrule(lr){9-10}
        &
        & Mean & Std.
        & Mean & Std.
        & Mean & Std.
        & Mean & Std. \\
        \midrule
        DeepSMP
        & $3\times10^{4}$
        & $0.2989$
        & $3.3164\times10^{-6}$
        & $-1.0984$
        & $1.6254\times10^{-4}$
        & $-1.1014$
        & $4.3017\times10^{-4}$
        & $1.0000$
        & $0.0000$ \\
        Direct method
        & $10^{5}$
        & $0.5164$
        & $0.0178$
        & $0.8982$
        & $0.0012$
        & $-1.0095$
        & $0.1676$
        & $0.7078$
        & $0.0238$ \\
        \bottomrule
    \end{tabular}%
    }
\end{table}

Figure~\ref{fig:comparison_1d_curves} shows a clear difference in the
optimization behavior of the two methods under the same iteration horizon.
DeepSMP rapidly enters a low-cost regime and stabilizes near $0.30$, while
the associated standard-deviation band becomes almost negligible. At the
same time, both $u_0$ and $u_T$ approach approximately $-1.1$, and
$\rho_{-}$ quickly reaches one. These observations indicate that DeepSMP
consistently identifies the negative low-cost control region across all
five random seeds. In contrast, the direct method decreases the cost much
more gradually and exhibits a mixed control pattern: its mean $u_0$ moves
toward the positive region, whereas $u_T$ remains negative, and
$\rho_{-}$ stays substantially below one.

The final statistics in Table~\ref{tab:comparison_1d_results} confirm that
this difference is not merely transient. DeepSMP attains a mean cost of
$0.2989$ with a standard deviation of only
$3.3164\times10^{-6}$ after $3\times10^{4}$ iterations. By comparison,
the direct method obtains a higher mean cost of $0.5164$ with a standard
deviation of $0.0178$, even after $10^{5}$ iterations. Moreover, DeepSMP
produces highly consistent endpoint controls, with
$u_0=-1.0984$, $u_T=-1.1014$, and $\rho_{-}=1.0000$, whereas the direct
method yields $u_0=0.8982$, $\rho_{-}=0.7078$, and a markedly larger
standard deviation of $0.1676$ for $u_T$. Hence, DeepSMP achieves not only
a lower control cost within a substantially smaller training budget, but
also a considerably more stable and coherent selection of the low-cost
control region.

The direct method is trained with hidden-layer batch normalization and for a
longer horizon than DeepSMP. It partially moves toward the negative
low-diffusion action region, as reflected by \(u_T\approx -1.01\) and
\(\rho_-\approx0.71\). However, its initial-time control remains close to the
positive high-diffusion region, \(u_0\approx0.90\), and the overall occupation
rate of the negative region remains far from one. In contrast, DeepSMP, even
without batch normalization in the control network, identifies the negative
low-diffusion region throughout the time interval and achieves a substantially
smaller validation cost.

\subsubsection{A five-dimensional coupled double-actuator problem}
\label{subsubsec:5d-double-actuator}

We next consider a genuinely five-dimensional coupled double-actuator stochastic
control problem. In this example,
\[
    X_t\in\mathbb R^5,\qquad
    u_t\in[-2,2]^5,\qquad
    W_t\in\mathbb R^5 .
\]
The state process satisfies
\[
    dX_t = b(X_t,u_t)\,dt+\Sigma(u_t)\,dW_t,
    \qquad X_0=0.
\]
For \(u=(u_1,\ldots,u_5)^\top\), define the componentwise actuator functions
\[
    \phi_+(u_i)=\exp\{-12(u_i-0.9)^2\},
    \qquad
    \phi_-(u_i)=\exp\{-12(u_i+1.1)^2\},
\]
and write
\[
    \phi_+(u)
    =
    \bigl(\phi_+(u_1),\ldots,\phi_+(u_5)\bigr)^\top,
    \qquad
    \phi_-(u)
    =
    \bigl(\phi_-(u_1),\ldots,\phi_-(u_5)\bigr)^\top .
\]
The drift is given by
\[
    b(x,u)
    =
    A x
    +
    B_+\phi_+(u)
    +
    B_-\phi_-(u)
    +
    b_0,
\]
where
\[
A=
\begin{pmatrix}
-0.70&0.10&0&0&0.04\\
0.06&-0.75&0.08&0&0\\
0&0.07&-0.68&0.06&0\\
0&0&0.09&-0.72&0.05\\
0.03&0&0&0.07&-0.74
\end{pmatrix},
\]
\[
B_+=
\begin{pmatrix}
1.45&0.12&0&0&0.05\\
0.10&1.35&0.10&0&0\\
0&0.12&1.42&0.10&0\\
0&0&0.08&1.38&0.10\\
0.06&0&0&0.12&1.40
\end{pmatrix},
\qquad
B_-=
\begin{pmatrix}
1.25&0.08&0&0&0.03\\
0.06&1.18&0.08&0&0\\
0&0.08&1.24&0.06&0\\
0&0&0.06&1.20&0.08\\
0.04&0&0&0.08&1.22
\end{pmatrix},
\]
and
\[
    b_0=(-0.05,-0.03,-0.04,-0.02,-0.05)^\top .
\]

The diffusion coefficient is a \(5\times5\) matrix depending on the control. It is
defined by
\[
    \Sigma(u)
    =
    D(u)
    +
    \operatorname{Diag}(\phi_+(u))S_+
    +
    \operatorname{Diag}(\phi_-(u))S_-,
\]
where
\[
    D(u)
    =
    \operatorname{Diag}
    \bigl(
    0.12+0.75\phi_+(u_i)+0.08\phi_-(u_i)
    \bigr)_{i=1}^5,
\]
and the off-diagonal diffusion patterns are
\[
S_+=
\begin{pmatrix}
0&0.05&0&0&0.03\\
0.04&0&0.04&0&0\\
0&0.04&0&0.04&0\\
0&0&0.05&0&0.04\\
0.03&0&0&0.05&0
\end{pmatrix},
\]
\[
S_-=
\begin{pmatrix}
0&0.015&0&0&0.010\\
0.012&0&0.012&0&0\\
0&0.012&0&0.012&0\\
0&0&0.015&0&0.012\\
0.010&0&0&0.015&0
\end{pmatrix}.
\]
Thus, each control component affects not only the corresponding diagonal
diffusion intensity but also the allocation of the Brownian noise among
different state coordinates through the off-diagonal entries of
\(S_+\) and \(S_-\).

The target vector is
\[
    x_{\rm tar}=(1.00,0.80,1.20,0.90,1.10)^\top .
\]
The running and terminal costs are coupled quadratic tracking costs:
\[
    f(x,u)
    =
    0.35(x-x_{\rm tar})^\top M_{\rm run}(x-x_{\rm tar})
    +
    0.015|u|^2,
\]
and
\[
    \Phi(x)
    =
    4(x-x_{\rm tar})^\top M_{\rm term}(x-x_{\rm tar}),
\]
where
\[
M_{\rm run}=
\begin{pmatrix}
1&0.06&0&0&0.03\\
0.06&1&0.06&0&0\\
0&0.06&1&0.06&0\\
0&0&0.06&1&0.06\\
0.03&0&0&0.06&1
\end{pmatrix},
\]
and
\[
M_{\rm term}=
\begin{pmatrix}
1&0.10&0&0&0.05\\
0.10&1&0.10&0&0\\
0&0.10&1&0.10&0\\
0&0&0.10&1&0.10\\
0.05&0&0&0.10&1
\end{pmatrix}.
\]
The objective is therefore
\[
    J(u)
    =
    \mathbb E
    \left[
    \Phi(X_T)
    +
    \int_0^T f(X_t,u_t)\,dt
    \right].
\]
This problem is a genuinely coupled multidimensional extension of the
one-dimensional two-mode actuator example. For each control coordinate,
the positive mode near \(u_i=0.9\) produces a stronger control-induced drift
effect but a substantially larger diagonal diffusion intensity, whereas
the negative mode near \(u_i=-1.1\) yields a slightly weaker drift effect
together with considerably smaller diffusion. The matrices \(S_+\) and
\(S_-\) additionally introduce off-diagonal noise coupling across the state
coordinates. The example therefore tests whether an algorithm can
consistently identify the favorable negative actuator mode across time,
sample paths, and control coordinates in a coupled nonconvex action
landscape.

For this experiment, all time-indexed control and adjoint-related networks
have two hidden layers of width $64$. The direct method is trained for
$10^{5}$ iterations, whereas DeepSMP is reported after $3\times10^{4}$
iterations, by which point its validation performance has stabilized.
The optimization trajectories of the two methods over the first
$3\times10^{4}$ iterations are shown in
Figure~\ref{fig:comparison_5d_curves}, while the corresponding
final-iteration statistics across the five random seeds are summarized in
Table~\ref{tab:comparison_5d_results}.

\begin{figure}[H]
    \centering

    \begin{minipage}{0.98\textwidth}
        \centering

        \captionsetup[subfigure]{
            font=footnotesize,
            skip=2pt
        }

        \begin{subfigure}[t]{0.48\linewidth}
            \centering
            \includegraphics[width=0.98\linewidth]
            {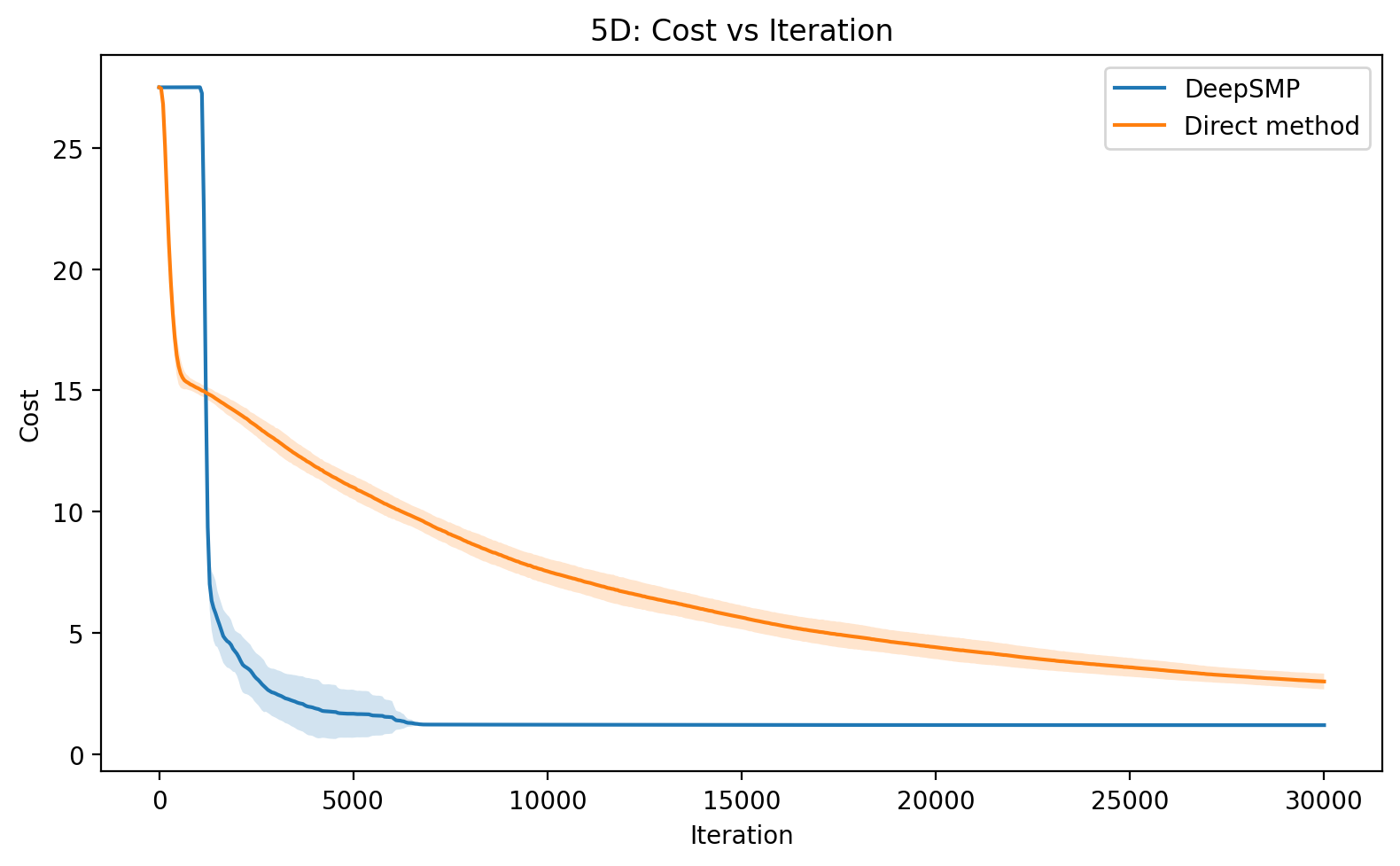}
            \caption{Cost.}
            \label{fig:comparison_5d_cost}
        \end{subfigure}
        \hfill
        \begin{subfigure}[t]{0.48\linewidth}
            \centering
            \includegraphics[width=0.98\linewidth]
            {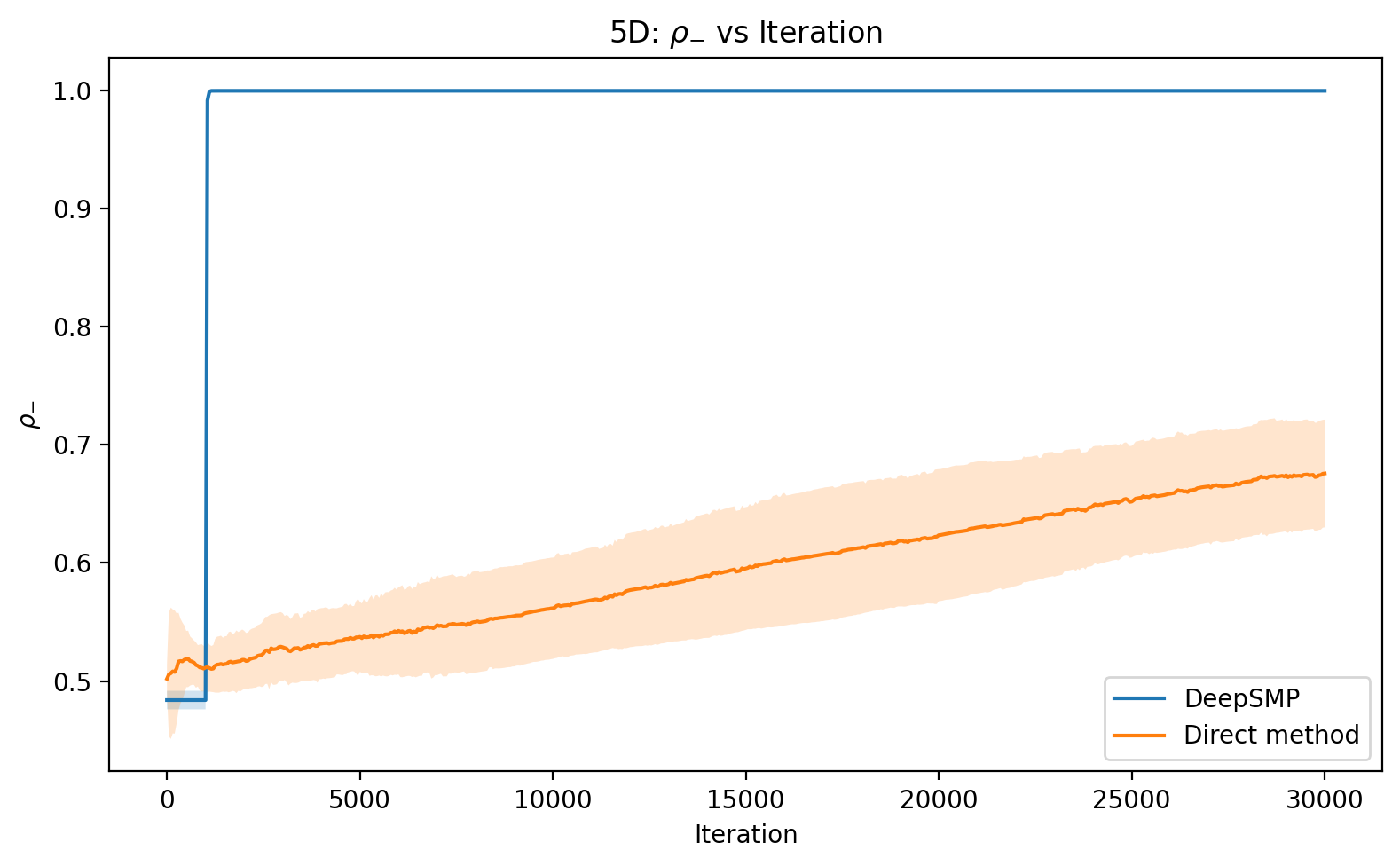}
            \caption{Negative-region proportion $\rho_{-}$.}
            \label{fig:comparison_5d_rho}
        \end{subfigure}

        \vspace{-0.05cm}

        \begin{subfigure}[t]{0.48\linewidth}
            \centering
            \includegraphics[width=0.98\linewidth]
            {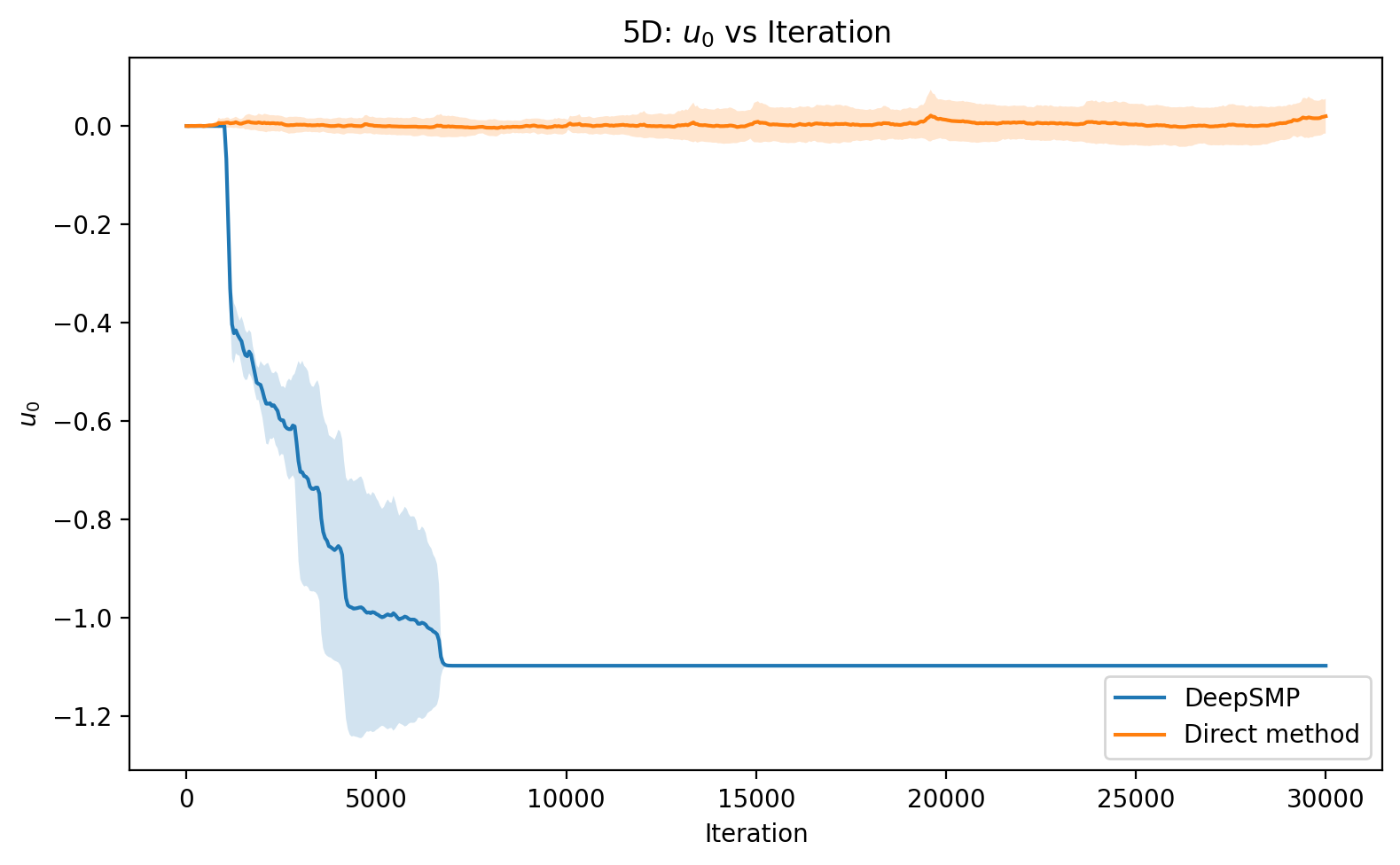}
            \caption{Initial control $u_0$.}
            \label{fig:comparison_5d_u0}
        \end{subfigure}
        \hfill
        \begin{subfigure}[t]{0.48\linewidth}
            \centering
            \includegraphics[width=0.98\linewidth]
            {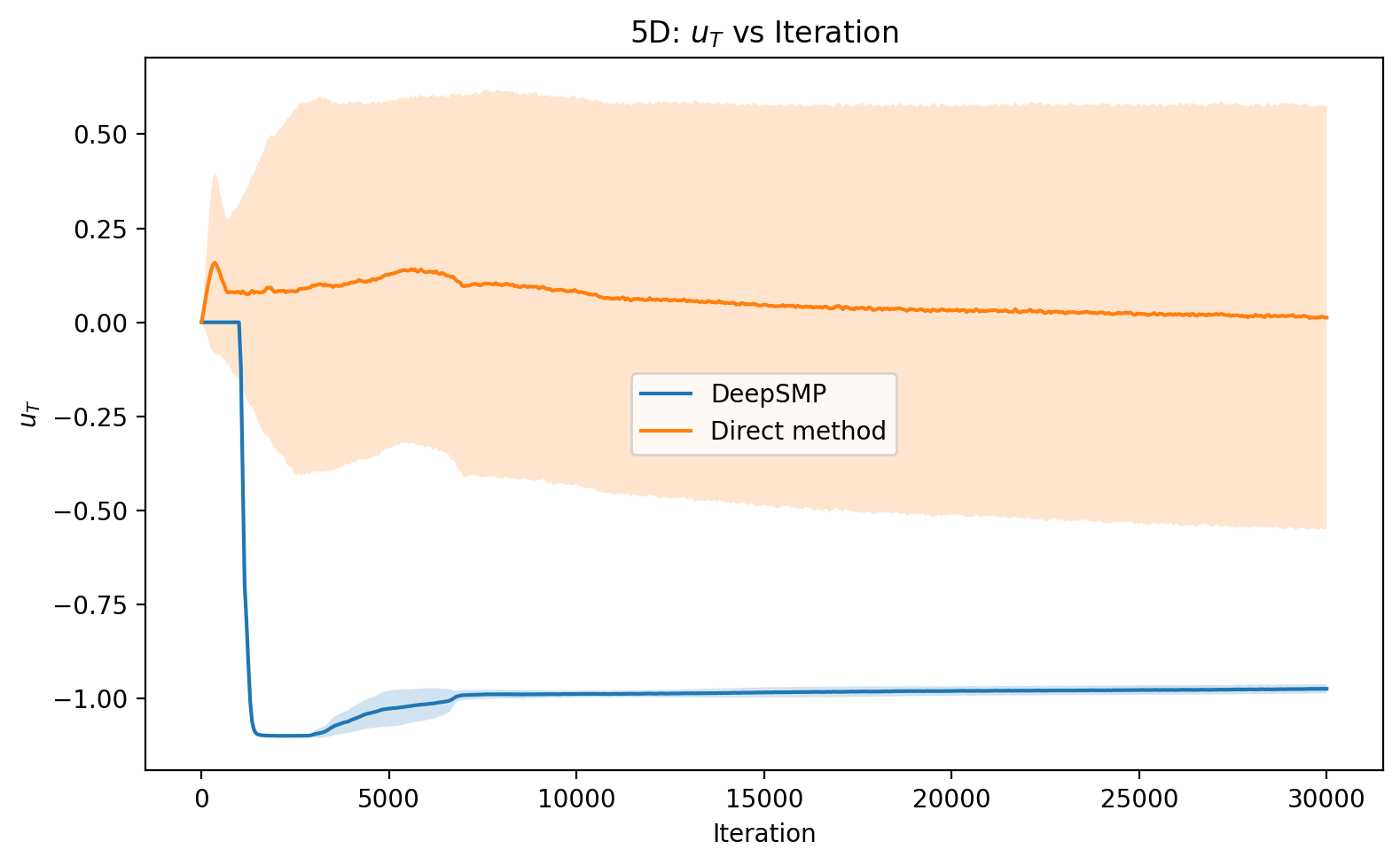}
            \caption{Terminal control $u_T$.}
            \label{fig:comparison_5d_uT}
        \end{subfigure}

    \end{minipage}

    \captionsetup{
        font=small,
        width=0.88\textwidth,
        skip=4pt
    }
    \caption{
        Optimization trajectories of DeepSMP and the direct method in the
        five-dimensional setting. 
    }
    \label{fig:comparison_5d_curves}
\end{figure}

\begin{table}[H]
    \centering
    \caption{Performance comparison of DeepSMP and the direct method under
    five random seeds in the five-dimensional setting.
    }
    \label{tab:comparison_5d_results}

    \setlength{\tabcolsep}{4pt}
    \renewcommand{\arraystretch}{1.15}

    \resizebox{\textwidth}{!}{%
    \begin{tabular}{l c cc cc cc cc}
        \toprule
        \multirow{2}{*}{Method}
        & \multirow{2}{*}{Iter.}
        & \multicolumn{2}{c}{Cost}
        & \multicolumn{2}{c}{$u_0$}
        & \multicolumn{2}{c}{$u_T$}
        & \multicolumn{2}{c}{$\rho_{-}$} \\
        \cmidrule(lr){3-4}
        \cmidrule(lr){5-6}
        \cmidrule(lr){7-8}
        \cmidrule(lr){9-10}
        &
        & Mean & Std.
        & Mean & Std.
        & Mean & Std.
        & Mean & Std. \\
        \midrule
        DeepSMP
        & $3\times10^{4}$
        & $1.2032$
        & $6.1521\times10^{-4}$
        & $-1.0982$
        & $1.6501\times10^{-6}$
        & $-0.9743$
        & $0.0118$
        & $1.0000$
        & $3.3783\times10^{-6}$ \\
        Direct method
        & $10^{5}$
        & $2.1420$
        & $0.1855$
        & $0.8975$
        & $8.6597\times10^{-4}$
        & $0.0046$
        & $0.5794$
        & $0.6942$
        & $0.0518$ \\
        \bottomrule
    \end{tabular}%
    }
\end{table}

Figure~\ref{fig:comparison_5d_curves} shows that the distinction observed
in the one-dimensional example persists in the coupled five-dimensional
setting. Over the same $3\times10^{4}$-iteration horizon, DeepSMP rapidly
enters a low-cost regime and subsequently exhibits negligible cross-seed
variation. Its target control is simultaneously driven toward the negative
actuator mode: the initial and terminal control statistics approach the
neighborhood of $-1.1$ and $-1.0$, respectively, while $\rho_{-}$ rapidly
converges to one. By contrast, the cost of the direct method decreases much
more gradually. Its initial control remains close to zero during the first
$3\times10^{4}$ iterations, and the large standard-deviation band associated
with $u_T$ indicates substantial heterogeneity across random seeds. Moreover,
$\rho_{-}$ remains well below one, showing that the learned policies retain
a mixed allocation between the two actuator modes.

The final statistics in Table~\ref{tab:comparison_5d_results} further confirm
that this discrepancy is not a transient optimization effect. DeepSMP attains
a mean cost of $1.2032$ with a standard deviation of
$6.1521\times10^{-4}$ after $3\times10^{4}$ iterations, whereas the direct
method yields a mean cost of $2.1420$ with a standard deviation of $0.1855$
even after $10^{5}$ iterations. This corresponds to a reduction of
approximately $43.8\%$ in the mean cost. DeepSMP also produces highly
consistent control statistics, with $u_0=-1.0982$, $u_T=-0.9743$, and
$\rho_{-}=1.0000$. In contrast, the direct method gives $u_0=0.8975$,
$u_T=0.0046$, and $\rho_{-}=0.6942$, together with a particularly large
standard deviation of $0.5794$ for $u_T$. Thus, despite using a substantially
smaller optimization budget, DeepSMP consistently selects the favorable
negative low-diffusion mode across control dimensions and random seeds,
while the direct method remains sensitive to initialization and converges
to a markedly less coherent control configuration.

\subsubsection{A twenty-dimensional stress test}
\label{subsubsec:20d-stress-test}

We finally consider a twenty-dimensional stress test. The state, control,
and Brownian motion are all twenty-dimensional. The drift consists of a
stable linear component with cyclic coupling and nonlinear two-mode actuator
terms, while the diffusion contains both controlled diagonal intensities
and off-diagonal noise coupling. The complete problem specification is
provided in Appendix~\ref{app:20d-setting}.

For this benchmark, the second-order adjoint equation admits an exact
structural simplification. Indeed, \(b_x=A\) is constant,
\(b_{xx}=0\), the diffusion coefficient is independent of the state, and
the state Hessians of the running and terminal costs are constant.
Consequently, the second-order adjoint equation reduces to a deterministic
matrix-valued backward equation, and its martingale-integrand process
satisfies
\[
    Q_t\equiv0.
\]
We therefore impose this structural identity exactly rather than
parameterizing \(Q\) by a neural network. In a generic twenty-dimensional
problem,
\[
    Q_k\in\mathbb R^{20\times20\times20},
\]
so this simplification avoids the redundant approximation of an identically
zero \(8000\)-dimensional tensor at each time level. The nontrivial
matrix-valued process \(P\) is nevertheless fully retained: \(P_0\) remains
a trainable parameter, and the subsequent values of \(P\) are generated
through the discretized second-order adjoint dynamics and used in the
second-order corrected Hamiltonian.

The twenty-dimensional action space also makes candidate exploration more
challenging. A finite candidate family generated solely by independent
local perturbations may provide insufficient coverage of correlated
low-Hamiltonian directions. We therefore combine independent local
perturbations with structured common-shift perturbations in the raw-action
space. Suppressing the batch-sample index, the local and structured
candidates are generated as
\[
\begin{aligned}
    \widetilde r_k^{(j),\mathrm{loc}}
    &=
    r_{\theta,k}(X_k)+\zeta_k^{(j)},\\
    \widetilde r_k^{(j),\mathrm{str}}
    &=
    r_{\theta,k}(X_k)
    +\gamma_k^{(j)}\mathbf 1_{20}
    +\widetilde\zeta_k^{(j)},
\end{aligned}
\]
where
\[
\begin{aligned}
    \zeta_k^{(j)}
    &\sim
    \mathcal N
    \bigl(0,\sigma_{\mathrm{loc}}^2I_{20}\bigr),\\
    \gamma_k^{(j)}
    &\sim
    \mathcal N
    \bigl(0,\sigma_{\mathrm{str}}^2\bigr),\\
    \widetilde\zeta_k^{(j)}
    &\sim
    \mathcal N
    \bigl(0,\sigma_{\mathrm{str,loc}}^2I_{20}\bigr).
\end{aligned}
\]
These random variables are sampled independently for each time level,
batch trajectory, and candidate. Thus, \(\gamma_k^{(j)}\) is a random
scalar common-shift coefficient, whereas
\(\sigma_{\mathrm{str}}\), \(\sigma_{\mathrm{str,loc}}\), and
\(\sigma_{\mathrm{loc}}\) are fixed exploration hyperparameters. The term
\(\gamma_k^{(j)}\mathbf 1_{20}\) produces a common-mode displacement whose
covariance has rank one, while \(\widetilde\zeta_k^{(j)}\) introduces a
smaller coordinatewise fluctuation around that displacement.

In addition to the randomized
candidates, the unperturbed raw action \(r_{\theta,k}(X_k)\) is retained as
one candidate. In the implementation, we use \(96\) structured candidates
and \(31\) independent local candidates, resulting in \(128\) candidates
after including the unperturbed raw action. The corresponding perturbation
scales are
\[
    \sigma_{\mathrm{str}}=1.35,
    \qquad
    \sigma_{\mathrm{str,loc}}=0.15,
    \qquad
    \sigma_{\mathrm{loc}}=0.60.
\]

To stabilize the evolution of the neural raw-action center, we further
introduce an elite-candidate fitting term. Let
\(\mathscr H_k^{(j)}\) denote the corrected-Hamiltonian value of the
\(j\)-th candidate at time \(t_k\), and let \(\mathcal E_k\) contain the
indices of the \(K_{\mathrm{elite}}\) best candidates according to the
corrected-Hamiltonian criterion. The corresponding elite raw-action average
is defined by
\[
    \bar r_k^{\mathrm{elite}}
    =
    \frac{1}{K_{\mathrm{elite}}}
    \sum_{j\in\mathcal E_k}
    \widetilde r_k^{(j)}.
\]
The auxiliary fitting loss is
\[
    \mathcal L_{\mathrm{fit}}
    =
    \mathbb E_{\mathrm{batch}}
    \sum_{k=0}^{N-1}
    \Delta t
    \left\|
        r_{\theta,k}(X_k)
        -
        \operatorname{sg}
        \bigl(\bar r_k^{\mathrm{elite}}\bigr)
    \right\|^2,
\]
where \(\operatorname{sg}(\cdot)\) denotes the stop-gradient operation. The
control-network loss in this experiment is therefore
\[
    \mathcal L_{\mathrm{policy}}
    =
    \lambda_H\mathcal L_H
    +
    \lambda_{\mathrm{fit}}\mathcal L_{\mathrm{fit}}.
\]
The elite-fitting term transfers the location information carried by the
best sampled candidates to the neural raw-action center. It neither replaces
the corrected-Hamiltonian smooth objective nor changes the candidate-ranking
criterion. Rather, it provides an auxiliary mechanism for retaining
high-quality candidate information and stabilizing the candidate center
during high-dimensional search.

In this experiment, we set
\[
    K_{\mathrm{elite}}=8,
    \qquad
    \lambda_H=1,
    \qquad
    \lambda_{\mathrm{fit}}=5,
\]
and use a smoothing temperature of \(0.18\) in the corrected-Hamiltonian
candidate objective. The time-indexed control and \(q\)-networks have two
hidden layers of width \(64\), while \(p_0\) and \(P_0\) are trainable
parameters. The training batch size is \(64\).

\begin{remark}
The structured perturbations and the elite fitting term are used in the
20-dimensional stress test to improve the stability of high-dimensional
candidate search. The underlying policy improvement principle remains the
candidate-based smooth minimization of the second-order corrected Hamiltonian.
\end{remark}

\begin{figure}[H]
    \centering

    \begin{minipage}{0.98\textwidth}
        \centering

        \captionsetup[subfigure]{
            font=footnotesize,
            skip=2pt
        }

        \begin{subfigure}[t]{0.48\linewidth}
            \centering
            \includegraphics[width=0.98\linewidth]
            {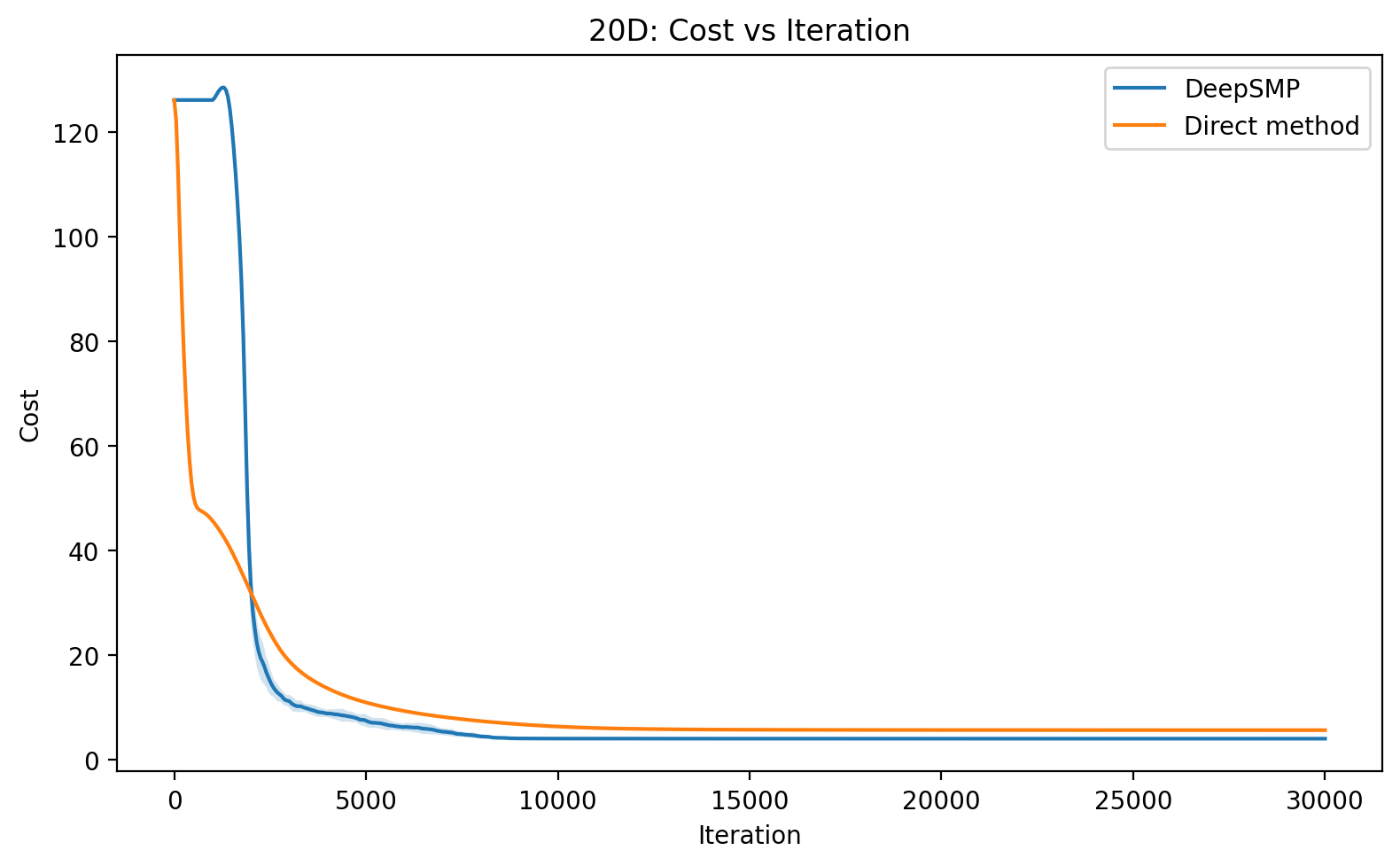}
            \caption{Cost.}
            \label{fig:comparison_20d_cost}
        \end{subfigure}
        \hfill
        \begin{subfigure}[t]{0.48\linewidth}
            \centering
            \includegraphics[width=0.98\linewidth]
            {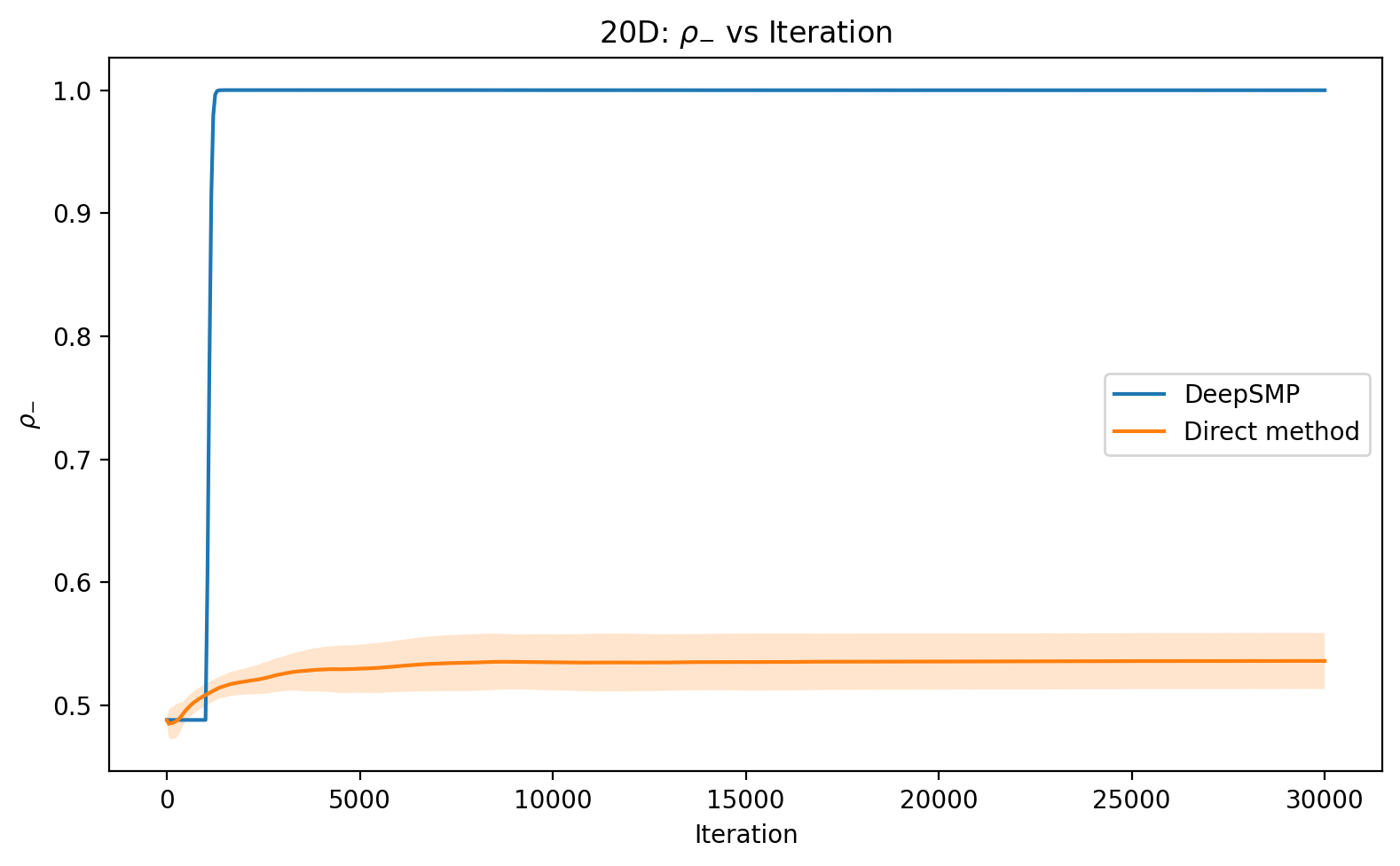}
            \caption{Negative-region proportion $\rho_{-}$.}
            \label{fig:comparison_20d_rho}
        \end{subfigure}

        \vspace{-0.05cm}

        \begin{subfigure}[t]{0.48\linewidth}
            \centering
            \includegraphics[width=0.98\linewidth]
            {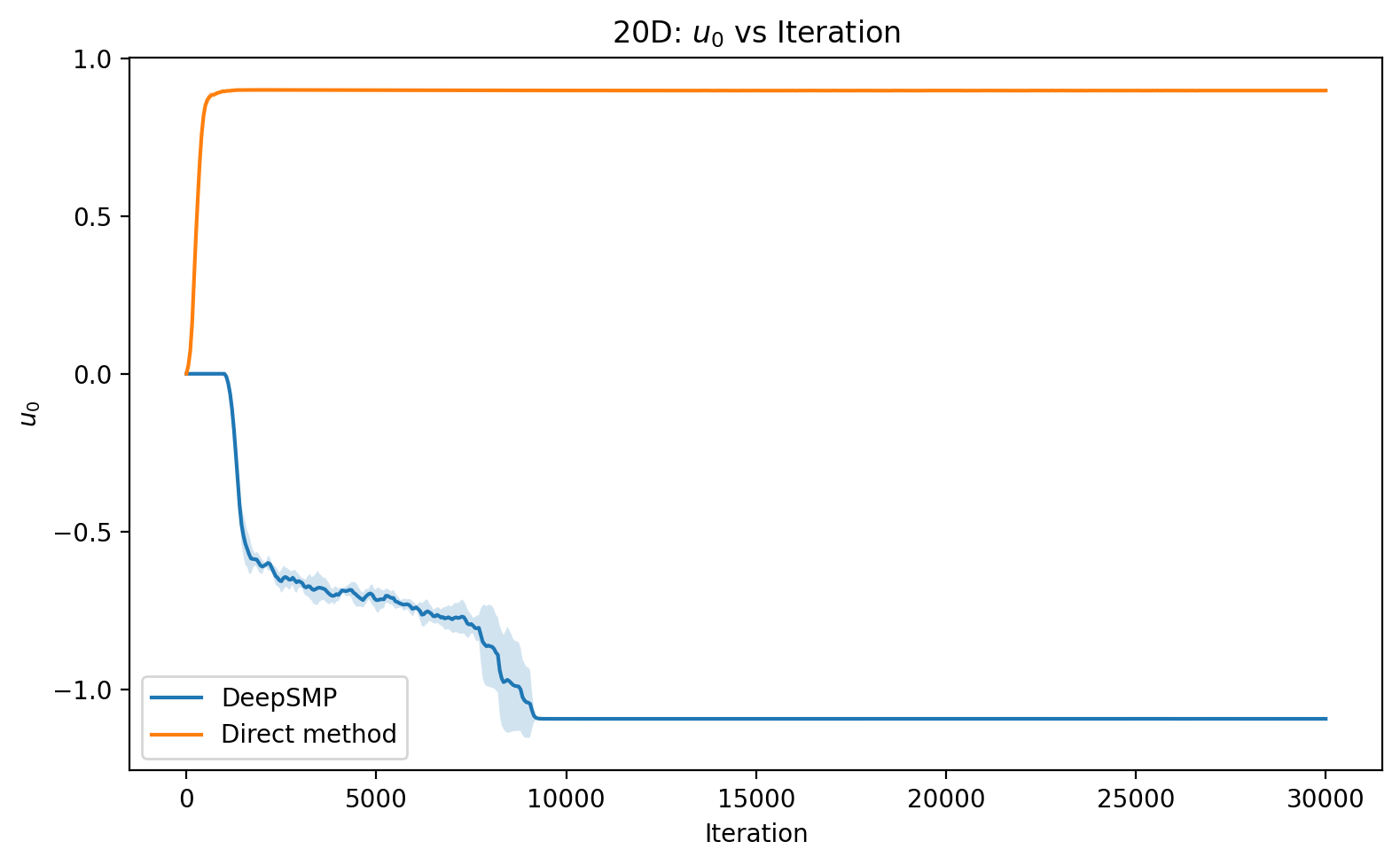}
            \caption{Initial control $u_0$.}
            \label{fig:comparison_20d_u0}
        \end{subfigure}
        \hfill
        \begin{subfigure}[t]{0.48\linewidth}
            \centering
            \includegraphics[width=0.98\linewidth]
            {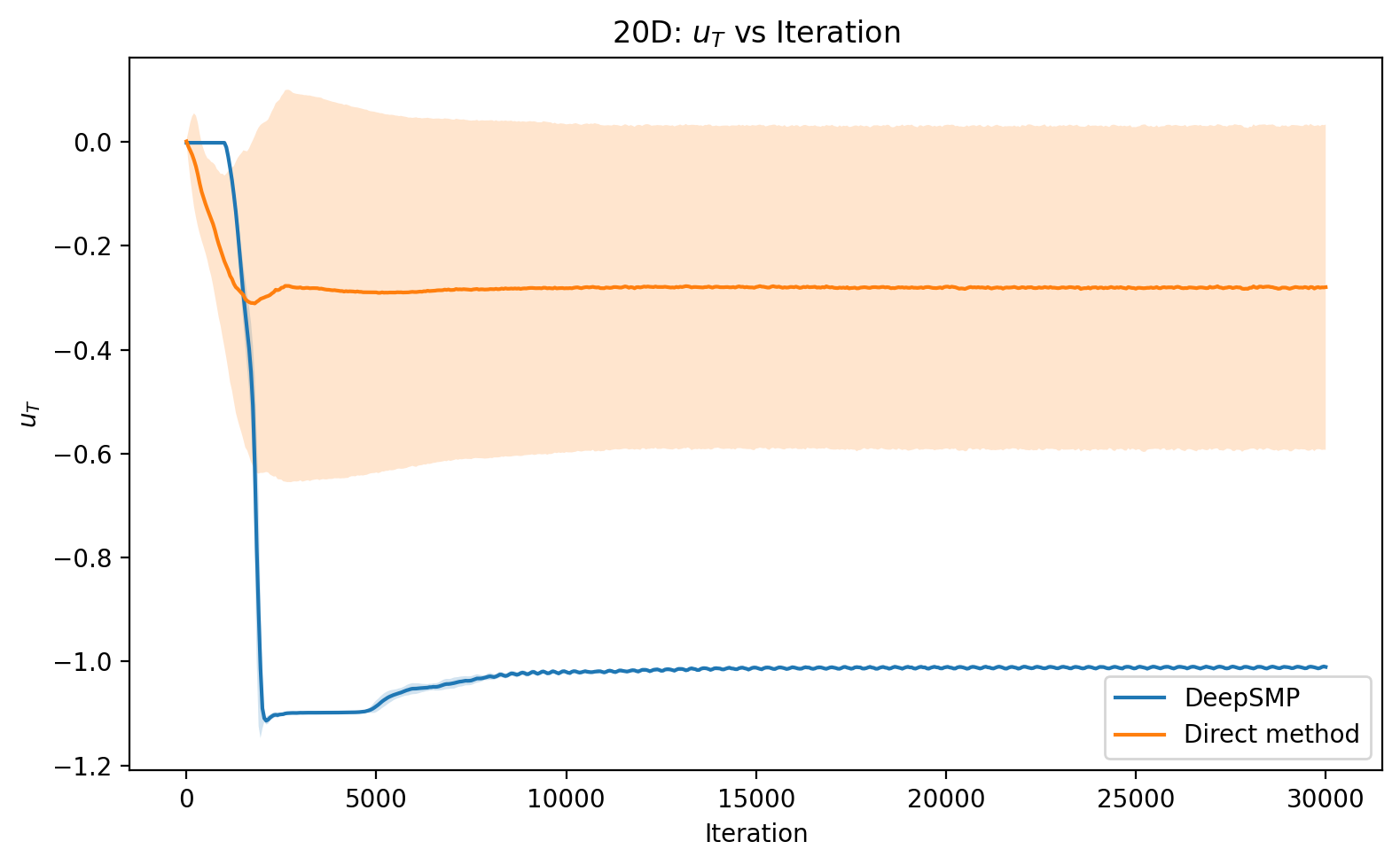}
            \caption{Terminal control $u_T$.}
            \label{fig:comparison_20d_uT}
        \end{subfigure}

    \end{minipage}

    \captionsetup{
        font=small,
        width=0.88\textwidth,
        skip=4pt
    }
    \caption{
        Optimization trajectories of DeepSMP and the direct method in the
        twenty-dimensional setting. The solid curves represent the means
        over five random seeds, while the shaded regions represent one
        standard deviation. 
    }
    \label{fig:comparison_20d_curves}
\end{figure}

\begin{table}[htbp]
    \centering
    \caption{
        Performance comparison of DeepSMP and the direct method under five
        random seeds in the twenty-dimensional setting. 
    }
    \label{tab:comparison_20d_results}

    \setlength{\tabcolsep}{4pt}
    \renewcommand{\arraystretch}{1.15}

    \resizebox{\textwidth}{!}{%
    \begin{tabular}{l c cc cc cc cc}
        \toprule
        \multirow{2}{*}{Method}
        & \multirow{2}{*}{Iter.}
        & \multicolumn{2}{c}{Cost}
        & \multicolumn{2}{c}{$u_0$}
        & \multicolumn{2}{c}{$u_T$}
        & \multicolumn{2}{c}{$\rho_{-}$} \\
        \cmidrule(lr){3-4}
        \cmidrule(lr){5-6}
        \cmidrule(lr){7-8}
        \cmidrule(lr){9-10}
        &
        & Mean & Std.
        & Mean & Std.
        & Mean & Std.
        & Mean & Std. \\
        \midrule
        DeepSMP (no BN)
        & $3\times10^{4}$
        & $4.0138$
        & $0.0073$
        & $-1.0935$
        & $8.5691\times10^{-5}$
        & $-1.0105$
        & $0.0027$
        & $0.9999$
        & $1.8103\times10^{-5}$ \\
        Direct method (BN)
        & $10^{5}$
        & $5.6233$
        & $0.0696$
        & $0.8976$
        & $1.6497\times10^{-4}$
        & $-0.2833$
        & $0.3093$
        & $0.5375$
        & $0.0230$ \\
        \bottomrule
    \end{tabular}%
    }
\end{table}

Figure~\ref{fig:comparison_20d_curves} and
Table~\ref{tab:comparison_20d_results} show that the performance gap between
the two methods persists, and becomes more pronounced, in the
twenty-dimensional coupled setting. After the policy update becomes active,
DeepSMP rapidly enters a low-cost regime and stabilizes at a mean cost of
$4.0138$ with a cross-seed standard deviation of $0.0073$ after
$3\times10^{4}$ iterations. Its dimension-averaged controls converge to
$u_0=-1.0935$ and $u_T=-1.0105$, while $\rho_{-}=0.9999$, indicating an
almost uniform selection of the negative low-diffusion action region across
time, validation paths, and control coordinates. By contrast, the direct
method converges more slowly and, even after $10^{5}$ iterations, attains a
higher mean cost of $5.6233$ with a standard deviation of $0.0696$. Its
initial control concentrates near the positive mode, with $u_0=0.8976$,
whereas $u_T=-0.2833$, $\rho_{-}=0.5375$, and the large terminal-control
standard deviation of $0.3093$ reveal a heterogeneous mixed-action
configuration. Thus, DeepSMP achieves lower cost, markedly stronger
cross-seed stability, and a substantially more coherent control selection
under a considerably smaller optimization budget.

\subsection{Comparison with the modified MSA}
\label{subsec:comparison-modified-msa}
We next compare DeepSMP with a modified-MSA-type neural baseline. The
purpose of this experiment is to isolate the directional effect of the
second-order adjoint correction in a problem with control-dependent
diffusion. In particular, the example is constructed so that two diffusion
perturbations have the same Frobenius norm but act along state directions
with substantially different terminal penalties. It therefore distinguishes
the scalar isotropic augmentation employed by modified MSA from the
matrix-weighted correction used by DeepSMP.

Let \(T=1\), \(X_0=0\), and let
\[
    W_t=(W_t^1,W_t^2)^\top
\]
be a two-dimensional standard Brownian motion. The scalar control takes
values in
\[
    U=[-2,2].
\]
For any admissible control \(u\), the state process satisfies
\[
    dX_t^u
    =
    AX_t^u\,dt+\Sigma(u_t)\,dW_t,
    \qquad
    X_0^u=0,
\]
where
\[
    A=
    \begin{pmatrix}
        -\lambda & \varepsilon\\
        0 & -\lambda
    \end{pmatrix},
    \qquad
    \lambda=0.20,
    \qquad
    \varepsilon=0.05.
\]
Define the smooth actuator gates
\[
    \phi_+(u)=\exp\{-\kappa(u-1)^2\},
    \qquad
    \phi_-(u)=\exp\{-\kappa(u+1)^2\},
    \qquad
    \kappa=20.
\]
The controlled diffusion matrix is
\[
    \Sigma(u)
    =
    \sigma_0 I_2
    +
    a\,\phi_+(u)E_{12}
    +
    a\,\phi_-(u)E_{21},
    \qquad
    \sigma_0=0.01,
    \qquad
    a=0.34,
\]
where \(E_{ij}\) denotes the matrix whose \((i,j)\)-entry is one and whose
remaining entries are zero. Thus, the action region near \(u=1\) routes
Brownian noise primarily into the first state component, whereas the region
near \(u=-1\) routes it primarily into the second state component.

The running and terminal costs are
\[
    f(u)
    =
    R\bigl(1-\phi_+(u)-\phi_-(u)\bigr)
    +
    \alpha u^2,
    \qquad
    R=0.02,
    \qquad
    \alpha=5\times10^{-4},
\]
and
\[
    \Phi(x)
    =
    \frac12x^\top Gx,
    \qquad
    G=
    \begin{pmatrix}
        q_1 & \gamma\\
        \gamma & q_2
    \end{pmatrix},
    \qquad
    q_1=10,
    \qquad
    q_2=0.1,
    \qquad
    \gamma=0.03.
\]
The objective functional is therefore
\[
    J(u)
    =
    \mathbb E
    \left[
        \Phi(X_T^u)
        +
        \int_0^T f(u_t)\,dt
    \right].
\]

For completeness, the augmented Hamiltonian of modified MSA contains the
three scalar-weighted terms
\[
\begin{aligned}
    \widetilde{\mathcal H}_{\rho}
    (t,x,p,q,\bar u,a)
    ={}&
    H(t,x,p,q,a)\\
    &+
    \frac{\rho}{2}
    \left|
        b(t,x,a)-b(t,x,\bar u)
    \right|^2\\
    &+
    \frac{\rho}{2}
    \left\|
        \Sigma(t,x,a)-\Sigma(t,x,\bar u)
    \right\|_F^2\\
    &+
    \frac{\rho}{2}
    \left|
        H_x(t,x,p,q,a)
        -
        H_x(t,x,p,q,\bar u)
    \right|^2.
\end{aligned}
\]
In the present example, \(b(x,a)=Ax\) is independent of the control.
Moreover, since \(\Sigma\) and \(f\) are independent of the state variable,
\[
    H_x(t,x,p,q,a)=A^\top p
\]
is also independent of \(a\). Hence the first and third augmentation terms
vanish identically, and the modified-MSA update reduces exactly to
\[
    \widetilde H_{\rho}
    (t,x,p,q,\bar u,a)
    =
    H(t,x,p,q,a)
    +
    \frac{\rho}{2}
    \left\|
        \Sigma(a)-\Sigma(\bar u)
    \right\|_F^2.
\]
Thus, the baseline used below retains the original modified-MSA
augmentation; its reduction to a single diffusion-difference term is a
consequence of the coefficient structure of this example rather than an
additional simplification of the algorithm.

The discriminating feature of the example is the direction of the
diffusion perturbation rather than its magnitude alone. Taking a control
near zero as the reference, the two active action regions generate, up to
exponentially small gate terms,
\[
    \Delta\Sigma_+
    :=
    \Sigma(1)-\Sigma(0)
    \simeq
    aE_{12},
    \qquad
    \Delta\Sigma_-
    :=
    \Sigma(-1)-\Sigma(0)
    \simeq
    aE_{21}.
\]
These perturbations have the same Frobenius norm:
\[
    \|\Delta\Sigma_+\|_F
    =
    \|\Delta\Sigma_-\|_F
    \simeq a.
\]
Consequently, the modified-MSA augmentation assigns approximately the same
penalty to the two directions:
\[
    \frac{\rho}{2}
    \|\Delta\Sigma_+\|_F^2
    \simeq
    \frac{\rho}{2}
    \|\Delta\Sigma_-\|_F^2
    \simeq
    \frac{\rho a^2}{2}.
\]
The scalar coefficient \(\rho\) controls the overall size of the
augmentation but does not distinguish between equal-norm directions in the
space of diffusion matrices.

DeepSMP instead evaluates a candidate action through the second-order
corrected Hamiltonian
\[
    \mathcal H^{\bar u}
    (t,x,p,q,P;a)
    =
    H(t,x,p,q,a)
    +
    \frac12
    \operatorname{Tr}
    \left[
        \Delta\Sigma(a;\bar u)^\top
        P
        \Delta\Sigma(a;\bar u)
    \right],
\]
where
\[
    \Delta\Sigma(a;\bar u)
    =
    \Sigma(a)-\Sigma(\bar u).
\]
For the two active diffusion directions,
\[
\begin{aligned}
    \frac12
    \operatorname{Tr}
    \left[
        \Delta\Sigma_+^\top
        P_t
        \Delta\Sigma_+
    \right]
    &\simeq
    \frac12a^2P_{11}(t),\\
    \frac12
    \operatorname{Tr}
    \left[
        \Delta\Sigma_-^\top
        P_t
        \Delta\Sigma_-
    \right]
    &\simeq
    \frac12a^2P_{22}(t).
\end{aligned}
\]
In this example the second-order adjoint is deterministic and satisfies
\[
    -\dot P_t
    =
    A^\top P_t+P_tA,
    \qquad
    P_T=G,
\]
so that
\[
    P_t
    =
    e^{A^\top(T-t)}
    G
    e^{A(T-t)}.
\]
Since \(q_1\gg q_2\), the strong directional asymmetry of the terminal
matrix \(G\) is retained by \(P_t\), and in particular
\(P_{11}(t)\gg P_{22}(t)\) throughout the time interval. The two
equal-Frobenius-norm diffusion perturbations are therefore assigned
different values by the \(P_t\)-weighted quadratic form.

The example thus compares two distinct update geometries. The modified-MSA
augmentation applies a scalar isotropic penalty in the space of diffusion
perturbations, whereas the DeepSMP correction applies a direction-dependent
matrix weighting determined by the second-order adjoint. When the
uncorrected Hamiltonian only weakly separates the two active action regions,
the scalar augmentation provides no additional discrimination between their
equal-norm diffusion perturbations. The \(P_t\)-weighted correction, by
contrast, supplies an explicit mechanism for distinguishing the two
noise-routing directions according to their different effects on the
state-dependent objective.

\begin{table}[H]
\centering
\caption{Comparison with the modified-MSA-type baseline at \(3\times10^4\) iterations.}
\label{tab:comparison-modified-msa}
\scriptsize
\setlength{\tabcolsep}{3pt}
\begin{tabular}{lcccc}
\toprule
Method & Target cost & \(u_0\) & \(u_T\) & \(\rho_-\) \\
\midrule
Modified MSA-type baseline
& \(0.020420934944\pm1.84{\times}10^{-10}\)
& \(-3.49{\times}10^{-4}\)
& \(9.29{\times}10^{-5}\)
& \(0.519177\pm0.126021\) \\
DeepSMP
& \(0.006980263\pm8.06{\times}10^{-7}\)
& \(-1.0183\)
& \(-1.0225\)
& \(1.000000\pm0.000000\) \\
\bottomrule
\end{tabular}

\vspace{0.5em}
\begin{minipage}{0.95\textwidth}
\footnotesize
All quantities are evaluated using the target policy. The reported standard
deviations are cross-seed sample standard deviations over five independent
runs. The quantity \(\rho_-\) denotes the occupation rate of the negative
noise-routing region \(\{u<0\}\).
\end{minipage}
\end{table}

The numerical results are consistent with this mechanism. The modified-MSA-type
baseline is extremely stable across seeds, but it stabilizes at a near-zero
control: \(u_0\approx -3.49\times10^{-4}\) and
\(u_T\approx 9.29\times10^{-5}\). Its target cost is almost identical across the five seeds, which shows that the iteration has reached a stable plateau. However, this plateau corresponds to the inactive region rather than to the lower-cost noise-routing direction.

DeepSMP behaves differently on the same problem. It learns a target policy with
\(u_0\approx-1.0183\), \(u_T\approx-1.0225\), and \(\rho_-=1\) for all five
seeds. Its target cost is about \(6.98\times10^{-3}\), substantially smaller
than the modified-MSA-type baseline. This supports the interpretation that the
\(P\)-weighted correction injects curvature information into the Hamiltonian
update and enables a directional discrimination that is not available to an
isotropic proximal correction.

\appendix
\section{The twenty-dimensional control problem}
\label{app:20d-setting}

For the experiment in Section~\ref{subsubsec:20d-stress-test}, we set
\[
    n=k=d=20,\qquad T=1,\qquad
    U=[-2,2]^{20},\qquad x_0=0\in\mathbb R^{20}.
\]
For an admissible control \(u\in\mathcal U\), the state equation and
objective are
\[
\begin{aligned}
    dX_t^u
    &=b(t,X_t^u,u_t)\,dt+\sigma(t,X_t^u,u_t)\,dW_t,
    \qquad X_0^u=0,\\
    J(u)
    &=\mathbb E\left[\Phi(X_1^u)
        +\int_0^1 f(t,X_t^u,u_t)\,dt\right],
    \qquad \inf_{u\in\mathcal U}J(u).
\end{aligned}
\]
All coefficients below are independent of time. The time discretization
uses \(N=40\) uniform intervals, with \(\Delta t=1/40\).

\paragraph{Cyclic coupling matrices.}
Let \(I_{20}\) denote the identity matrix, \(e_i\) the \(i\)-th
standard basis vector of \(\mathbb R^{20}\), and
\(\mathbf 1_{20}=(1,\ldots,1)^\top\).
For \(\ell=1,2\), define \(C_\ell\in\mathbb R^{20\times20}\) by
\[
    (C_\ell)_{ij}
    =\begin{cases}
        1,&j\equiv i+\ell\ \text{or}\ i-\ell\pmod{20},\\
        0,&\text{otherwise},
    \end{cases}
    \qquad i,j=1,\ldots,20.
\]
The state and control-to-drift matrices are
\[
\begin{aligned}
    A&=-0.72I_{20}+0.035C_1+0.012C_2,\\
    B_+&=\operatorname{Diag}\left(
        1.25+0.030\sin\frac{2\pi(i-1)}{20}
        \right)_{i=1}^{20}
        +0.035C_1+0.010C_2,\\
    B_-&=\operatorname{Diag}\left(
        1.30+0.025\sin\frac{2\pi(i-1)}{20}
        \right)_{i=1}^{20}
        +0.030C_1+0.008C_2,\\
    b_0&=-0.05\mathbf 1_{20}.
\end{aligned}
\]

\paragraph{Drift and diffusion.}
For a scalar \(z\in[-2,2]\), define
\[
\begin{aligned}
    \psi_+(z)&=\exp\{-5(z-0.9)^2\},
    &\psi_-(z)&=\exp\{-6(z+1.1)^2\},\\
    \chi_+(z)&=\exp\{-10(z-0.9)^2\},
    &\chi_-(z)&=\exp\{-6(z+1.1)^2\}.
\end{aligned}
\]
For \(a=(a_1,\ldots,a_{20})^\top\in U\), these functions act
componentwise, so that
\[
    \psi_\pm(a)=(\psi_\pm(a_i))_{i=1}^{20},
    \qquad
    \chi_\pm(a)=(\chi_\pm(a_i))_{i=1}^{20}.
\]
The drift is
\[
    b(t,x,a)=Ax+B_+\psi_+(a)+B_-\psi_-(a)+b_0.
\]
The diffusion matrices are
\[
    S_+=0.010C_1+0.0030C_2,
    \qquad S_-=0.0020C_1+0.0005C_2,
\]
and
\[
\begin{aligned}
    D(a)&=\operatorname{Diag}\left(
        \max\{0.08,\,0.20+0.12\chi_+(a_i)-0.08\chi_-(a_i)\}
        \right)_{i=1}^{20},\\
    \sigma(t,x,a)
    &=D(a)+\operatorname{Diag}(\chi_+(a))S_+
        +\operatorname{Diag}(\chi_-(a))S_-.
\end{aligned}
\]
The lower cutoff in \(D\) is inactive, since
\(0.20+0.12\chi_+(z)-0.08\chi_-(z)\geq0.12\).
In particular, \(\sigma\) is independent of the state.

\paragraph{Running and terminal costs.}
The target vector is \(x_{\rm tar}=((x_{\rm tar})_i)_{i=1}^{20}\), where
\[
    (x_{\rm tar})_i
    =1+0.15\sin\frac{2\pi(i-1)}{20}
       +0.07\cos\frac{4\pi(i-1)}{20}.
\]
The symmetric positive definite weight matrices are
\[
\begin{aligned}
    M_{\rm run}
    &=I_{20}+0.025C_1+0.006C_2
      -0.0125(e_1e_{20}^\top+e_{20}e_1^\top),\\
    M_{\rm term}
    &=I_{20}+0.040C_1+0.010C_2
      -0.020(e_1e_{20}^\top+e_{20}e_1^\top).
\end{aligned}
\]
Thus, the coupling between coordinates \(1\) and \(20\) in each cost
matrix is half the other nearest-neighbor coupling coefficients.
The costs are
\[
\begin{aligned}
    f(t,x,a)
    &=0.40(x-x_{\rm tar})^\top M_{\rm run}(x-x_{\rm tar})
      +0.010|a|^2,\\
    \Phi(x)
    &=5(x-x_{\rm tar})^\top M_{\rm term}(x-x_{\rm tar}).
\end{aligned}
\]

\paragraph{Second-order adjoint simplification.}
Under the minimization convention used in this paper,
\(b_x=A\), \(b_{xx}=0\), \(\sigma_x=0\),
\(f_{xx}=0.80M_{\rm run}\), and \(\Phi_{xx}=10M_{\rm term}\).
Consequently, for every reference control \(u\), the second-order adjoint
is the deterministic solution of
\[
    \frac{dP_t^u}{dt}
    =-\bigl(A^\top P_t^u+P_t^uA+0.80M_{\rm run}\bigr),
    \qquad P_1^u=10M_{\rm term},
\]
and \(Q_t^{u,i}\equiv0\) for \(i=1,\ldots,20\).

\renewcommand{\refname}{References}
\bibliographystyle{plain}
\bibliography{references}

\end{document}

%% file: preprint-formatting.tex
\usepackage{amsfonts,amssymb,amsthm,mathtools,mathrsfs,bm}
\usepackage[a4paper,margin=1in]{geometry}
\usepackage{float,indentfirst,anyfontsize}
\usepackage{booktabs,multirow,multicol}
\usepackage[shortlabels,inline]{enumitem}
\usepackage{appendix}
\usepackage[dvipsnames]{xcolor}
\usepackage{graphicx,subcaption,listings}
\graphicspath{{./figure/}{./figures/}{./image/}{./images/}{./graphic/}{./graphics/}{./picture/}{./pictures/}}
\usepackage[hidelinks]{hyperref}
\theoremstyle{plain}
\newtheorem{theorem}{Theorem}[section]
\newtheorem{lemma}{Lemma}[section]
\newtheorem{proposition}{Proposition}[section]
\newtheorem{corollary}[theorem]{Corollary}
\theoremstyle{definition}
\newtheorem{assumption}{Assumption}[section]
\theoremstyle{remark}
\newtheorem{remark}{Remark}[section]